\documentclass[11pt,letterpaper]{amsart}
\usepackage[utf8]{inputenc}
\usepackage{amsmath,amssymb,graphicx,setspace,verbatim,url}
\usepackage{color}
\usepackage{hyperref}
\usepackage{multicol}
\usepackage{enumitem}
\usepackage{float}
\usepackage{soul}
\usepackage{tikz}
\usepackage[normalem]{ulem}

\usepackage{tikz-cd}
\usetikzlibrary{arrows.meta}
\newenvironment{red}{\relax\color{red}}{\hspace*{.5ex}\relax}
\newcommand{\bcp}{\begin{red} }
\newcommand{\ecp}{\end{red}}

\usepackage{comment}

\tikzcdset{
  arrow style={tikz,diagrams={>=Triangle}}
}

\input xy
\xyoption{all}

\newtheorem{prop}{Proposition}[section]
\newtheorem{coro}[prop]{Corollary}
\newtheorem{thm}[prop]{Theorem}
\newtheorem{lemma}[prop]{Lemma}

\newtheorem{definition}[prop]{Definition}

\newtheorem{example}[prop]{Example}

\DeclareMathOperator{\Cor}{Cor}

\newcommand{\GG}{\mathcal{G}}
\newcommand{\TT}{\mathcal{T}}

\newcommand{\KK}{\mathcal{K}}

\newcommand{\II}{\mathcal{I}}

\newcommand{\Aut}{\mathrm{Aut}}

\newcommand{\ZZ}{\mathbb{Z}}

\begin{document}
\title[From Group Operations to Geometric Structures]{From Group Operations to Geometric Structures: Amalgamations, HNN-Extensions, and Twisting in Coset Geometries}

\author[Claudio Alexandre Piedade]{Claudio Alexandre Piedade}
\address{Claudio Alexandre Piedade, Centro de Matemática da Universidade do Porto, Universidade do Porto, Portugal, Orcid number 0000-0002-0746-5893
}
\email{claudio.piedade@fc.up.pt}

\author[Philippe Tranchida]{Philippe Tranchida}
\address{Philippe Tranchida, Max Planck Institute for Mathematics in the Sciences, Inselstrasse 22 D-04107 Leipzig, Orcid number 0000-0003-0744-4934.
}
\email{tranchida.philippe@gmail.com}
{}

\thanks{\textit{Declaration of interests:} none}
\date{\today}

\begin{abstract}
    Coset incidence geometries, introduced by Jacques Tits, provide a versatile framework for studying the interplay between group theory and geometry. 
    In this article, we build upon that idea by extending classical group-theoretic constructions (amalgamated products, HNN-extensions, semi-direct products, and twisting) to the setting of coset geometries. This gives a general way to glue together incidence geometries in various ways. This provides a general framework for combining or gluing incidence geometries in different ways while preserving essential properties such as flag-transitivity and residual connectedness.
    Using these techniques, we analyze families of Shephard groups, which generalize both Coxeter and Artin-Tits groups, and their associated simplicial complexes.
    Our results also point to the existence of a Bass-Serre theory for coset geometries and of a fundamental geometry of a graph of coset geometries.
\end{abstract}

\maketitle


\section{Introduction}

The fruitful interplay between geometry and group theory has shaped much of modern mathematics. A central theme, originating in 1892 with Felix Klein’s Erlangen program, is the study of geometric structures through the lens of their symmetry groups. Klein’s insight sparked countless investigations into how groups and geometries interact, building the foundations for exploring the relationships between group theory, geometry, and combinatorics.

A major advance in this direction came from Jacques Tits ~\cite{Tits1957,Tits1963geometries}, who introduced the notion of coset incidence system in the 1950s.
Given a group $G$ and a family $(G_i)_{i \in I}$ of subgroups, Tits’ construction builds an incidence structure, nowadays known as a coset incidence system, on which $G$ acts naturally. This framework provides a unifying language for studying a wide variety of combinatorial and geometric objects, such as Coxeter complexes, vector and projective spaces, and abstract polytopes. Later, Tits’ development of the theory of buildings~\cite{Handbook,Tits1974} -- highly symmetric simplicial complexes assembled from Coxeter complexes -- provided powerful tools for the study of algebraic groups and their actions. The theory of buildings has had far-reaching applications, from the cohomology of arithmetic groups~\cite{KenBrownBook2,Handbook} to the structure of $p$-adic groups~\cite{IM}, and J. Tits work was later recognized with the Abel Prize in 2008.

Coset incidence systems are particularly significant because they encode both the algebraic properties of the group and the combinatorial structure of the geometry. For example, the coset incidence system constructed from a Coxeter group $G$ and its standard parabolic subgroups recovers the Coxeter complex of $G$, a fundamental object in the theory of reflection groups and buildings~\cite{KenBrownBook2,Tits1974}. 
Similarly, for Artin-Tits groups, the associated coset geometry gives rise to the Artin-Tits complex~\cite{ParisGodelle2012}, which plays a key role in topological and combinatorial investigations of these groups. In particular, simplicial complexes, such as the (modified) Deligne complex~\cite{Charney1995}, the Salvetti complex~\cite{Salvetti_1987}, and the Artin-Tits complex~\cite{ParisGodelle2012}, have been widely used to study the $K(\pi,1)$ conjecture of Artin groups~\cite{paris2014k}. 
One of the main reasons these complexes have been used so successfully is that their associated coset incidence systems are flag-transitive and residually connected, showing that the properties of the incidence structures associated to Coxeter and Artin-Tits groups have far reaching consequences.
There has recently been a renewed interest in the study of coset geometries, particularly after the introduction of incidence structures that generalize polytopes, known as hypertopes \cite{hypertopes, zhang2024abelian, fernandes2018hexagonal,halving,ens2018rank, leemans2025flag,piedade2023infinite}.

This paper set out with two main objectives. The first one is to extend famous and ubiquitous group-theoretic constructions to coset incidence systems, providing a framework for ``gluing" incidence systems together while both respecting the group operation and
preserving many of the properties of the original coset incidence systems. The second objective is to strengthen the connections between coset incidence systems, especially coset geometries, and other areas of mathematics, such as geometric group theory and combinatorics. The authors started investigating the impact of algebraic constructions on incidence geometries in~\cite{halving}, where the halving operation is studied under a group-theoretical lens. The group-theoretical constructions investigated in the present paper are free amalgamated products, HNN-extensions and semi-direct products. At first glance, it seems rather surprising that the effects of such operations on incidence geometries has not already been extensively studied. 
We believe that the lack of research in this direction is explained by the fact that these group operations, when applied to incidence geometries, almost never preserve the linearity of the Buekenhout diagrams~\cite{buekenhout2013diagram} of these geometries. While this is not a problem in itself (in fact, it can be seen as an advantage), historically, the spotlight has mainly been on geometries with linear diagrams, such as regular polytopes or projective spaces. Therefore, group operations on coset geometries have been considered only occasionally, and always in very restricted settings that ensured the preservation of the linearity of diagrams. This can be seen quite clearly in the study of abstract polytopes (see \cite{ARP}), where amalgamations, twisting and other group-theoretical constructions sometimes play a role, but are heavily restricted to some very specific cases for which the resulting objects are still polytopes. In our opinion, this focus on linear diagrams heavily hindered the study of these operations in their most natural form.

The first part of the present paper focuses on amalgamated products and HNN-extensions, two fundamental operations in combinatorial and geometric group theory, introduced in the seminal papers \cite{neumann1948generalized}, \cite{higman1949embedding}. These operations allow the creation of complex groups starting from simpler components, while preserving some structural properties. The amalgamated product of two groups generalizes free products by identifying isomorphic subgroups, while HNN-extensions build over-groups in which two isomorphic subgroups become conjugated. Both operations are deeply connected with Bass-Serre theory, which describes groups acting on trees through their decompositions into a graph of groups~\cite{serre2002trees}. In the context of coset geometries, we will show that these group operations acquire geometric significance, enabling the creation of new geometries from existing ones, while preserving properties like flag-transitivity and residual connectedness. Amalgamations and HNN-extensions have also been explored in other geometric settings (see \cite{bestvinacombination}, \cite{AnosovAmalgam}, for example), showing the wide range of application of these concepts.

The second part of the paper investigates how semi-direct products and twisting constructions can be extended to coset incidence geometries. The semi-direct product has a nice and natural extension into coset geometries. However, this natural construction does not preserve the thinness of the starting geometries. That is where the twisting construction comes into play. The twisting operation we define for semi-direct products is inspired by the classical polytope operation of the same name, which itself borrows his name from the notion of twisted simple groups~\cite[Chapter 8]{ARP}. This operation classically comes with many conditions whose goal is to preserve linear diagram. Our generalized construction is obtained by keeping only the necessary conditions for preserving thinness.

The last part of the paper is dedicated to examples and applications. Indeed, the multiple constructions we present in this paper do not only generalize classical results in a theoretical sense, but are also easily and directly applicable. We already mentioned the various simplicial complexes associated to Coxeter and Artin-Tits groups. In general, it is possible to build such simplicial complexes from arbitrary coset incidence systems.
This construction can be formalized in a general setting~\cite{Handbook}, obtaining what we will call a coset complex.
Shephard groups~\cite{Kapovich1998,Goldman2024} are a generalization of both Coxeter and Artin-Tits groups, where the generators are allowed to be of arbitrary orders (see ~\cite{Kapovich1998}, for example).
Shephard groups can all be described by a graph, usually called the extended Coxeter diagram.
They have many applications, reaching beyond their two most prominent families (the Coxeter and Artin–Tits groups): these groups include the complex reflection groups, the automorphisms of regular complex polytopes~\cite{ShephardTodd1954}, the family $B_n(2,\infty)$, which led to the classification of Mikado braids~\cite{Baumeister2017}, and family $B_n(\infty,2)$, which was crucial to proving that Artin-Tits groups of affine type are in fact $K(\pi,1)$~\cite{McCammond2017,Paolini2020}.
In this article, we show that many Shephard groups, and their respective simplicial complexes, are as well-behaved as their Coxeter and Artin-Tits counterparts, particularly when these groups are built by free product with amalgamation, HNN-extension or twisting of Coxeter and Artin-Tits groups. Lastly, we propose a definition of a graph of coset incidence systems, inspired by the notion of graph of groups in Bass-Serre theory. We leave the precise study of this notion to the future, but we work out the details of the construction on an example, showing that different choices of trees yield the same fundamental coset geometry, in that case at least.

\subsection{Description of the results}
Here, we summarize the main contributions of this paper. The more precise statements can be found in the appropriate sections of the paper.

Let $\alpha = (A,(A_i)_{i \in I_\alpha})$ and $\beta = (B,(B_i)_{i \in I_\beta})$ be two coset incidence systems. We define a coset incidence system $\Gamma =\alpha * \beta$ for the free product $A * B$. Moreover, suppose that $\alpha$ and $\beta$ are compatible (see Section \ref{sec:Amalgams+HNN}) with respect to an isomorphism $\varphi$ between a parabolic subgroup of $\alpha$ and a parabolic subgroup of $\beta$. Then, we construct a coset incidence system $\Gamma = \alpha *_\varphi \beta$ for the free amalgamated product $A *_\varphi B$. If instead, $\varphi$ is an admissible isomorphism between two parabolic subgroups of $\alpha$, we show how to build a coset incidence system $\Gamma = \alpha *_\varphi$ for the HNN-extension $A *_\varphi$. Finally, if $\varphi \colon B \to \Aut(A)$ is an admissible action of $B$ on $A$, we construct two coset incidence systems $\Gamma = \alpha \rtimes_\varphi \beta$ and $\Gamma = T(\alpha,\beta)$ for the semi-direct product $A \rtimes_\varphi B$. The latter construction is called a Twisting. 
In each of the above constructions, we show that $\Gamma$ inherits both flag-transitivity and residual connectedness from $\alpha$ and $\beta$. Other specific properties, such as thinness and finiteness, are investigated in some of the cases.

We then use these constructions to build many new examples of flag-transitive and residually connected geometries. These examples include geometries associated to Shephard groups that can be built from Coxeter and Artin-Tits groups by using amalgamations and twistings. In particular, this shows that the coset complex associated to these Shephard groups is very well behaved (see Corollary \ref{coro:FreeProductArtinCoxeter} for an example). 
Finally, we introduce the concept of a graph of coset incidence systems, in order to introduce the idea of a Bass-Serre theory on coset incidence systems.

\subsection{Organization of the paper}

Section \ref{sec:prelims} contains all the necessary background knowledge for the rest of the paper.
Section \ref{sec:Amalgams+HNN} defines coset incidence systems for free products (with amalgamation) and HNN-extensions and proves that these constructions preserves flag-transitivity and residual connectedness. Both constructions are done with the objective of introducing a Bass-Serre theory for coset incidence systems, as discussed in Section~\ref{subsec:graphofcosetgeom}, by defining a notion of graph of coset incidence systems.

Section \ref{sec:SplitExt} investigates coset incidence systems built from semi-direct products. We first give a very general and natural construction and then a second, more refined version, that is called twisting, that allow us to preserve thinness.

Section \ref{sec:applications} contains various applications of the concept developed in the article. In particular, in Section \ref{sec:subsec:ProductArtinCoxeter}, we begin by performing free products of Coxeter/Artin-Tits groups to generate some Shephard groups, using our results to show intersection properties of standard parabolic subgroups of these Shephard groups. In Section \ref{sec:twist_shephard_groups}, we go one step further by applying our twisting operations on Shephard groups and some of their quotients.
In Section \ref{sec:TwistArtin}, we amalgamate and twist some Artin-Tits groups, allowing us to prove intersection properties 
of different families of Shephard groups, including the Shephard groups $B_n(2,\infty)$ associated with Mikado braids.
We conclude with Section \ref{sec:future}, where we give directions for future works related to the content exposed in this article.

\subsection{Acknowledgments}
We would like to thank Luis Paris for answering our questions and pointing us to the article~\cite{ParisGodelle2012}, which contains the necessary results for showing that the coset incidence systems constructed from Artin-Tits groups are flag-transitive (Theorem~\ref{thm:ArtinProperties}(c)).
We are also grateful to Travis Scrimshaw for pointing out that the groups we were defining in Section~\ref{sec:applications} are known as Shephard groups. The second author would like to thank Dimitri Leemans for his continued support and numerous useful conversations.

The first author was partially supported by CMUP, member of LASI, which is financed by national funds through FCT -- Funda\c c\~ao para a Ci\^encia e a Tecnologia, I.P., under the projects with reference numbers UIDB/00144/2020 and UIDP/00144/2020.

\section{Background and Preliminary results}\label{sec:prelims}

In this section, we present the essential background on incidence geometry, coset complexes, and Coxeter and Artin–Tits groups required for the remainder of the paper. While most of the results discussed here are available in the literature, they are often spread across various sources and stated in different forms, using a variety of different conventions. This section thus also serves the secondary purpose of consolidating these scattered results and notation under a common roof, to make them more accessible to the reader.
\subsection{Incidence geometry and Coset Geometry}
\label{sec:back:subsec:incidence}
    A 4-tuple $\Gamma = (X,I,*,t)$ is called an \textit{incidence system} if
    \begin{enumerate}
        \item $X$ is a set whose elements are called the \textit{elements} of $\Gamma$,
        \item $I$ is the set of types of $\Gamma$, 
        \item $*$ is a symmetric and reflexive relation (called the \textit{incidence relation}) on $X$, and
        \item $t$ is a map from $X$ to $I$, called the \textit{type map} of $\Gamma$, such that distinct elements $x,y \in X$ with $x * y$ satisfy $t(x) \neq t(y)$. 
    \end{enumerate}

The \textit{rank} of $\Gamma$ is the cardinality of the type set $I$.
In an incidence system $\Gamma$, a \textit{flag} is a set of pairwise incident elements. The rank of a flag, \text{rank}(F), is the cardinality of $F$, that is $\textnormal{rank}(F)=|F|$. The type of a flag $F$, $t(F)$, is the set of types of the elements of $F$. 
A \textit{chamber} is a flag of type $I$. An incidence system $\Gamma$ is an \textit{incidence geometry} if all its maximal flags are chambers.

The \textit{incidence graph} of $\Gamma$ is a graph with vertex set $X$ and where two elements $x$ and $y$ are connected by an edge if and only if $x * y$.
We say $\Gamma$ is \textit{finite} if the incidence graph of $\Gamma$ is finite.

Let $F$ be a flag of $\Gamma$. An element $x\in X$ is {\em incident} to $F$ if $x*y$ for all $y\in F$. The \textit{residue} of $\Gamma$ with respect to $F$, denoted by $\Gamma_F$, is the incidence system formed by all the elements of $\Gamma$ incident to $F$ but not in $F$. Whenever $\Gamma$ is a geometry, the \textit{rank} of the residue $\Gamma_F$ is defined to be rank$(\Gamma) - |F|$. If $\Gamma$ is of rank $n$, then the \textit{corank} of a flag $F$ is the difference between the rank of $\Gamma$ and the rank of $F$.

An incidence geometry $\Gamma$ is \textit{connected} if its incidence graph is connected. It is \textit{residually connected} if all its residues of rank at least two are connected. Note that, by \cite[Lemma 1.6.3]{buekenhout2013diagram}, a residually connected geometry $\Gamma$ of rank greater or equal to $2$ is always connected. On the other hand, rank $1$ geometries are all trivially residually connected but are never connected. This fact is usually irrelevant as rank $1$ geometries have no structure and are thus intrinsically not very interesting. That being said, it will be relevant for this article, as we will allow to combine one or more rank $1$ geometries to build higher-rank geometries, and it will then be important to recall that rank $1$ geometries are technically residually connected.

An incidence geometry is \textit{firm}, \textit{thin} or \textit{thick} if all its residues of flags of corank one contain, respectively, at least two elements,  exactly two elements, more than two elements.

Let $\Gamma = \Gamma(X,I,*,t)$ be an incidence geometry. A {\em correlation} of $\Gamma$ is a bijection $\phi$ of $X$ respecting the incidence relation $*$ and such that, for every $x,y \in X$, if $t(x) = t(y)$ then $t(\phi(x)) = t(\phi(y))$. If, moreover, $\phi$ fixes the types of every element (i.e $t(\phi(x)) = t(x)$ for all $x \in X$), then $\phi$ is said to be an {\em automorphism} of $\Gamma$. The group of all correlations of $\Gamma$ is denoted by $\Cor(\Gamma)$ and the automorphism group of $\Gamma$ is denoted by $\Aut(\Gamma)$. Remark that $\Aut(\Gamma)$ is a normal subgroup of $\Cor(\Gamma)$ since it is the kernel of the action of $\Cor(\Gamma)$ on $I$.

If $\Aut(\Gamma)$ is transitive on the set of flags of $\Gamma$ then we say that $\Aut(\Gamma)$ is {\em flag-transitive} on $\Gamma$. We say $\Aut(\Gamma)$ is {\em chamber-transitive} on $\Gamma$ if $\Aut(\Gamma)$ is transitive on the set of chambers of $\Gamma$. Note that, if $\Gamma$ is a geometry, being chamber-transitive is equivalent to being flag-transitive~\cite[Proposition 2.2]{hypertopes}. If moreover, the stabilizer of a chamber in $\Aut(\Gamma)$ is reduced to the identity, we say that $\Gamma$ is {\em simply transitive} or {\em regular}. If $G$ is any group together with a homomorphism $\psi \colon G \to \Aut(\Gamma)$, we also say that $G$ is flag-transitive, simply transitive or regular on $\Gamma$ if that is true of $\psi(G) \leq \Aut(\Gamma)$.
A \textit{regular hypertope} is a thin, residually connected and flag-transitive incidence geometry~\cite{hypertopes}.

The main interest of this paper is incidence geometries that are obtained from a group $G$ together with a set $(G_i)_{i \in I}$ of subgroups of $G$ as described in~\cite{Tits1957}. 
    A \emph{coset incidence system} $\Gamma(G,(G_i)_{i \in I})$, denoted also as $(G,(G_i)_{i \in I})$, is the incidence system over the type set $I$ where:
    \begin{enumerate}
        \item The elements of type $i \in I$ are left cosets of the form $g G_i $, $g \in G$.
        \item The incidence relation is given by non-empty intersection. More precisely, the element $g G_i$ is incident to the element $k G_j $ if and only if $g G_i \cap k G_j \neq \emptyset$.
    \end{enumerate}
We usually call the subgroups $G_i$ as \emph{(standard) maximal parabolic subgroups}. For $J\subseteq I$, we define \emph{(standard) parabolic subgroups} as $G_J=\cap_{i\in J} G_i$, and define $G_\emptyset=G$.
When $J=I\setminus\{i\}$, we say that $G_J=G_{I\setminus\{i\}}$ are the \emph{minimal parabolic subgroups}, and we will denote them as $G^i$. Finally, the subgroup $G_I=\cap_{i\in I} G_i$ is called the \emph{Borel subgroup} of $\Gamma$. 
    
A coset incidence system that is an incidence geometry is called a \emph{coset geometry}.
Flag-transitive geometries can always be constructed as coset geometries.
The following theorem gives a group theoretical condition to check the flag-transitivity of $G$ on the coset incidence system $\Gamma(G,(G_i)_{i\in I})$. Notice that $G$ always naturally acts on $\Gamma(G,(G_i)_{i \in I})$ by left multiplication.

\begin{thm}\cite[Theorem 1.8.10]{buekenhout2013diagram}\label{thm:cosetFT}
Let $\Gamma = (G,(G_i)_{i\in I})$ be a coset incidence system. Then $G$ is flag-transitive on $\Gamma$ if and only if $G_JG_i = \cap_{j \in J}(G_jG_i)$ for each $J \subseteq I$ and each $i\in I\setminus J$.
Additionally, if $G$ is flag-transitive on $\Gamma$, then $\Gamma$ is a geometry.
\end{thm}

The group-theoretical condition for flag-transitivity can be stated in many different ways, as shown in the following proposition.

\begin{prop}\label{prop:FTequivs}
    Let $\Gamma = (G,(G_i)_{i\in I})$ be a coset incidence system. Then, the following statements are equivalent:
    \begin{enumerate}
        \item $G$ is flag-transitive on $\Gamma$;\label{prop:FTequivs:itm:IsFT}
        \item for each $J\subseteq I$ and each $i\in I\setminus J$, $G_JG_i=\cap_{j\in J}(G_jG_i)$;\label{prop:FTequivs:itm:FTBueken}
        \item for each $J\subseteq 
        I$, there is a natural isomorphism of incidence systems over $I\setminus J$
        $$\phi_J: \Gamma(G_J, (G_{J\cup \{i\}})_{i\in I\setminus J})\to \Gamma_{\{G_j\mid j\in J\}}$$
        given by $\phi_J(aG_{J\cup \{i\}})= a G_i$, where $a\in G_J$ and $i\in I\setminus J$;\label{prop:FTequivs:itm:Isomorphism}
        \item for each distinct types $i,k\in I$, each subset $J\subseteq I\setminus\{i,k\}$ and each element $g\in G_J$, if $G_i\cap gG_k\neq \emptyset$, then $G_J\cap G_i\cap gG_k \neq \emptyset$;\label{prop:FTequivs:itm:FTPassini}
        \item for each non-empty subset $J\subseteq I$, if there is a family of elements of $G$, $(g_j)_{j\in J}$,  such that, for every choice of $i,j\in J$, $g_iG_i \cap g_j G_j\neq \emptyset$, then $\cap _{j\in J}g_jG_j = g G_J$ for some $g\in G$;\label{prop:FTequivs:itm:FTPassiniProp}
        \item for every choice of subsets $J,H,K$ of $I$ and of elements $f,g,h$ of $G$, if the cosets $fG_J$, $gG_H$ and $hG_K$ have pairwise nonempty intersection, then $fG_J\cap gG_H\cap hG_K\neq \emptyset$;\label{prop:FTequivs:itm:FTIntersectionCosets}
        \item for every $J,H,K\subseteq I$ we have 
            $(G_J\cap G_H)(G_J\cap G_K)= G_J\cap (G_HG_K)$;\label{prop:FTequivs:itm:FTProductOfIntersections}
        \item for every $J,H,K\subseteq I$ we have 
        $(G_JG_H)\cap (G_JG_K) = G_J(G_H\cap G_K)$.\label{prop:FTequivs:itm:FTIntersectionOfProducts}
    \end{enumerate}
\end{prop}
\begin{proof}
    We have that $(\ref{prop:FTequivs:itm:IsFT})\Leftrightarrow (\ref{prop:FTequivs:itm:FTBueken})\Leftrightarrow (\ref{prop:FTequivs:itm:Isomorphism})$ by Theorem~\ref{thm:cosetFT} and~\cite[Theorem 1.8.10]{buekenhout2013diagram}.
    Moreover, $(\ref{prop:FTequivs:itm:FTPassini})\Leftrightarrow (\ref{prop:FTequivs:itm:FTPassiniProp})$ and $(\ref{prop:FTequivs:itm:FTPassini})\Leftrightarrow (\ref{prop:FTequivs:itm:FTIntersectionCosets}) \Leftrightarrow (\ref{prop:FTequivs:itm:FTProductOfIntersections}) \Leftrightarrow (\ref{prop:FTequivs:itm:FTIntersectionOfProducts})$ by~\cite[pp. 306]{Pasini1994}.
    It remains only to show that $(\ref{prop:FTequivs:itm:IsFT})\Leftrightarrow (\ref{prop:FTequivs:itm:FTPassiniProp})$

    $(\ref{prop:FTequivs:itm:IsFT})
    \Rightarrow (\ref{prop:FTequivs:itm:FTPassiniProp})$.
    Suppose that $G$ is flag-transitive in $\Gamma=(G,(G_i)_{i\in I})$. Hence, for $J\subseteq I$, consider the two distinct type $J$ flags $F=\{G_j:j\in J\}$ and $F'= \{g_j G_j:j\in J\}$. As $F'$ is a flag, we have a family of elements $(g_j)_{j\in J}$ of elements of $G$ such that $g_iG_i\cap g_jG_j\neq \emptyset$, for any choice of $i,j\in J$. By the flag-transitivity of $G$, we have an element $g\in G$ such that $F' = gF$, i.e. $\forall j\in J,\ g_j G_j = g G_j$. Hence $\cap_{j\in J} g_j G_j = \cap_{j\in J} g G_j = g \cap_{j\in J} G_j = g G_J$.

    $(\ref{prop:FTequivs:itm:FTPassiniProp})
    \Rightarrow (\ref{prop:FTequivs:itm:IsFT})$. 
    Suppose that condition $(e)$ holds. As $\Gamma$ is a coset incidence system, we have always a base chamber $C=\{G_i:i\in I\}$. Let $F$ be the flag of type $J\subseteq I$ in $C$, i.e. $F= \{G_j:j\in J\}$.
    Consider another flag of type $J$, $F'=\{g_jG_j:j\in J\}$, where $g_j\in G$. As $F'$ is a flag, $g_iG_i\cap g_jG_j\neq \emptyset$, for any $i,j\in J$. Hence, by $(d)$ we have that there is a $g\in G$ such that $\cap_{j\in J} g_j G_j = g G_J$. Hence, for each $j\in J$ we have that $g \in g_jG_j$, i.e. $g G_j = g_j G_j$. This implies that $F' = \{g_j G_j:j\in J\} = \{g G_j:j\in J\} = gF$.
    Hence, we have proved we can go from $F$ to any flag of type $J$, and therefore $G$ is flag-transitive.
\end{proof}

The most interesting coset geometries $\Gamma = (G,(G_i)_{i\in I})$ are those for which $G$ acts flag-transitively on $\Gamma$. In that case, properties like residual connectedness, firmness and thinness can also be checked via group-theoretical conditions.

\begin{thm}\cite[Corollary 1.8.13]{buekenhout2013diagram} \label{thm:CosetRC}
    Suppose that $\Gamma = (G,(G_i)_{i\in I})$ is a coset incidence system over
the finite set $I$ on which G acts flag-transitively. Then $\Gamma$ is a residually connected geometry if and only if $G_J = \langle G_{\{i\} \cup J} | i \in I \setminus J \rangle$ for all $J \subset I$ such that $|I \setminus J| \geq 2$.
\end{thm}

A slight reformulation of the condition for residual-connectedness is presented in the following result.

\begin{lemma}\cite[Proposition 1.8.12]{buekenhout2013diagram} \label{prop:RCcondi}
    Let $\Gamma = (G,(G_i)_{i\in I})$ be a coset incidence system. Then the following statements are equivalent
    \begin{equation}\tag{RC1}
        G_J = \langle G_{\{i\} \cup J} | i \in I \setminus J \rangle \textnormal{ for all }J \subset I \textnormal{ such that }|I \setminus J| \geq 2.
    \end{equation}
    \begin{equation}\tag{RC2}
        G_J = \langle G_{\{i\} \cup J} , G_{\{k\} \cup J} \rangle \textnormal{ for all }J \subset I \textnormal{ and distinct }i,k\in I\setminus J.
    \end{equation}
\end{lemma}

Recall that the minimal parabolic subgroups of $\Gamma = (G,(G_i)_{i\in I})$ are the groups $G^i :=G_{I\setminus\{i\}}$. For a subset $J\subseteq I$, we let $G^J=\langle G^i\mid i\in J\rangle$. In other words, $G^J$ is the subgroup of $G$ generated by the minimal parabolic subgroups $G^i$, for $i\in J$. Notice that $G^I$ does not need to be necessarily equal to $G$. With these notations, the following conditions are all equivalent to (RC1) and (RC2).

\begin{lemma}\label{lem:RCequivs}
    Let $\Gamma = (G,(G_i)_{i\in I})$ be a coset incidence system. Then the following statements are equivalent:
    \begin{enumerate}
        \item for all $J \subset I$ such that $|I \setminus J| \geq 2$, we have
        $G_J = \langle G_{\{i\} \cup J} | i \in I \setminus J \rangle$\label{lem:RCequivs:itm:RC1}
        \item for all $J \subset I$ and distinct $i,k\in I\setminus J$, we have 
        $G_J = \langle G_{\{i\} \cup J} , G_{\{k\} \cup J} \rangle$ \label{lem:RCequivs:itm:RC2}
        \item for every $J\subseteq I$, $G_{J} = G^{I\setminus J}$.\label{lem:RCequivs:itm:BottomUp}
        \item for every $J\subseteq I$, $G^J = G_{I\setminus J}$.\label{lem:RCequivs:itm:UpBottom}
        \item for every $J,K\subseteq I$, $G^J\cap G^K = G^{J\cap K}$.\label{lem:RCequivs:itm:Intersection}
        \item for all $J \subset I$ with $| I \setminus J| \geq 2$ and for any system $(G_{K_i})$ of parabolic subgroups such that $I \setminus J \subseteq \cup_i (I \setminus K_i)$, we have that $G_J \subseteq \langle G_{K_i} \rangle$. Note that the equation $I \setminus J \subseteq \cup_i (I \setminus K_i)$ can be rewritten as $\cap_i K_i \subseteq J$.\label{lem:RCequivs:itm:Parabolics}
    \end{enumerate}
    Moreover, if one (hence, all) of this statements is true, and $\Gamma$ is flag-transitive, then $\Gamma$ is a residually connected coset geometry.
\end{lemma}
\begin{proof}
    By Lemma~\ref{prop:RCcondi}, we have that $(\ref{lem:RCequivs:itm:RC1})\Leftrightarrow (\ref{lem:RCequivs:itm:RC2})$. Moreover, $(\ref{lem:RCequivs:itm:BottomUp})\Leftrightarrow (\ref{lem:RCequivs:itm:Intersection})$ by~\cite[pp.306]{Pasini1994}.
    Let us prove then that $(\ref{lem:RCequivs:itm:BottomUp})\Leftrightarrow (\ref{lem:RCequivs:itm:UpBottom})$,
    $(\ref{lem:RCequivs:itm:RC1}) \Rightarrow (\ref{lem:RCequivs:itm:UpBottom})$, $(\ref{lem:RCequivs:itm:UpBottom})\Rightarrow (\ref{lem:RCequivs:itm:RC2})$, $(\ref{lem:RCequivs:itm:UpBottom})\Rightarrow (\ref{lem:RCequivs:itm:Parabolics})$ and $(\ref{lem:RCequivs:itm:Parabolics})\Rightarrow (\ref{lem:RCequivs:itm:RC1})$.

    $(\ref{lem:RCequivs:itm:BottomUp})\Leftrightarrow (\ref{lem:RCequivs:itm:UpBottom})$: Consider that for every $J\subseteq I$, $G_{J} = G^{I\setminus J}$. Let $K = I\setminus J$ and notice that $J=I\setminus K$. Hence $G^K = G_{I\setminus K}$. The other direction is equally trivial.

    $(\ref{lem:RCequivs:itm:RC1}) \Rightarrow (\ref{lem:RCequivs:itm:UpBottom})$: 
    Assume $\Gamma$ respects (RC1).
    We will proceed by induction on the size of the set $J$.
    Suppose $|J|=1$. Hence $J=\{j\}$, for some $j\in I$.
    By definition we have that that $G_{I\setminus \{j\}} = G^j$.
    Suppose then that, for $n\geq 2$ such that $|J|<n$, our hypothesis works, i.e. $G_{I\setminus J}= G^J$. Then consider $G_{I\setminus J}$, with $|J|=n$.
    By $(\ref{lem:RCequivs:itm:RC1})$ we have that $G_{I\setminus J} = \langle G_{(I\setminus J) \cup \{k\}} \ : \ k\in J\rangle$. Notice that, as $k\in J$, $(I\setminus J)\cup \{k\} = I\setminus(J\setminus\{k\})$.
    As for all $k\in J$, $|J\setminus\{k\}|<n$, we have, by the induction hypothesis, that 
    $G_{I\setminus J} =\langle G^{J\setminus\{k\}}\ : \ k\in J\rangle$. By the definition of $G^{J\setminus\{k\}}$, we can simplify to  $G_{I\setminus J}= \langle G^k\ : k\in J\rangle$.
    Hence, we have proved that $G_{I\setminus J} = G^J$.

    $(\ref{lem:RCequivs:itm:UpBottom})\Rightarrow (\ref{lem:RCequivs:itm:RC2})$: 
    Assume now that, for any $J\subseteq I$, $\Gamma$ respects $G_{I\setminus J} = G^J$. 
    This is equivalent to saying that $G_{J} = G^{I\setminus J}$, for any $J\subseteq I$.
    Consider $J\subset I$, distinct $i,j\in I\setminus J$, and consider the group $\langle G_{J\cup \{i\}}, G_{J\cup\{j\}}\rangle$.
    Hence, we have that $\langle G_{J\cup \{i\}}, G_{J\cup\{j\}}\rangle = \langle G^{I\setminus(J\cup \{i\})}, G^{I\setminus(J\cup\{j\})}\rangle$.
    Finally, we get that $$\langle G^{I\setminus(J\cup \{i\})}, G^{I\setminus(J\cup\{j\})}\rangle = \langle G^k\ | \ k\in (I\setminus(J\cup \{i\}))\cup (I\setminus(J\cup \{j\}))\rangle = \langle G^k\ | \ k \in I\setminus J\rangle,$$
    which is $G^{I\setminus J} = G_J$.
    Hence $G_J = \langle G_{J\cup \{i\}}, G_{J\cup\{j\}}\rangle$, proving $(\ref{lem:RCequivs:itm:RC2})$.
    
    $(\ref{lem:RCequivs:itm:BottomUp})\Rightarrow (\ref{lem:RCequivs:itm:Parabolics})$: By $(\ref{lem:RCequivs:itm:BottomUp})$, we have that $G_J = \langle G^j, j \in I \setminus J \rangle$. Under the condition of $(\ref{lem:RCequivs:itm:Parabolics})$ on the system $(G_{K_i})$, we get that every $G^j$ with $j \in I\setminus J$ is contained in some $G_{K_i}$. Hence, $\langle G_{K_i} \rangle$ contains $\langle G^j, j \in I \setminus J \rangle = G_J$.
    
    $(\ref{lem:RCequivs:itm:Parabolics})\Rightarrow (\ref{lem:RCequivs:itm:RC1})$: It suffices to take $(G_{K_i}) = (G_{J \cup \{i\}}, i \in I \setminus J)$. We then obtain that $\langle G_{J \cup \{i\}}, i \in I \setminus J \rangle$ contains $G_J$. The other inclusion always holds.
    
    Lastly, suppose that one of the conditions above applies to $\Gamma$. If $\Gamma$ is flag-transitive, then $\Gamma$ is a coset geometry, and by Theorem~\ref{thm:CosetRC}, we get that $\Gamma$ is residually connected.

\end{proof}

Finally, we present here the group theoretical conditions for thinness, firmness and finiteness of flag-transitive coset geometries. The condition for thinness is not stated verbatim in \cite{buekenhout2013diagram}, but it is an almost trivial consequence of the result for firmness.

\begin{thm}\label{thm:CosetFIRM}\cite[Corollary 1.8.15]{buekenhout2013diagram}
     Suppose that $\Gamma = (G,(G_i)_{i\in I})$ is a coset incidence system over
the finite set $I$ on which G acts flag-transitively. Then $\Gamma$ is a firm geometry if and only if
\begin{equation}\tag{FIRM}
    G_{I\setminus \{j\} } \neq G_I \textnormal{ for each }j \in I.
\end{equation}
\end{thm}

\begin{thm}\label{thm:CosetThin}\cite[Corollary 1.8.15]{buekenhout2013diagram}
     Suppose that $\Gamma = (G,(G_i)_{i\in I})$ is a coset incidence system over
the finite set $I$ on which $G$ acts flag-transitively. Then $\Gamma$ is a thin geometry if and only if
\begin{equation}\tag{THIN}
    [G_{I\setminus \{j\} } : G_I] = 2 \textnormal{ for each }j \in I.
\end{equation}
\end{thm}
\begin{proof}
    The proof of \cite[Corollary 1.8.15]{buekenhout2013diagram} can be adapted easily to this case. Since $G$ acts flag-transitively on $\Gamma = (G,(G_i)_{i\in I})$, we have that the residues of co-type $j$ in $\Gamma$ are isomorphic to $(G_{I\setminus \{j\} }, (G_I))$ (see Proposition~\ref{prop:FTequivs}(\ref{prop:FTequivs:itm:Isomorphism})). Therefore, these residues contain exactly two elements if and only if $[G_{I\setminus \{j\} } : G_I] = 2$.
\end{proof}

\begin{prop}\label{prop:FiniteCosetGeom}
    Let $\Gamma=(G,(G_i)_{i\in I})$ be a coset incidence system.
    Suppose $G$ is flag-transitive on $\Gamma$.
    Then $\Gamma$ is finite if and only if $[G:G_I]$ is finite.
\end{prop}
\begin{proof}
    Since $G$ acts flag-transitively on $\Gamma$, $G$ also acts transitively on the set $\mathcal{C}$ of chambers of $\Gamma$. Moreover, as $G_I$ is the stabilizer of a chamber under this action, we can always identify the set $\mathcal{C}$ with the set $G/G_I$ of cosets of $G_I$ in $G$. 
    Notice also that $\Gamma$ is finite if and only if $\mathcal{C}$ is finite. 
    Indeed, suppose first that $\Gamma$ is finite.
    Therefore, by the orbit-stabilizer theorem, we get that $|\mathcal{C}| = |G / G_I| = [G : G_I]$. Conversely, if $[G : G_I]$ is finite, so is $\mathcal{C}$ as we can identify this set with the cosets $G/G_I$.
\end{proof}

From now on, we will say that a coset incidence system $\Gamma$ is $(X)$, with $X\in\{$RC1,RC2,FIRM,THIN$\}$, whenever its parabolic subgroups satisfy the group theoretical condition $(X)$. Beware that satisfying the property $(X)$ does not in general guarantee that the geometry $\Gamma$ is residually-connected, firm or thin, respectively. This only holds for coset geometries on which $G$ acts flag-transitively. Indeed, if the group $G$ does not encode all symmetries of $\Gamma$, it would be rather optimistic to expect $G$ to describe $\Gamma$ accurately.

\subsection{Coset complex}

Here, we will follow the notation given by Godelle and Paris in~\cite{ParisGodelle2012} for simplicial complexes.
Let $\KK$ be an (abstract) simplicial complex, and let $V$ be its set of vertices. A subset $F$ of $V$ is called a \emph{presimplex} if $\{u,v\}$ is an edge (i.e. a 1-simplex) for all $u,v \in F$, with $u\neq v$. We
say that $\KK$ is a \emph{flag complex} if every presimplex is a simplex.

It is a very common mantra to say that groups are best studied by understanding their actions on some spaces. In the context of coset incidence systems, there is a natural space that we can use. 
Consider a coset incidence system $\Gamma=(G,(G_i)_{i\in I})$, and consider the (abstract) simplicial complex $\KK=\KK(\Gamma)$, where its vertex set is the set of elements of $\Gamma$, i.e. $V(\KK)= \{gG_i\mid g\in G, i\in I\}$, the union of all the cosets $G/G_i$, for all $i\in I$.
A subset $\{g_1 G_{i_1}, \cdots, g_n G_{i_n} \}$ of vertices forms a simplex of $\mathcal{K}$ if $g_1 G_{i_1} \cap \cdots \cap g_n G_{i_n} \neq \emptyset$.
Notice that the set of all presimplices of $\KK(\Gamma)$ corresponds to the set of flags of $\Gamma$. In fact, checking if a subset of vertices is a presimplex turns out to be simply checking pairwise incidence inside this subset of cosets. Moreover, we have the following result.

\begin{prop}\label{prop:FComp:equiv:FT}
    Let $\Gamma=(G,(G_i)_{i\in I})$ be a coset incidence system. The simplicial complex $\KK(\Gamma)$ is a flag complex if and only if $G$ is flag-transitive on $\Gamma$.
\end{prop}
\begin{proof}
    Consider first that $G$ is flag-transitive in $\Gamma$.
    Let $F=\{g_1 G_{i_1}, \cdots, g_n G_{i_n} \}$ be an arbitrary flag of $\Gamma$. Being a flag, we have that this determines a presimplex of $\KK(\Gamma)$.
    By Proposition~\ref{prop:FTequivs}(\ref{prop:FTequivs:itm:FTPassiniProp}), we have that 
    $g_1 G_{i_1} \cap \cdots \cap g_n G_{i_n} \neq \emptyset$. Hence every presimplex of $\KK(\Gamma)$ is a simplex.
    
    Suppose now that $\KK(\Gamma)$ is a flag complex. Hence, for every flag $\{g_1 G_{i_1}, \cdots, g_n G_{i_n} \}$ of $\Gamma$, we have that $g_1 G_{i_1} \cap \cdots \cap g_n G_{i_n} \neq \emptyset$.
    We will prove that this implies the condition $(\ref{prop:FTequivs:itm:FTIntersectionCosets})$ of Proposition~\ref{prop:FTequivs}. In Proposition~\ref{prop:FTequivs}$(\ref{prop:FTequivs:itm:FTIntersectionCosets})$, we can consider, without loss of generality, that the sets $J,H,K$ are disjoint. Indeed, suppose that $i \in J \cap H$ for some $i \in I$. Then, either $fG_i = gG_i$ or $fG_i$ is disjoint from $gG_i$. In the first case, we can choose to remove $i$ from $H$ (or $J$) without altering any of the intersections appearing in the statement. In the second case, we will have that $fG_J \cap gG_H = \emptyset$, which would make $fG_J$ and $gG_H$ non-incident, a contradiction. Hence, since we can consider $J,H,K$ to be disjoint, we can look at the set $\{f G_j : j \in J\} \cup \{g G_h : h \in H\} \cup\{h G_k : k \in K \} $ of vertices of $\mathcal{K}$. The assumption of Proposition~\ref{prop:FTequivs}$(\ref{prop:FTequivs:itm:FTIntersectionCosets})$ implies then that this set will form a presimplex. Hence, since $\mathcal{K}$ is a flag complex, we have that the desired intersection is non-empty.
\end{proof}

\subsection{Coxeter and Artin-Tits groups}\label{sec:back:subsec:CoxArtGroupComplex}

Our main source of examples in this article will come from groups known as Coxeter and Artin-Tits groups. We give their definition here and explain the standard way to obtain coset incidence systems from Coxeter and Artin-Tits systems.

A group $W$ is said to be a \emph{Coxeter group} if $W$ admits a presentation of the form
$$\langle r_1,r_2,\ldots,r_n\mid (r_ir_j)^{m_{i,j}}=1 \rangle,$$
where, for $i\neq j$, $m_{i,j}=m_{j,i}\geq 2$ is either an integer or $\infty$, and $m_{i,i}=1$ for all $i =1,\cdots, n$. Such a presentation will be called a \textit{Coxeter presentation} for $W$.
Let $W$ be a Coxeter group and let $S= \{r_1,\cdots, r_n\}$ be a set of generators in a Coxeter presentation for $W$. The pair $(W,S)$ is called a \emph{Coxeter system}. For $J\subseteq \{1,\ldots,n\}=I$, let $W^{(J)}=\langle r_i\mid i\in J\rangle$ be the subgroup of $W$ generated by the generators with index in $J$.
\begin{prop}\cite[Chapter 4, Section §1,Theorem 2; Chapter 4, Exercise §1.1.b)]{Bourbaki2006}\label{prop:CoxeterProperties}
    Let $(W,S)$ be a Coxeter system.
    Then
    \begin{enumerate}
            \item For $J\subseteq \{1,\ldots,n\}$, the subgroup $W^{(J)}$ together with the set $S^J=\{r_j\mid j\in J \}$
            forms a Coxeter system $(W^{(J)},S^J)$. 
            \item For $J,K\subseteq \{1,\ldots,n\}$, we have that $W^{(J)} \cap W^{(K)} =W^{(J\cap K)}$;
            \item For $J,K,L\subseteq \{1,\ldots,n\}$, we have that $(W^{(J)}W^{(K)})\cap (W^{(J)}W^{(L)})=W^{(J)}(W^{(K)}\cap W^{(L)})$.
    \end{enumerate}
\end{prop}
Notice that, although the definition of $W^{(J)}=\langle r_i\mid i\in J\rangle$ does not match our usual $W^J=\langle W^j\mid j\in J\rangle$, by Proposition~\ref{prop:CoxeterProperties}$(b)$, we can easily prove that these are equivalent.
Let the maximal parabolic subgroups be $W_i = \langle r_j\mid j\in I\setminus\{i\}\rangle=W^{(I\setminus\{i\})}$. For $j\in I$, the minimal parabolic subgroup $W^j = \cap_{i\in I\setminus\{j\}} W^{(I\setminus\{i\})} = \langle r_j\rangle$, hence the definition of $W^J$ as $\langle W^j\mid j\in J\rangle$ is equivalent to $\langle r_j\mid j\in J\rangle$, so that, in this case $W^J=W^{(J)}$. 

A group $A$ is said to be an \emph{Artin-Tits group} if $A$ admits a presentation of the form
$$\langle a_1,a_2,\ldots,a_n\mid (a_ia_j)_{m_{i,j}}=(a_ja_i)_{m_{j,i}} \rangle,$$
where, for $i\neq j$, $m_{i,j}=m_{j,i}\geq 2$ is either an integer or $\infty$, and 
 \begin{equation*}
       (a_ia_j)_{m_{i,j}} =
        \begin{cases}
            \underbrace{a_ia_ja_i\ldots a_j }_{m_{i,j}\textnormal{ factors}}, \textnormal{ if } m_{i,j} \textnormal{ is even } \\
            \underbrace{a_ia_ja_i\ldots a_i }_{m_{i,j}\textnormal{ factors}}, \textnormal{ if } m_{i,j} \textnormal{ is odd }
        \end{cases}
    \end{equation*}

Notice that, contrary to the case of Coxeter groups, the generators $a_i$ have infinite order. A presentation of this form is called an \textit{Artin-Tits} presentation.
Let $A$ be an Artin-Tits group and $S$ be a set of generators of $A$ from an Artin-Tits presentation for $A$. The pair $(A,S)$ is called an \emph{Artin system}. For $J\subseteq \{1,\ldots,n\}$, let $A^{(J)}=\langle a_i\mid i\in J\rangle$ be the subgroup of $A$ generated by the generators whose index is in $J$.

    \begin{thm}\cite[Chapter II, Theorem 4.13]{van1983homotopy}\cite[Proposition 4.5]{ParisGodelle2012}\label{thm:ArtinProperties}
        Let $(A,S)$ be an Artin system, where $A$ is the Artin-Tits group, and $S=\{a_1,\ldots,a_n\}$ is the set of generators.
        Then
        \begin{enumerate}
            \item For $J\subseteq \{1,\ldots,n\}$, the subgroup $A^{(J)}$ together with the set $S^J=\{s_j\mid j\in J \}$
            forms an Artin system $(A^{(J)},S^J)$. 
            \item For $J,K\subseteq \{1,\ldots,n\}$, we have that $A^{(J)} \cap A^{(K)} =A^{(J\cap K)}$;
            \item For $J,K,L\subseteq \{1,\ldots,n\}$, we have that $(A^{(J)}A^{(K)})\cap (A^{(J)}A^{(L)})=A^{(J)}(A^{(K)}\cap A^{(L)})$.
        \end{enumerate}
    \end{thm} 

As for Coxeter groups, the subgroups $A^J = \langle A^j\mid j\in J\rangle$ and $A^{(J)}$ match due to condition $(b)$ of Theorem~\ref{thm:ArtinProperties}.
Let $\mathcal{A} = (A,S)$ be either a Coxeter system or an Artin-Tits system with $S = \{  a_1, \cdots, a_n \}$. We can naturally define a system $(A_i)_{i \in I}$ of subgroups of $A$ over the type set $I = \{1,\cdots,n\}$ by setting $A_i$ to be the subgroup of $A$ generated by all the generators except for $a_i$. We can then always form the coset incidence system $\alpha =(A,(A_i)_{i \in I})$. 
Moreover, due to the above results, we have that $A_J = \langle r_i \mid i\in I\setminus J \rangle$ and $A^J = \langle r_i\mid i\in J \rangle$, for $J\subseteq I$.
It turns out that $A$ always acts flag-transitively on $\alpha$, that $\alpha$ is residually connected and that $\alpha$ is thin if $A$ is a Coxeter group and $\alpha$ is thick if $A$ is an Artin-Tits group.

\begin{thm}\label{thm:Artin_Coxeter_CG}
    Let $\mathcal{A}=(A,S)$ be either a Coxeter or an Artin-Tits system, where $S=\{a_i\mid i\in I\}$, with $I=\{1,\ldots,n\}$.
    If $\alpha=(A,(A_i)_{i\in I})$ is the coset incidence system, with $A_i=\langle a_j \mid j\in I\setminus\{i\}\rangle$, then
    \begin{enumerate}
        \item $A$ acts flag-transitively on $\alpha$;
        \item $\alpha$ is a geometry;
        \item $\alpha$ is residually connected;
        \item $\alpha$ is thin if $\mathcal{A}$ is a Coxeter system and $\alpha$ is thick if $\mathcal{A}$ is an Artin-Tits system;
        \item $\KK(\alpha)$ is a flag-complex.
    \end{enumerate}

\end{thm}
\begin{proof}
    (a) This is a direct consequence of the definition of the parabolic subgroups of $\alpha$, the item (c) in both Proposition~\ref{prop:CoxeterProperties} and Theorem~\ref{thm:ArtinProperties}, and Proposition~\ref{prop:FTequivs}.
    
    (b) A flag-transitive coset incidence system is always a geometry.
    
    (c) In both cases, it is known that $A_J = \langle a_i \mid i \in I \setminus J \rangle$. We can then immediately conclude by item (a), Lemma \ref{lem:RCequivs} (\ref{lem:RCequivs:itm:BottomUp}) and Theorem~\ref{thm:ArtinProperties}, for example.
    
    (d) Since $\alpha$ is flag-transitive, we have that the size of the residues of rank $1$ are the indices of $A^i$ in $A_I = \{1_A\}$. When $A$ is a Coxeter group, we have that $A^i = \langle a_i \rangle$ is of size two, while it is infinite when $A$ is an Artin-Tits group.

    (e) This is a direct consequence of item (a) and Proposition~\ref{prop:FComp:equiv:FT}.
\end{proof}

\subsection{Permutation of Maximal Parabolic Subgroups}\label{sec:pre:subsec:PermMaxPara}

We conclude this section by an elementary result that will be used in various places in the remainder of the article. Let $(G_i)_{i\in I}$ be a family of subgroups of $G$ and let $G_J = \cap_{i\in J}G_i$.
Consider $\phi : (G_i)_{i\in I} \rightarrow (G_i)_{i\in I}$ to be an automorphism of $G$ such that it sends subgroups in the family $(G_i)_{i\in I}$ to subgroups in the same family.
This extends to a permutation on the set of types $I$, $\phi_I: I\rightarrow I$, where $\phi(G_i)=G_{\phi_I(i)}$, for $i\in I$. 

\begin{lemma}\label{lem:lattice}
    $\phi(G_J)=\cap_{j\in J} \phi(G_j)$
\end{lemma}
\begin{proof}
    Let $J\subseteq I$ and let $G_J=\cap_{j\in J} G_j$.
    Consider $x\in G_J= \cap_{j\in J} G_j $. Then any element of $\phi(G_J)$ is of the form $\phi(x)$, for $x\in G_J$. Since $x\in G_j$, for each $j\in J$, then $\phi(x)\in \phi(G_j)$, for each $j\in J$. Hence $\phi(x)\in \cap_{j\in J} \phi(G_j)$.

    Consider now $y\in \cap_{j\in J} \phi(G_j)$. Then, for each $j\in J$, we have that $y\in \phi(G_j)$. Hence $\phi^{-1}(y)\in G_j$, for each $j\in J$, so that $\phi^{-1}(y)\in G_J$. This implies that $\phi(\phi^{-1}(y)) = y\in \phi(G_J)$.
\end{proof}

\section{Free amalgamated products and HNN-extensions of coset incidence systems}\label{sec:Amalgams+HNN}

In this section, we show how to extend the concepts of free products, free amalgamated products and HNN-extensions to coset incidence systems.
We begin by establishing some notations and definitions that we will use throughout the whole section. Let $G$ be a group. A \textit{word} $w$ in $G$ is a succession $w = g_1 \cdots g_n$ of elements of $G$. We say that the \textit{length} of $w = g_1 \cdots g_n $ is $n$. Given a word $w = g_1 \cdots g_n$, a \textit{prefix} of $w$ is any word $w'= g_1 \cdots g_j$ with $ 1 \leq j \leq n$ obtained by taking only the $j$ first elements of the word $w$. Similarly, a \textit{suffix} of $w$ is any word $w'= g_j \cdots g_n$ with $n \geq j \geq 1$ obtained by taking only the $n-(j-1)$ last elements of the word $w$. 

We begin with the case of free products, as it is the simplest of the three cases. Since free products are a special case of free amalgamated products, almost all results contained in this section are implied by the ones in the free amalgamated product section. Nonetheless, the core ideas in both sections are very similar, and free products are simpler to handle notation-wise. Moreover, when the starting coset incidence systems $\alpha$ and $\beta$ have non-trivial Borel subgroups, the free product $\alpha * \beta$ is not a special case of our construction for free amalgamated products. Indeed, the Borel subgroup of $\alpha *\beta$ will always be the free product of the Borel subgroups of $\alpha$ and $\beta$. In the case of free amalgamated products, we will always require that at least the Borel subgroups of $\alpha$ and $\beta$ are isomorphic and identified through the amalgamating isomorphism $\varphi$, so that the two constructions do not coincide entirely in this case.

For more detailed notations and definitions regarding words, normal forms and reduced sequences for free products (with amalgamation) and HNN-extensions, please refer to~\cite{magnus2004combinatorial,Lyndon2001}.

\subsection{Free Products of Coset Incidence Systems}

The goal of this section is to show how to build a coset incidence system for the free product $G = A * B$ from two existing coset incidence systems $\alpha = (A, (A_i)_{i \in I_\alpha})$ and $ \beta =(B, (B_i)_{i \in I_\beta})$.

Let $A$ and $B$ be two groups and $A * B$ be their free product. Any element $g  \in A * B$ can be represented by a word $w = g_1 \cdots g_n$ where $g_i \in A$ or $B$ for each $i = 1,\cdots, n$. We say that a sequence $g_1, \cdots, g_n$ is \textit{reduced} if, for each $i = 1, \cdots, n$, we have that $g_i \in A$ or $B$, $g_i \neq 1_G$ and any two consecutive $i,i+1$, $g_i$ and $g_{i+1}$ are never in the same factor $A$ or $B$. We allow the empty sequence as a reduced sequence, and it will represent the identity.

\begin{thm}[Normal form for free products]\cite[Theorem 1.2 -- Chapter IV]{Lyndon2001}\label{thm:normal form Free}
    Let $G = A * B$ be the free product of two groups $A$ and $B$. Then, the following are equivalent:
    \begin{itemize}
        \item If $w = g_1\cdots g_n$ with $n >1$ and $g_1, \cdots, g_n$ a reduced sequence, then $w \neq 1_G$ in $G$.
        \item Each element of $G$ can be written uniquely as $w = g_1 \cdots g_n$ where $g_1, \cdots g_n$ is a reduced sequence.
    \end{itemize}
\end{thm}

We begin by a technical lemma about intersections and product of some particular type of subgroups of free products.
We say that a subgroup $H$ of $G$ is a \textit{special subgroup} if $H = H_A * H_B$ for some subgroups $H_A \leq A$ and $H_B \leq B$.

\begin{lemma}\label{lem:freeProductIntersection}
    Let $C,D,E$ be three special subgroups of $G = A * B$ such that $C = C_A * C_B, D = D_A * D_B$ and $E = E_A * E_B$. Then,
    \begin{enumerate}
        \item $C \cap D = (C_A \cap D_A) * (C_B \cap D_B)$;
        \item If, moreover, we have that $C_AD_A \cap C_AE_A = C_A(D_A \cap E_A)$ and $C_BD_B \cap C_BE_B = C_B(D_B \cap E_B)$, then $CD \cap CE = C(D \cap E)$.
    \end{enumerate}
\end{lemma}

\begin{proof}
    (a) This is more or less a direct consequence of the normal form theorem for free products. Indeed, any $g \in G$ is uniquely represented by a reduced sequence $g_1, \cdots g_n$. If $g$ is in $C$ (in $D$, respectively), it follows that the elements of this sequence are in $C_A$ and $C_B$ (in $D_A$ and $D_B$, respectively).
    Hence, if $g \in C \cap D$, all elements of the reduced sequence $g_1, \cdots g_n$ should be in $C_A \cap D_A$ or $C_B \cap D_B$, implying that $g \in (C_A \cap D_A) * (C_B \cap D_B)$. This shows that $C \cap D \subseteq (C_A \cap D_A) * (C_B \cap D_B)$. The reverse inclusion holds trivially.

    (b) The inclusion $C(D \cap E) \subseteq CD \cap CE$ always holds. We thus only need to prove the reverse inclusion. Suppose now that $g \in CD$, meaning that $g$ is the product of some reduced word $w = h_1k_1h_2k_2 \cdots h_mk_m$ of $C$ and some reduced word $\Bar{w} =\Bar{h}_1\Bar{k}_1 \cdots \Bar{h}_n\Bar{k}_n$ of $D$. Then $$g = w \Bar{w} = h_1k_1h_2k_2 \cdots h_mk_m \Bar{h}_1\Bar{k}_1 \cdots \Bar{h}_n\Bar{k}_n$$ with $h_i \in C_A, k_i \in C_B, \Bar{h}_i \in D_A$ and $\Bar{k}_i \in D_B$, and all of the elements are non trivial, except perhaps $h_1,k_m,\Bar{h}_1$ and $\Bar{k}_n$. We will divide the proof in different cases, according to whether $k_m$ and $\Bar{h}_1$ are trivial or not.
    
    Suppose first that neither $k_m$ nor $\Bar{h}_1$ are trivial. We can then always assume that $\Bar{h}_1$ is not in $C_A$ so that $w = h_1k_1h_2k_2 \cdots h_mk_m$ is the longest prefix of $g$ belonging to $C$. Indeed, if $\Bar{h}_1$ was in $C_A$, we would just add it to $w$ and continue to add the elements of $\Bar{w}$ to $w$ as long as it is possible and rename $w$ and $\Bar{w}$ accordingly. Suppose now that $g \in CE$. Then, we must be able to express $g =w\Bar{w}$ as a product of a reduced word in $C$ followed by a reduced word in $E$. Since $w$ is the longest prefix of $g$ belonging to $C$, this means that there exists a suffix $h =x_j \cdots h_mk_m \Bar{w}$ of $g$ that belongs to $E$, where $x_j$ is either $k_j$ or $h_jk_j$, which is in $E$. As $h$ is a reduced word in $E$ if and only if all its sub-words are either in $E_A$ or $E_B$, then $\Bar{w}$ is in $E$. Hence, $g = w \Bar{w}$ with $w \in C$ and $\Bar{w} \in D \cap E$.

    Suppose now that $k_m$ is trivial so that $$g = h_1k_1h_2k_2 \cdots h_{m-1}k_{m-1} h  \Bar{k}_1 \cdots \Bar{h}_n\Bar{k}_n,$$ where $h = h_m \Bar{h}_1$. If $h$ is in $D_A$ or $C_A$, then we can use the previous case by including $h$ in either $w$ or $\Bar{w}$. We can thus suppose that $h$ is in neither $C_A$ nor $D_A$ so that $h_1k_1h_2k_2 \cdots h_{m-1}k_{m-1}$ is the largest prefix of $g= h_1k_1h_2k_2 \cdots h_{m-1}k_{m-1} h  \Bar{k}_1 \cdots \Bar{h}_n\Bar{k}_n$ contained in $C$ and $\Bar{k}_1 \cdots \Bar{h}_n\Bar{k}_n$ is the largest suffix of $g$ contained in $D$. Suppose that $g \in CE$. As in the previous case, this implies that $\Bar{k}_1 \cdots \Bar{h}_n\Bar{k}_n \in E$. We thus have that $$g = (h_1k_1h_2k_2 \cdots h_{m-1}k_{m-1} h_m)\Bar{h}_1 (\Bar{k}_1 \cdots \Bar{h}_n\Bar{k}_n)$$ with $h_1k_1h_2k_2 \cdots h_{m-1}k_{m-1} h_m \in C$ and $\Bar{k}_1 \cdots \Bar{h}_n\Bar{k}_n \in E$. This implies that $\Bar{h}_1 \in E$ so that $h \in C_AE_A$. But $h$ was in $C_AD_A$, so, by hypothesis, we have $h \in C_A (D_A \cap E_A)$. Therefore, $$g =  h_1k_1h_2k_2 \cdots h_{m-1}k_{m-1} h_m'\Bar{h}_1'  \Bar{k}_1 \cdots \Bar{h}_n\Bar{k}_n$$ with $h_i,h_m' \in C_A, k_i \in C_B, \Bar{h}_1',\Bar{h}_i \in D_A \cap E_A$ and $\Bar{k}_i \in D_B \cap E_B$. Hence, $g \in C(D\cap E)$.

    Finally, suppose that $\Bar{h}_1$ is trivial. We then have that $$g =  h_1k_1h_2k_2 \cdots h_{m-1}k_{m-1} h_m k  \cdots \Bar{h}_n\Bar{k}_n,$$ where $k = k_m \Bar{k}_1$. This case is then identical to the previous case, where $k_m$ was trivial.

\end{proof}

We now have the tools to discuss free products of coset incidence systems. Let $\alpha = (A, (A_i)_{i \in I_\alpha})$ and $ \beta =(B, (B_i)_{i \in I_\beta})$ be two coset incidence systems of finite ranks such that $I_\alpha \cap I_\beta = \emptyset$. 

\begin{definition}[Free Product of $\alpha$ and $\beta$]\label{def:freeProdCosetInc}
    The \textit{free product} of $\alpha$ and $\beta$ is the coset incidence system $\Gamma = \alpha * \beta = (G,(G_i)_{i \in I})$ where $G = A * B$, with $G_i = A_{\{i\} \cap I_\alpha} * B_{ \{i\} \cap I_\beta}$ and $I=I_\alpha \sqcup I_\beta$. 
\end{definition}

We use the convention that $A_\emptyset = A$ and $B_\emptyset = B$, so that $G_i = A_i *B$ or $A *B_i$ depending on whether $i$ is in $I_\alpha$ or $I_\beta$. The parabolical subgroups of $\alpha *\beta$ are very easily understood.

\begin{lemma}
    The parabolic subgroups of the coset incidence system $\Gamma=\alpha*\beta=(G,(G_i)_{i\in I_\alpha})$ are $G_J= A_{J_\alpha}*B_{J_\beta}$, with  $J_\alpha=I_\alpha\cap J$ and $J_\beta = I_\beta\cap J$.
\end{lemma}
\begin{proof}
    Since $G_J=\cap_{i\in J} G_i$, 
    this is a direct consequence of the definition of the maximal parabolic subgroups of $\Gamma$ and Lemma~\ref{lem:freeProductIntersection}(a).
\end{proof}

We now show that this free product of incidence systems preserves natural properties like flag-transitivity, residual connectedness and firmness.
We start with residual connectedness. The main idea is that residues of $\alpha * \beta$ that are ``mixed", meaning that they contain types in both $I_\alpha$ and $I_\beta$, are automatically connected. Therefore, the residual connectedness of $\alpha * \beta$ will only depend on the residual connectedness of $\alpha$ and $\beta$. Recall that rank $1$ incidence systems are trivially residually connected, so we allow the factors $\alpha$ and $\beta$ to be of rank $1$.

\begin{prop}\label{prop:RCfree}
    The free product $\Gamma = \alpha * \beta$ satisfies (RC1) if and only if both $\alpha$ and $\beta$ satisfy (RC1).
\end{prop}

\begin{proof}

    We first show that $\alpha * \beta$ satisfies (RC1) as long as both $\alpha$ and $\beta$ also satisfy (RC1). By Theorem \ref{thm:CosetRC}, we need to show that $G_J = \langle G_{\{i\} \cup J} | i \in I \setminus J \rangle$ for all $J \subset I$ such that $|I \setminus J| \geq 2$, where $G_J = \cap_{j \in J} G_j$. The reverse inclusion is trivial, hence we will prove that $G_J \subseteq \langle G_{\{i\} \cup J} | i \in I \setminus J \rangle$. Let $J_\alpha = I_\alpha \cap J$ and $J_\beta = I_\beta \cap J$.
    
    First of all, let us consider the cases where either $I_\alpha$ or $I_\beta$ is contained in $J$. 
    Suppose first that $I_\alpha \subseteq J$. Then $G_J=A_{I_\alpha}*B_{J_\beta}$, which can be expressed as $G_J=\langle A_{I_\alpha},B_{J_\beta}\rangle$, as it is always true that a free product is generated by its two factors.
    Since $\beta$ is (RC1), we have that $B_{J_\beta}=\langle B_{J_\beta\cup\{i\}}\mid i\in I_\beta\setminus J_\beta\rangle$. Hence $G_J = \langle A_{I_\alpha},B_{J_\beta}\rangle = \langle A_{I_\alpha}, B_{J_\beta\cup\{i\}}\mid i\in I_\beta\setminus J_\beta\rangle$.
    Since $I_\alpha \subseteq J$, we can write the above as 
    $\langle A_{I_\alpha}, B_{J_\beta\cup\{i\}}\mid i\in I\setminus J\rangle$.
    Notice that, for each $i\in I\setminus J$ in this case, we have that $G_{J\cup\{i\}}=\langle A_{I_\alpha}, B_{J_\beta\cup\{i\}}\rangle$.
    Hence, $$G_J=\langle A_{I_\alpha}, B_{J_\beta\cup\{i\}}\mid i\in I\setminus J\rangle \subseteq \langle G_{J\cup\{i\}} \mid i \in I\setminus J\rangle  ,$$
    as wanted.
    The case where $I_\beta \subseteq J$ is equivalent.
    
    Suppose now that $J$ contains neither $I_\alpha$ nor $I_\beta$. Then $G_J = A_{J_\alpha} * B_{J_\beta} = \langle A_{J_\alpha},B_{J_\beta}\rangle$.
    Since $\alpha$ and $\beta$ are (RC1), we have that $A_{J_\alpha} = \langle A_{J_\alpha \cup \{j\}} | j \in I_\alpha \setminus J_\alpha \rangle$ and $B_{J_\beta} = \langle B_{J_\beta \cup \{j\}} | j \in I_\beta \setminus J_\beta \rangle$.
    Consider the group $\langle G_{J\cup\{i\}}\mid i\in I\setminus J\rangle$. Notice that, depending on whether $i\in I_\alpha$ or $i\in I_\beta$, we have that $G_{J\cup\{i\}}= A_{J_\alpha\cup\{i\}}* B_{J_\beta}=\langle A_{J_\alpha\cup\{i\}}, B_{J_\beta} \rangle$ or $G_{J\cup\{i\}}= A_{J_\alpha}* B_{J_\beta\cup \{i\}}=\langle A_{J_\alpha}, B_{J_\beta\cup\{i\}} \rangle$.
    Therefore $$\langle G_{J\cup\{i\}}\mid i\in I\setminus J\rangle = \langle A_{J_\alpha},B_{J_\beta}, A_{J_\alpha\cup \{i\}}, B_{J_\beta\cup \{j\}}\mid i\in I_\alpha\setminus J_\alpha \wedge j\in I_\beta\setminus J_\beta\rangle. $$
    It is direct then that in this case $G_J \subseteq \langle G_{J\cup\{i\}}\mid i\in I\setminus J\rangle$.

Finally, to prove the direct implication, suppose that $\alpha$ is not (RC1). Then we have that $A_{J_\alpha} \neq \langle A_{J_\alpha \cup \{i\}} | i \in I_\alpha \setminus J_\alpha \rangle$ for some choice of $J_\alpha \subset I_\alpha$. But then, $\alpha * \beta$ cannot be (RC1).
Indeed, we have that $G_{J_\alpha \cup I_\beta} = A_{J_\alpha}*B_{I_\beta} = \langle A_{J_\alpha},B_{I_\beta}\rangle$ and $$\langle G_{J_\alpha \cup I_\beta\cup\{i\}} \mid i \in I\setminus J_\alpha\rangle = \langle G_{J_\alpha \cup I_\beta\cup\{i\}} \mid i \in I_\alpha\setminus J_\alpha\rangle = \langle A_{J_\alpha\cup\{i\}}* B_{I_\beta}\mid i \in I_\alpha\setminus J_\alpha\rangle,$$
which can be written as $\langle A_{J_\alpha\cup\{i\}}\mid i \in I_\alpha\setminus J_\alpha\rangle*B_{I_\beta}$.
As $A_{J_\alpha} \neq \langle A_{J_\alpha \cup \{i\}} | i \in I_\alpha \setminus J_\alpha \rangle$, we have that $$G_{J_\alpha \cup I_\beta}=A_{J_\alpha}*B_{I_\beta}\neq \langle A_{J_\alpha\cup\{i\}}\mid i \in I_\alpha\setminus J_\alpha\rangle*B_{I_\beta} = \langle G_{J_\alpha \cup I_\beta\cup\{i\}} \mid i \in I\setminus J_\alpha\rangle.$$

\end{proof}

\begin{prop}\label{prop:FTfree}
        If both $A$ and $B$ act flag-transitively on $\alpha$ and $\beta$, respectively, then the group $G = A * B$ also acts flag-transitively on the free product $\Gamma =\alpha * \beta$. Additionally, if $A_{I_\alpha}=\{1_A\}$ and $B_{I_\beta}=\{1_B\}$, then the converse also holds. 

\end{prop}

\begin{proof}
    We first show that $\Gamma = \alpha*\beta$ is flag-transitive if both $A$ and $B$ act flag-transitively on $\alpha$ and $\beta$, respectively. By Theorem \ref{thm:cosetFT}, we need to show that $G_JG_i = \cap_{j\in J} G_jG_i$ or, equivalently, that $G_iG_J = \cap_{j \in J} G_iG_j$. 
    First, notice that all the $G_J$ are special subgroups of $G$. Also, if $\alpha$ and $\beta$ are flag-transitive, the extra hypothesis of (b) in Lemma~\ref{lem:freeProductIntersection} holds. Therefore, by that same lemma, we have that $\cap_{j \in J} G_iG_j = G_i (\cap_{j \in J} G_j) = G_iG_J$. Hence, the forward implication is proved.

    Suppose now that $G=A*B$ is flag-transitive on the free product $\Gamma=\alpha*\beta$ and, additionally, suppose that $A_{I_\alpha}=\{1_A\}$ and $B_{I_\beta}=\{1_B\}$. 
    We have that, for $J\subseteq I$ such that $I_\beta \subseteq J$, $G_{J}= A_{J_\alpha}*B_{I_\beta} = A_{J_\alpha}$. As $\Gamma$ is flag transitive, by Proposition~\ref{prop:FTequivs}(\ref{prop:FTequivs:itm:FTProductOfIntersections}), for every $J,H,K\subseteq I$ such that $I_\beta\subseteq J,H,K$, we have that $(G_J\cap G_H)(G_J\cap G_K)= G_J\cap (G_H G_K)$, which is equivalent to $(A_{J_\alpha}\cap A_{H_\alpha})(A_{J_\alpha}\cap A_{K_\alpha})= A_{J_\alpha}\cap (A_{H_\alpha} A_{K_\alpha})$. As this covers all possible parabolic subgroups of $\alpha$, this implies that $A$ is flag-transitive in $\alpha$. Under a similar setting, we can also prove that $B$ is flag-transitive in $\beta$. 
\end{proof}

\begin{prop}\label{prop:FIRM_free}
The free product $\alpha * \beta$ is (FIRM) if and only if both $\alpha$ and $\beta$ are (FIRM).
\end{prop}
\begin{proof}
        We will show the reverse implication first. To prove that $\alpha * \beta$ is (FIRM) we need to show that $G_I\neq G_{I\setminus\{i\}}$, for all $i\in I$. Notice that, for each $i\in I$, either $I_\alpha\subseteq I\setminus\{i\}$ or $I_\beta\subseteq I\setminus\{i\}$, meaning that $G_{I\setminus\{i\}}$ is either $A_{I_\alpha\setminus\{i\}}*B_{I_\beta}$ or $A_{I_\alpha} * B_{I_\beta\setminus\{i\}}$. Note also that $G_I = A_{I_\alpha} * B_{I_\beta}$. As both $\alpha$ and $\beta$ are (FIRM), we have that $A_{I_\alpha\setminus\{i\}}\neq A_{I_\alpha}$ and $B_{I_\beta\setminus\{i\}}\neq B_{I_\beta}$. Hence, as $G_{I\setminus\{i\}}$ is either $\langle A_{I_\alpha\setminus\{i\}},B_{I_\beta}\rangle$ or 
        $\langle A_{I_\alpha}, B_{I_\beta\setminus\{i\}}\rangle$, then clearly
        $G_I\neq G_{I\setminus\{i\}}$.
    Suppose now that $\alpha$ is not (FIRM). For some $i\in I_{\alpha}$, we have $A_{I_\alpha\setminus\{i\}} = A_{I_\alpha}$. Hence, for some $i\in I$ we have 
    $G_{I\setminus\{i\}} = A_{I_\alpha\setminus\{i\}}* B_{I_\beta} = A_{I_\alpha} * B_{I_\beta}  = G_I$, hence $\alpha * \beta$ is not (FIRM).
\end{proof}

We can combine the results obtained so far. In particular, this shows that the free product of regular hypertopes is a regular hypertope.

\begin{thm}\label{thm:FP_a)FT_b)Geo_c)_RC_b)_FIRM}
    Let $\alpha=(A,(A_i)_{i\in I_\alpha})$ and $\beta=(B,(B_i)_{i\in I_\beta})$ be two coset incidence systems and suppose that $A$ and $B$ are flag-transitive in $\alpha$ and $\beta$, respectively. Then,
\begin{enumerate}
    \item the group $G = A * B$ is flag-transitive in $\alpha*\beta$;
    \item $\alpha*\beta$ is a coset geometry;
    \item $\alpha*\beta$ is residually connected if and only if both $\alpha$ and $\beta$ are residually connected;
    \item $\alpha*\beta$ is firm if and only if $\alpha$ and $\beta$ are firm;
\end{enumerate}
\end{thm}
\begin{proof}
    $a)$ This is a direct consequence of Proposition~\ref{prop:FTfree};

    $b)$ From item $(a)$, this is a consequence of Theorem~\ref{thm:cosetFT};

    $c)$ Putting together item $(a)$ and Proposition~\ref{prop:RCfree}, with Theorem~\ref{thm:CosetRC}, we get that $\Gamma$ is residually connected if and only if the same is true for $\alpha$ and $\beta$. 

    $d)$ Since $G$ is flag-transitive in $\Gamma$, then by Proposition~\ref{prop:FIRM_free}, item (a) and Theorem~\ref{thm:CosetFIRM}, we have that $\Gamma$ is firm if and only if $\alpha$ and $\beta$ are firm.
    
\end{proof}

\begin{coro}\label{coro:FreeProductHypertopes}
    The free product $\Gamma  = \alpha * \beta$ is a regular hypertope if and only if both $\alpha$ and $\beta$ are regular hypertopes.
\end{coro}
\begin{proof}
    Recall that a regular hypertope is a thin flag-transitive and residually connected geometry. Proposition~\ref{prop:FTfree} and Theorem~\ref{thm:FP_a)FT_b)Geo_c)_RC_b)_FIRM} handle the flag-transitivity, residual connectedness and geometry. The only thing left to take care of is thinness.

    Suppose first that $\alpha$ and $\beta$ are regular hypertopes. Then, by \cite[Lemma~3.4]{hypertopes}, we have that $A_{I_\alpha}= B_{I_\beta}=\{1_G\}$.
    We have previously seen that $G_{I\setminus\{i\}}$ is either $A_{I_\alpha\setminus\{i\}}*B_{I_\beta}$ or $A_{I_\alpha} * B_{I_\beta\setminus\{i\}}$. Therefore, in this case, we get that $G_{I\setminus\{i\}}$ is either $B_{I_\beta\setminus\{i\}}$ or $A_{I_\alpha\setminus\{i\}}$, and $G_I = \{1_G\}$. 
    
    As both $\alpha$ and $\beta$ are flag-transitive, by Proposition~\ref{prop:FTequivs}(\ref{prop:FTequivs:itm:Isomorphism}) we have that the residues of flags of corank $1$ are, respectively, isomorphic to $A_{I_\alpha\setminus\{i\}}/A_{I_\alpha}$ and $B_{I_\beta\setminus\{i\}}/B_{I_\beta}$. Thinness implies than $[A_{I_\alpha\setminus\{i\}}:A_{I_\alpha}]= |A_{I_\alpha\setminus\{i\}}|= 2$ and $[B_{I_\beta\setminus\{i\}} :B_{I_\beta}] =|B_{I_\beta\setminus\{i\}}| = 2$.
    Hence, we can easily conclude that $[G_{I\setminus\{i\}} : G_I ] = |G_{I\setminus\{i\}}| = 2$, and thus that $\Gamma$ is thin.

    Alternatively, if $\Gamma$ is thin, as $\Gamma$ is residually connected and $G$ is flag-transitive in $\Gamma$, we get that $G_I=\{1_G\}$ and $|G_{I\setminus\{i\}}|=2$. Since $G_I = A_{I_\alpha}*B_{I_\beta}$, we get that $A_{I_\alpha}=\{1_G\}$ and $B_{I_\beta}=\{1_G\}$. Moreover, as $G_{I\setminus\{i\}}$ is either $A_{I_\alpha\setminus\{i\}}$ or $B_{I_\beta\setminus\{i\}}$, we get that $|A_{I_\alpha\setminus\{i\}}| = 2$ for $i\in I_\alpha$, and $|B_{I_\beta\setminus\{i\}}| = 2$ for $i\in I_\beta$. Hence both $\alpha$ and $\beta$ are thin.
\end{proof}

We conclude this section with a few examples.

\begin{example}
    Let $\alpha$ be any coset incidence system and let $\beta = (\{1_B\},(\{1_B\}))$ be the trivial coset incidence system. Then $\alpha * \beta \cong \alpha$.
\end{example}
\begin{example}
    Let $\alpha = (C_2, (\{1\})) = \beta$. Then $\alpha * \beta = (C_2 * C_2, (C_2 * \{1_B\}, \{1_A\} * C_2))$ is the geometry of an infinite bivalent tree.
\end{example}
\begin{example}
    Let $\alpha = (D_n, (\langle a_1\rangle, \langle a_2 \rangle))$ be the coset incidence geometry of an $n$-sided polygon and let $\beta =  (C_2, (\{1_B\}))$. Then, we have that both $a_1$ and $a_2$ are involutions such that $a_1a_2$ has order $n$. The geometry $\alpha * \beta = (D_n *C_2, (\langle a_2 \rangle * C_2, \langle a_1 \rangle * C_2, D_n))$ is the geometry of a tessellation of the hyperbolic plane $\mathcal{H}^2$ of type $(n, \infty, \infty)$. Indeed $D_n * C_2 = \langle a_1,a_2,a_3 \rangle$ is the triangle group $\Delta(n,\infty,\infty)$.

\end{example}

\subsection{Free Products with Amalgamation of Coset Incidence Systems}

Let $A$ and $B$ be two groups with subgroups $C_A \leq A$ and $C_B \leq B$ together with an isomorphism $\varphi \colon C_A \to C_B$. Then, we can form the free amalgamated product

$$ 
G = A *_\varphi B = \langle A * B \mid c = \varphi(c), c \in C_A \rangle.
$$

As $A$ and $B$ are always embedded in $G$, we can consider their intersection $A \cap B$ in $G$. By abuse of notation, we will write $C = C_A = C_B = A \cap B$. Then, we can omit the isomorphism $\varphi$ and write $G = A *_C B$. This will ease the notations. 

Any element $g \in A *_C B$ can be obtained by multiplying in an alternating way a sequence of elements of $A$ and elements of $B$. There are different ways to select such sequences, each with its own advantages and disadvantages. We will use the so-called reduced sequences. These are easy to handle and work well enough in our cases.
\begin{definition}(Reduced sequence for free products with amalgamation)

A sequence $c,k_1,k_2, \cdots , k_n, n\geq 0$ of elements of $A*_CB$ will be called \textit{reduced} if 
\begin{enumerate}
    \item $c \in C$,
    \item each $k_i$ is either in $A \setminus C$ or $B \setminus C$,
    \item the sequence is alternating, meaning that successive $k_i, k_{i+1}$ are in different factors $A$ or $B$.
\end{enumerate}
\end{definition}

Note that the elements of $C$ correspond to sequences where $n =0$, and the elements of $A$ or $B$, that are not in $C$, correspond to sequences where $n=1$.
The main result about reduced sequences is the following.

\begin{thm}\cite[Theorem 2.6 -- Chapter IV]{Lyndon2001}
If $c,k_1,\cdots,k_n$ is a reduced sequence with $n \geq 1$, then the product $ck_1 \cdots k_n \neq 1_G$ in $G$. In particular, both $A$ and $B$ are embedded in $G = A *_C B$.
\end{thm}

We now need to prove a few technical results about some subgroups of free amalgamated products.
Following H. Neumann's terminology~\cite{neumann1948generalized}, a subgroup $D$ of $G$ is said to be \textit{special} if it is generated in $G$ by a subgroup $D_A$ of $A$ and a subgroup $D_B$ of $B$, such that $C\cap D_A=C\cap D_B=:C_D$. It follows from \cite[Corollary 8.11]{Neumann1948} that $D \cong D_A *_{C_D} D_B$. Therefore, elements of $D$ also admit reduced sequences.

\begin{lemma}\label{lem:freeIntersectionAmalgam}
    Let $D,E,F$ be three special subgroups of $G = A *_C B$ such that $D = D_A *_{C_D} D_B, E = E_A *_{C_E} E_B$ and $F = F_A *_{C_F} F_B$. Then,
    \begin{enumerate}
        \item $D \cap E = (D_A \cap E_A) *_{C_{D\cap E}} (D_B \cap E_B)$, where $C_{D\cap E}=C_D\cap C_E$;
        \item If, moreover, we have that $D_A E_A \cap D_A F_A = D_A(E_A \cap F_A)$ and $D_B E_B \cap D_B F_B = D_B(E_B \cap F_B)$, then $DE \cap DF = D(E \cap F)$.
    \end{enumerate}
\end{lemma}
\begin{proof}

    (a)  Any element $g \in G$ can be written as $g = c k_1 \cdots k_n$ with $c \in C$ and $k_i$ in $A$ or $B$, in an alternating fashion. This word for $g$ is not unique but all words representing $g$ have a similar structure. Indeed, all words for $g$ can be obtained from $ck_1 \cdots k_n$ by replacing some $k_i$ by any element of the corresponding coset $C k_i$ and then changing the value of $c \in C$ in the word of $g$ accordingly.
    As $D$ is at the same time a subgroup of $G$ and an amalgamated product itself, if $g \in D$, it means that we can choose the $k_i$ so that they are in $D$ and that in that case $c \in C_D$. 
    In other words, we can choose $k_i$ such that $C k_i \cap D \neq \emptyset$ for each $i = 1, \cdots, n$. Similarly, if $g \in E$, it means that we can choose the $k_i$ to be in $E$ and that if we do so, we will have $c \in C_E$. Hence, it is possible to to write $g = c k_1 \cdots k_n$ with $c \in C_{D \cap E}$ and $k_i \in D_A \cap E_A$ or $k_i \in D_B \cap E_B$, alternatively. This shows that $D \cap E \subseteq (D_A \cap E_A) *_{C_{D\cap E}} (D_B \cap E_B)$. The reverse inclusion is trivial.

(b) The inclusion $D(E \cap F) \subseteq DE \cap DF$ always holds. We thus only need to prove the reverse inclusion. Suppose thus that $g \in DE$. Then, $g$ can be represented by a word $w = cd_1 \cdots d_m c'e_{1}e_2 \cdots e_n$ such that each $d_i$ is in $D_A$ or $D_B$ but not in $C_D$, each $e_i$ is in $E_A$ or $E_B$ but not in $C_E$, $c\in C_D$ and $c'\in C_E$. Let $w_D = cd_1 \cdots d_m$ and $w_E = e'_{1}e_2 \cdots e_n$, where $e'_1=c'e_1$. We can assume that both $w_D$ and $w_E$ are reduced words. We can also suppose that $e'_1$ is not in $D_A$ nor $D_B$, so that 
$w_D$ is the longest prefix of $w$ contained in $D$. If this was not the case, just keep adding $e'_1, e_2, \cdots$ to $w_D$ until you reach some element that is not in $D_A$ nor $D_B$ and redefine $w_E$ to be what is leftover at the end of $w$.

Suppose that $w$ is reduced in $G$, meaning that $d_m$ and $e'_1$ are in different factors of $G$ and that $e'_1$ is not in $C$. Since $g \in DF$, we must have that $w$ is a concatenation of a word in $D$ with a word in $F$. We claim that this implies that $w_E$ is in $F$. Indeed, $w_D$ is the longest prefix of $w$ belonging to $D$. Therefore, no matter how we divided $w$ into two parts, the first belonging to $D$ and the second to $F$, the second must contain $w_E$ as a suffix. However, this must imply that $w_E \in F$. Hence, $w = w_D w_E$ with $w_D \in D$ and $w_E \in E \cap F$. Therefore, $g \in D(E \cap F)$.

    Suppose instead that $w$ is not reduced. Then, either both $d_m$ and $e'_1$ are in the same factor of $G$, or $e'_1$ is in $C$. In the first case, let us suppose, without loss of generality, that $d_m \in D_A$, and thus that $e'_1 \in E_A$. Then, a reduced form of $w$ is then $w = d_1 \cdots d_{m-1} x e_{2} \cdots e_n$ where $x = d_m e'_1 \in D_A E_A$. 
    Since $g \in DF$, the exact same argument as in the previous case shows that $e_2 \cdots e_n$ must be in $F$ and $x\in D_AF_A$. Then, we get that $x = d_m e'_1 \in D_A E_A \cap D_A F_A = D_A(E_A \cap F_A)$. Hence, $x = d'_m e''_1$ with $d'_m \in D_A$ and $e''_1 \in E_A \cap F_A$ and $w =  (d_1 \cdots d_{m-1} d'_m) (e''_1 e_{2} \cdots e_n)$ is indeed in $D(E \cap F)$. Suppose instead that $e_1 \in C$. 
    As $e'_1=c'e_1$, that would imply that $e_1\in C$, a contradiction with the fact $w_E$ is reduced. Hence, this can only happen if $w_E=c'$.
    In this case, 
     we have $w = c d_1 \cdots d_m c'$. Since $w \in DF$, it must then be, as in the previous cases, that $c' \in F$. As $c'\in C$, we have that $c' \in C_F$, meaning $c'\in C_E\cap C_F= C_{E\cap F}$.
    Therefore, $g \in D(E\cap F)$.
\end{proof}

We can now turn our attention to coset incidence systems once more. Let $\alpha = (A, (A_i)_{i \in I_\alpha})$ and $ \beta =(B, (B_i)_{i \in I_\beta})$ be two coset incidence systems of finite rank. Let $L = I_\alpha \cap I_\beta$ and suppose that $A_{I_\alpha \setminus L} \cong B_{I_\beta \setminus L}$. This means both $A$ and $B$ share an isomorphic parabolic subgroup  $A_{I_\alpha \setminus L} \cong B_{I_\beta \setminus L}$. We will denote that isomorphic parabolic subgroup by $C$. Let $G = A *_{C} B$, and define $G_i := \langle A_{\{i\} \cap I_\alpha} , B_{ \{i\} \cap I_\beta}\rangle$, for $i\in I = I_\alpha\cup I_\beta$.  

\begin{definition}
    Let $\alpha = (A, (A_i)_{i \in I_\alpha})$ and $ \beta =(B, (B_i)_{i \in I_\beta})$ be two coset incidence systems of finite rank. Let $L = I_\alpha \cap I_\beta$.
    We say $\alpha$ and $\beta$ are \emph{compatible on $L$}
    if, for all $J\subseteq L$, $A_{I_\alpha\setminus (L\setminus J)} \cong B_{I_\beta\setminus (L\setminus J)}$.
\end{definition}

\begin{lemma}\label{lem:equivSpecialSubAmalg}
    The following statements are equivalent:
    \begin{enumerate}
        \item The maximal parabolic subgroups $G_i$, $i \in I$, are special subgroups.
        \item $\alpha$ and $\beta$ are compatible on $L=I_\alpha\cap I_\beta$.
    \end{enumerate}
\end{lemma}
\begin{proof}
    $(a)\Rightarrow (b)$: If $G_i$ are special subgroups, then we have that $G_i = A_{\{i\} \cap I_\alpha} *_{C_{G_i}} B_{ \{i\} \cap I_\beta}$, amalgamating with respect to the subgroup $C_{G_i} = C\cap A_{\{i\} \cap I_\alpha} = C \cap   B_{ \{i\} \cap I_\beta}$.
    For $J\subseteq I$, let $G_J = \cap_{i\in J} G_j$.
    As each $G_j$ is a special subgroup, using Lemma~\ref{lem:freeIntersectionAmalgam}(a), we have that $G_J = (\cap_{j \in J\cap I_\alpha}A_j) *_{C'} (\cap_{j \in J\cap I_\beta}B_j) = A_{J\cap I_\alpha} *_{C'} B_{J\cap I_\beta}$, with amalgamation subgroup $C'= C \cap A_{J\cap I_\alpha} = C \cap B_{J\cap I_\beta} = \cap_{j\in J} C_{G_i}$, which we will denote as $C_{G_J}$.
    Notice that since $C = A_{I_\alpha\setminus L} = B_{I_\beta\setminus L}$, we have that $C_{G_J} = A_{(I_\alpha\setminus L) \cup (J\cap I_\alpha)} = B_{(I_\beta\setminus L) \cup (J\cap I_\beta)}$, which can be rewritten as 
    $C_{G_J} = A_{I_\alpha\setminus (L\setminus J)} = B_{I_\beta\setminus (L\setminus J)}$. This means that for every $J \subseteq  L$, the parabolic subgroups $A_{I_\alpha\setminus (L\setminus J)}$ and $B_{I_\beta\setminus (L\setminus J)}$ are isomorphic. 

    $(b)\Rightarrow (a)$: We have that $G_i$ is the subgroup generated by $\langle A_{\{i\}\cap I_\alpha}, B_{\{i\}\cap I_\beta}\rangle$. If $i\notin L$, then we have $C = C \cap A_{\{i\}\cap I_\alpha} = C \cap B_{\{i\}\cap I_\beta}$, implying that $G_i$ is special. If $i\in L$, then by the hypothesis $(b)$ we have that $A_{I_\alpha\setminus(L\setminus\{i\})}= B_{I_\beta\setminus(L\setminus\{i\})}$. This again implies that $C \cap A_{\{i\}\cap I_\alpha} = C \cap B_{\{i\}\cap I_\beta}$. Hence, in both cases we get that $G_i$ are special subgroups.
\end{proof}

It can evidently happen that $\alpha$ and $\beta$ have an isomorphic parabolical subgroup but are not compatible, as shown in the following example.

\begin{example}
    Let $A$ and $B$ be the groups $$A=\langle a_0,a_1,a_2,a_3\mid a_i^2=(a_0a_1)^4=(a_1a_2)^3=(a_2a_3)^4=(a_0a_2)^2 =$$ 
    $$=(a_0a_3)^2 = (a_1a_3)^2=(a_0a_1a_2)^3=1\rangle$$
    and 
    $$B=\langle b_0,b_1,b_2,b_4\mid b_i^2=(b_0b_1)^3=(b_1b_2)^3=(b_2b_4)^4=(b_0b_2)^2=$$ $$=(b_0b_4)^2=(b_1b_4)^2=1\rangle.$$
   Let $I_\alpha=\{0,1,2,3\}$ and $I_\beta=\{0,1,2,4\}$, and consider $\alpha=(A,(A_i)_{i\in I_\alpha})$, with $A_i=\langle a_j\mid j\in I_{\alpha}\setminus\{i\}\rangle$ and $\beta=(B,(B_i)_{i\in I_\beta})$, with $B_i=\langle b_j\mid j\in I_{\beta}\setminus\{i\}\rangle$.
   The group $B$ is a Coxeter group, hence its parabolic subgroups are $B_J = \langle b_i\mid i\in I_\beta\setminus J\rangle$, for $J\subseteq I_\beta$. The same can be said about $A$, as it is the automorphism group of the regular polytope $\{\{4,3\}_3,\{3,4\}\}$\cite{Hartley2006}, respecting an intersection condition like Proposition~\ref{prop:CoxeterProperties}(b), meaning that $A_J = \langle a_i\mid i\in I_\alpha\setminus J\rangle$.
   It is known that the automorphism group of $A_3\cong B_4 \cong S_4$. Taking $L=I_\alpha\cap I_\beta = \{0,1,2\}$, we have that $A_{I_\alpha\setminus L}= A_3$ and $B_{I_\beta\setminus L}= B_4$. As $A_3\cong B_4$, we have found isomorphic parabolical subgroups in $\alpha$ and $\beta$. However, using Lemma~\ref{lem:equivSpecialSubAmalg}, we can see that $A_{I_\alpha\setminus(L\setminus\{2\})}=A_{2,3}\ncong B_{2,4} = B_{I_\beta\setminus(L\setminus\{2\})}$. Thus, $\alpha$ and $\beta$ are not compatible.
\end{example}

From now on, we will suppose that all the maximal parabolic subgroups $G_i$ are special subgroups. Lemma \ref{lem:equivSpecialSubAmalg} shows that it is a natural assumption to make, since it is equivalent to asking that the parabolic subgroups of both $\alpha$ and $\beta$ match as long as they are inside of $C$. It would not make much sense to attempt to amalgamate both geometries when that does not hold.

\begin{definition}[Free product with amalgamation of $\alpha$ and $\beta$]
    The free product with amalgamation of two compatible coset incidence systems $\alpha = (A, (A_i)_{i \in I_\alpha})$ and $ \beta =(B, (B_i)_{i \in I_\beta})$ is the coset incidence system $\alpha *_L \beta = (G,(G_i)_{i \in I})$ where $G = A *_{C} B$, and $G_i = \langle A_{\{i\} \cap I_\alpha} , B_{ \{i\} \cap I_\beta}\rangle = A_{\{i\} \cap I_\alpha} *_{C_{G_i}} B_{ \{i\} \cap I_\beta}$.
\end{definition}
  In this case, by Lemma~\ref{lem:equivSpecialSubAmalg}, the maximal parabolic subgroup $G_i$ are special subgroups of $G$, with $C_{G_i}=C\cap A_{\{i\} \cap I_\alpha} = C \cap B_{ \{i\} \cap I_\beta}$. From the proof of Lemma~\ref{lem:equivSpecialSubAmalg}, we know that in this case, for $J\subseteq I$, we have 
$G_J = A_{J\cap I_\alpha} *_{C_{G_J}} B_{J\cap I_\beta}$, with $C_{G_J}= C \cap A_{J\cap I_\alpha} = C \cap B_{J\cap I_\beta}$.
Let $L_\alpha = I_\alpha \setminus L$ and $L_\beta = I_\beta \setminus L$. Then $G_{L_\alpha} = B$ and $G_{L_\beta} = A$. Indeed, we have that $G_{L_\alpha} = A_{(I_\alpha\setminus L)\cap I_\alpha} *_{C_{G_{L_\alpha}}} B_{(I_\alpha\setminus L)\cap I_\beta}$. Since $I_\alpha\cap I_\beta=L$, $A_{(I_\alpha\setminus L)\cap I_\alpha}= A_{I_\alpha\setminus L}=C$ and $C_{G_{L_\alpha}}= C\cap A_{I_\alpha\setminus L}=C$, we also have that  $G_{L_\alpha} = C *_{C} B_{\emptyset} = B$. The same holds for $G_{L_\beta} = A$. Similarly, we obtain that, for any $J \subseteq I$ such that $L_\alpha \subseteq J$, we have that $G_J = B_{J \cap I_\beta}$. Also, for any $J \subseteq I$ such that $L_\beta \subseteq J$, we have that $G_J = A_{J \cap I_\alpha}$. Therefore, the parabolic subgroups of $\alpha$ and $\beta$ are also parabolic subgroups of $\alpha *_L \beta$.

We now investigate the relation between residual connectedness, flag-transitivity and firmness of the coset geometries $\alpha$, $\beta$ and $\alpha *_L \beta$. All cases here assume the compatibility of $\alpha$ and $\beta$ as we have only defined this incidence system in those terms.
The following proofs work in essentially the same way as the ones in the case of free products.

\begin{prop}\label{prop:RC_FPAmalg}
    The free amalgamated product $\alpha *_L \beta$ is (RC1) if and only if both $\alpha$ and $\beta$ are (RC1).
\end{prop}

\begin{proof}
    We first show that $\alpha *_L \beta$ is (RC1) as long as both $\alpha$ and $\beta$ are (RC1). In other words, we need to show that $G_J = \langle G_{\{i\} \cup J} | i \in I \setminus J \rangle$ for all $J \subset I$ such that $|I \setminus J| \geq 2$, where $G_J = \cap_{j \in J} G_j$.
    Notice that $\langle G_{\{i\} \cup J} | i \in I \setminus J \rangle \subseteq G_J$ is trivial. We will then prove the reverse inclusion.
    First of all, if $L_\alpha$ or $L_\beta$ is contained in $J$, then $G_J = \langle G_{\{i\} \cup J} | i \in I \setminus J \rangle$ holds since both $\alpha$ and $\beta$ are (RC1) and that the parabolics subgroups under considerations are parabolic subgroups of either $\alpha$ or $\beta$. Suppose thus that $J$ contains neither $L_\alpha$ nor $L_\beta$. 
    Then $G_J = A_{J_\alpha} *_{C_{G_J}} B_{J_\beta}$ with $J_\alpha = I_\alpha \cap J$ and $J_\beta = I_\beta \cap J$. It is always true that a free amalgamated product is generated by its two factors. Hence, we have that $G_J = \langle A_{J_\alpha} , B_{J_\beta} \rangle$. Since $\alpha$ and $\beta$ are (RC1), we have that $A_{J_\alpha} = \langle A_{J_\alpha \cup \{j\}} | j \in I_\alpha \setminus J_\alpha \rangle$ and $B_{J_\beta} = \langle B_{J_\beta \cup \{j\}} | j \in I_\beta \setminus J_\beta \rangle$. But it is clear that $\langle A_{J_\alpha \cup \{j\}} | j \in I_\alpha \setminus J_\alpha \rangle < \langle  G_{\{i\} \cup J} | i \in I \setminus J \rangle$ and $\langle B_{J_\beta \cup \{j\}} | j \in I_\beta \setminus J_\beta \rangle < \langle  G_{\{i\} \cup J} | i \in I \setminus J \rangle.$ This concludes the proof of the forward implication.

    Suppose that $\alpha$ is not (RC1). Then $\alpha *_L \beta$ cannot be (RC1) since some parabolic subgroups of $\alpha$ must fail the (RC1) condition, and these parabolic subgroups are also parabolic subgroups of $\alpha *_L \beta$.
    
\end{proof}

\begin{prop}\label{prop:FT_FPAmalg}
        The group $G$ is flag-transitive on the free amalgamated product $\alpha *_L \beta$ if and only if both $A$ and $B$ are, respectively, flag-transitive in $\alpha$ and $\beta$.
\end{prop}

\begin{proof}
    We first show the reverse implication. By Theorem \ref{thm:cosetFT}, we need to show that, for each $J\subseteq I$ and each $i\in I\setminus J$, $G_JG_i = \cap_{j\in J} G_jG_i$ or, equivalently, that $G_iG_J = \cap_{j \in J} G_iG_j$.  
    First, notice that all the $G_J$ are also special subgroups of $G$. Also, since $A$ is flag-transitive in $\alpha$ and $B$ is flag-transitive in $\beta$, then the hypothesis of Lemma~\ref{lem:freeIntersectionAmalgam}(b) hold. Therefore, by that same lemma, we have that $\cap_{j \in J} G_iG_j = G_i (\cap_{j \in J} G_j) = G_iG_J$. Hence, the reverse implication is proved.

    Suppose now that $A$ is not flag-transitive in  $\alpha$. 
    As stated previously, we have that for $J\subseteq I$ such that $L_\beta\subseteq J$, $G_J = A_{J_\alpha}$ with $J_\alpha=J\cap I_\alpha$. Notice that we write all parabolics of $\alpha$ like this. Hence, this directly implies that $G$ cannot be flag-transitive in $\alpha *_L \beta$, as the parabolic subgroups of $\alpha$ fail the condition of Theorem \ref{thm:cosetFT}, and these parabolic subgroups are also parabolic subgroups of $\alpha *_L \beta$.
\end{proof}

\begin{prop}\label{prop:FIRM_FPAmalg}
The free amalgamated product $\alpha *_L \beta$ is (FIRM) if and only if both $\alpha$ and $\beta$ are (FIRM).
\end{prop}
\begin{proof}
    We will show the reverse implication first. To prove that $\alpha *_L \beta$ is (FIRM) we need to show that $G_I\neq G_{I\setminus\{i\}}$, for all $i\in I$. Notice that, for each $i\in I$, either $L_\alpha\subseteq I\setminus\{i\}$ or $L_\beta\subseteq I\setminus\{i\}$, meaning that $G_{I\setminus\{i\}}$ is either $A_{I_\alpha\setminus\{i\}}$ or $B_{I_\beta\setminus\{i\}}$. Note also that, by Lemma~\ref{lem:equivSpecialSubAmalg}(b), $G_I = A_{I_\alpha} = B_{I_\beta}$. As both $\alpha$ and $\beta$ are (FIRM), we have that $A_{I_\alpha\setminus\{i\}}\neq A_{I_\alpha}$ and $B_{I_\beta\setminus\{i\}}\neq B_{I_\beta}$. Hence $G_I\neq G_{I\setminus\{i\}}$.
    Suppose now that $\alpha$ is not (FIRM). For $i\in I\setminus L_{\beta}$, we have $G_{I\setminus\{i\}} = A_{I_\alpha\setminus\{i\}} = A_{I_\alpha} = G_I$, hence $\alpha *_L \beta$ is not (FIRM).
\end{proof}

\begin{thm}\label{thm:Amalg_a)FT_b)Geo_c)_RC_b)_FIRM}
    Let $\alpha=(A,(A_i)_{i\in I_\alpha})$ and $\beta=(B,(B_i)_{i\in I_\beta})$ be two coset incidence systems compatible on $L=I_\alpha\cap I_\beta$.     
    The group $G = A *_C B$ is flag-transitive in $\alpha*_L\beta$, with $C=A_{I_\alpha \setminus L}\cong B_{I_\beta\setminus L}$, if and only if $A$ and $B$ are flag-transitive in $\alpha$ and $\beta$, respectively. Additionally, if $G = A *_C B$ is flag-transitive in $\alpha*_L\beta$, then
\begin{enumerate}
    \item $\alpha*_L\beta$ is a coset geometry;
    \item $\alpha*_L\beta$ is residually connected if and only if both $\alpha$ and $\beta$ are residually connected;
    \item $\alpha*_L\beta$ is firm (resp. thin/thick) if and only if $\alpha$ and $\beta$ are firm (resp. thin/thick);
\end{enumerate}
\end{thm}

\begin{proof}
    The first part of the theorem is a direct consequence of Proposition~\ref{prop:FT_FPAmalg}.
    
    Suppose now that $G = A *_C B$ is flag-transitive in $\alpha*_L\beta$.
    
    $(a)$ This is a direct consequence of Theorem~\ref{thm:cosetFT}.

    $(b)$ From Proposition~\ref{prop:RC_FPAmalg}, Proposition~\ref{prop:FT_FPAmalg} and Theorem~\ref{thm:CosetRC}, we get that $\Gamma$ is residually connected if and only if the same is true for $\alpha$ and $\beta$. 

    $(c)$ Since $G$ is flag-transitive in $\Gamma$, then by Proposition~\ref{prop:FIRM_FPAmalg}, Proposition~\ref{prop:FT_FPAmalg} and Theorem~\ref{thm:CosetFIRM}, we have that $\Gamma$ is firm if and only if $\alpha$ and $\beta$ are firm.
    Additionally, suppose now that $\alpha$ and $\beta$ are also thin (resp. thick). Hence, both are firm and, by Proposition~\ref{prop:FTequivs}(\ref{prop:FTequivs:itm:Isomorphism}), we know that $|A_{I_\alpha\setminus\{i\}}:A_{I_\alpha}|=2$ (resp. $\geq 3$) and $|B_{I_\beta\setminus\{j\}}:B_{I_\beta}|=2$ (resp. $\geq 3$), for each $i\in I_\alpha$ and $j\in I_\beta$.
    Therefore, as in the proof of Proposition~\ref{prop:FIRM_FPAmalg}, we have either $G_{I\setminus\{i\}} =B_{I_\beta\setminus\{i\}}$ or $G_{I\setminus\{i\}} = A_{I_\alpha\setminus\{i\}}$,
    and $G_{I}=A_{I_\alpha}=B_{I_\beta}$.
    Hence, for all cases, $|G_{I\setminus\{i\}}:G_{I}|=2$ (resp. $\geq 3$), i.e. $\alpha *_L \beta$ is thin (resp. thick). 
    Suppose now that $\alpha$ is not thin (resp. not thick). Then, by the definition of the parabolics, we can easily see that, for some $i\in I$, we have $|G_{I\setminus\{i\}}:G_{I}|\neq 2$ (resp. $|G_{I\setminus\{i\}}:G_{I}|\ngeq 3$). Hence $\alpha *_L \beta$ is not thin (resp. not thick).
\end{proof}

We now once again investigate the case of regular hypertopes.

\begin{coro}
    $\alpha *_L \beta$ is a regular hypertope if and only if $\alpha$ and $\beta$ are regular hypertopes.
\end{coro}

\begin{proof}
    The proof follows directly from Theorem~\ref{thm:Amalg_a)FT_b)Geo_c)_RC_b)_FIRM}
\end{proof}

\begin{example}
    Taking $L = \emptyset$, we get that $\alpha *_L \beta$ is obtained from amalgamating $\alpha$ and $\beta$ along their isomorphic Borel subgroup $A_{I_\alpha} \cong B_{I_\beta}$. In the case where the common Borel subgroup is trivial, as it  always happen for flag-transitive thin geometries, for example, we get that $ \alpha *_L \beta = \alpha * \beta$, the free product of $\alpha$ and $\beta$.
    On the other hand, taking $\alpha = \beta$ so that $L = I_\alpha = I_\beta$, we get that $\alpha *_L \beta = \alpha = \beta$.
\end{example}

\begin{example}
    Let $\alpha = (D_n, (\langle a_1\rangle, \langle a_2 \rangle))$ be the coset incidence geometry of an $n$-sided polygon over the type set $I_\alpha = \{1,2\}$ and let $\beta =  (D_m, (\langle b_2\rangle, \langle b_3 \rangle))$ be the coset incidence geometry of an $m$-sided polygon over the type set $I_\beta = \{2,3\}$. We then have that $L = \{2\}$ and $\langle a_2 \rangle \cong \langle b_2 \rangle$ since both $a_2$ and $b_2$ are involutions. The geometry $\alpha *_L \beta = (D_n *_{\langle a_2 \rangle = \langle b_2 \rangle} D_m, (D_n, \langle a_1 \rangle * \langle b_3 \rangle , D_m))$ is the geometry of a tessellation of the hyperbolic plane $\mathcal{H}^2$ of type $(n, \infty, m)$. Indeed, $D_n *_{C_2} D_m = \langle a_1,a_2,b_3 \rangle$ is the triangle group $\Delta(n,m,\infty)$.
\end{example}

\subsection{HNN-Extensions of Coset Incidence Systems}

Let $A$ be a group and $\varphi \colon A_1 \to A_2$ be an isomorphism between two subgroups of $A$. We can then form the HNN-extension $G = A *_\varphi$. Recall that if $A = \langle S \mid R \rangle$, then the group $G = A *_\varphi$ admits $\langle S,t \mid R, t^{-1}at = \varphi(a), \forall a \in A_1 \rangle$ as a presentation. 

Once again, we begin by some preliminary results on intersections and products of some special subgroups of HNN-extensions.
A subgroup $H$ of $G$ is \textit{special} if $H = \langle B, t \rangle$ for some subgroup $B$ of $A$ such that $\varphi(B\cap A_1)=B\cap A_2$. In this case, we can consider the restriction $\varphi_B \colon B \cap A_1 \to B \cap A_2$ of $\varphi$ to $B$. This restriction $\varphi_B$ is an isomorphism, and $H \cong B *_{\varphi_B}$.

We recall here the concept of reduced sequences for HNN-extensions.

\begin{definition}[Reduced sequences and reduced words for HNN-extensions]

    Let $G=A*_\varphi$ be the HNN-extension of $A$, where $\varphi: A_1\to A_2$ is an isomorphism, with $A_1,A_2\leq A$.
    Then, for $n\geq 0$, every element of $G$ can be represented by a \textit{reduced word}  $w=a_0t^{\epsilon_1}a_1t^{\epsilon_2}a_2 t^{\epsilon_3}\ldots t^{\epsilon_n}a_n$, where $\epsilon_i=\pm 1$ and
    \begin{enumerate}
        \item $a_i$ is an arbitrary element of $A$;
        \item for $a_i\in A_1$, there is no subword $t^{-1}a_it$ in $w$; and
        \item for $a_i \in A_2$, there is no subword $t a_it^{-1}$ in $w$.
    \end{enumerate}
    A sequence $a_0, t^{\epsilon_1},a_1,t^{\epsilon_2},a_2, t^{\epsilon_3},\ldots, t^{\epsilon_n},a_n$ satisfying the above conditions is called a \textit{reduced sequence}.
\end{definition}

\begin{lemma}\cite[Britton's Lemma -- Chapter IV]{Lyndon2001}\label{lem:BrittonLemma}
    If the sequence $a_0$, $t^{\epsilon_1}$, $a_1$, $t^{\epsilon_2}$, $a_2$, $t^{\epsilon_3},\ldots, t^{\epsilon_n},a_n$ is reduced and $n\geq 1$, then we have that $a_0t^{\epsilon_1}$$a_1t^{\epsilon_2}a_2 t^{\epsilon_3}$$\ldots t^{\epsilon_n}$$a_n\neq 1_G$ in $G$.
\end{lemma}

We now explore how special subgroups of $G$ behave under intersections and products.

\begin{lemma}\label{lem:HNNIntersection}
    Let $C,D,E$ be three special subgroups of $G$ such that $C=C_A*_{\varphi_{C_A}}$, $D=D_A*_{\varphi_{D_A}}$ and $E=E_A*_{\varphi_{E_A}}$, for $C_A,D_A,E_A\leq A$. Then,
    \begin{enumerate}
        \item $C \cap D = (C_A \cap D_A)*_{\varphi_{C_A\cap D_A}}$, and it is a special subgroup, meaning in particular that $\varphi_{C_A\cap D_A}: C_A\cap D_A\cap A_1 \to C_A \cap D_A\cap A_2$ is an isomorphism.
        \item If, moreover, we have that $C_AD_A \cap C_AE_A = C_A(D_A \cap E_A)$, then $C D \cap CE= C(D \cap E)$.
    \end{enumerate}
\end{lemma}

\begin{proof}
    $a)$ Any element $g \in G$ can be written as a reduced word $g=a_0t^{\epsilon_1}a_1t^{\epsilon_2}a_2 t^{\epsilon_3}\ldots t^{\epsilon_n}a_n$
    with $a_i\in A$, $\epsilon_i=\pm1$ and no subwords $t^{-1}a_it$, if $a_i\in A_1$, and $ta_it^{-1}$, if $a_i\in A_2$.
    Notice that, for $i\geq 1$ we can always replace each $a_i$  by an element of $A_1a_i$, if $\epsilon_i=-1$, or of $A_2a_i$, if $\epsilon_i=+1$, as long as we choose $a_0$ at the end so that we indeed are equal to $g$ and not another element of $At^{\epsilon_1}a_1t^{\epsilon_2}a_2 t^{\epsilon_3}\ldots t^{\epsilon_n}a_n$.
Since $C$ is a subgroup of $G$ and, at the same time, an HNN-extension, if $g\in C$, then we can choose all $a_i\in C_A$.
In other words, for $i \geq 1$, we have that, when $\epsilon_i=-1$, $A_1 a_i \cap C_A \neq \emptyset$, and when $\epsilon_i=+1$, $A_2 a_i \cap C_A \neq \emptyset$.
Similarly, if $g \in D$, it means that we can choose all $a_i$ to be in $D_A$. Hence, it is possible to write $g=a_0t^{\epsilon_1}a_1t^{\epsilon_2}a_2 t^{\epsilon_3}\ldots t^{\epsilon_n}a_n$ with $a_i \in C_A\cap D_A$.
This shows that $C \cap D \subseteq \langle C_A\cap D_A,t\rangle$. The reverse inclusion is trivial.

     Finally, notice that, as both $C$ and $D$ are special subgroups, and $\varphi$ is bijective, we have that 
     
     $$\varphi(A_1\cap C_A\cap D_A) = \varphi(A_1\cap C_A\cap A_1\cap D_A) = \varphi(A_1\cap C_A)\cap 
     \varphi(A_1\cap D_A)=A_2\cap C_A\cap D_A.$$
     
     Hence, we verified that $C\cap D = \langle C_A\cap D_A,t\rangle = (C_A\cap D_A)*_{\varphi_{C_A\cap D_A}}$ is a special subgroup, with $\varphi_{C_A\cap D_A}$ being the isomorphism equal to the action of $\varphi$, sending $A_1\cap C_A\cap D_A$ to $A_2\cap C_A\cap D_A$. 

    $b)$ The inclusion $C(D\cap E)\subseteq CD\cap CE$ always holds. Consider $w\in CD\cap CE$. Then, as $w\in CD$, we can write it as $w = cd$ where $c =c_0 t^{\epsilon_1} c_1 t^{\epsilon_2} c_2\ldots c_{n-1}t^{\epsilon_n}c_n$ and $d=d_0 t^{\epsilon'_1} d_1 t^{\epsilon'_2} d_2\ldots d_{m-1}t^{\epsilon'_m}d_m$, with $c_i\in C_A$, $d_i\in D_A$ and $\epsilon_i,\epsilon'_i\in\{-1,+1\}$. We will assume that both $c$ and $d$ are reduced, so that there is no subword in $c$ such that $t^{-1}c_it$ with $c_i\in A_1\cap C_A$, or $tc_it^{-1}$ with $c_i\in A_2\cap C_A$, and there is no subword in $d$ such that $t^{-1}d_it$ with $d_i\in A_1\cap D_A$, or $td_it^{-1}$ with $d_i\in A_2\cap D_A$. Additionally, we will assume that $c$ is the longest prefix of $w$ that belongs to $C$ (i.e. $d_0\notin C_A$), as otherwise we could extend the word $c$ until we reach a $d_i$ of $d$ that does not belong to $C_A$.

    Suppose that $w$ is reduced in $G$. Then we have that $t^{\epsilon_n} c_n d_0 t^{\epsilon'_1}$ is neither in $A_1$ nor in $A_2$. 
    Since $w \in DF$, we must have that $w$ is a concatenation of a word in $c'\in C$ with a word in $e \in E$, so that $w=c'e$. We claim that this implies that $d$ is in $E$. Indeed, $c$ is the longest prefix of $w$ belonging to $C$. Therefore, no matter how we divided $w$ into two parts, the first belonging to $C$ and the second to $E$, the second must contain $d$ as a suffix. This implies that $d \in E$. Hence, $w = cd$ with $c \in C$ and $d \in D \cap E$. Therefore, $w \in C(D \cap E)$.
    
    Suppose now that $w$ is not reduced in $G$. Hence, we have that $t^{\epsilon_n} c_n d_0 t^{\epsilon'_1}$ is in $A_1$ or in $A_2$. 
    Considering $x = t^{-1} c_n d_0 t$, $w$ contains the subword $t^{\epsilon_{n-1}}c_{n-1} x d_1 t^{\epsilon_2'}$, which might or might not be in $A_1$ or $A_2$, so that $w$ might still not be reduced.
    Therefore, for some $i\leq \min(n,m)$, let $y$ be the biggest subword of $w$ of the form 
    $$y = c_{n-i}t^{\epsilon_{n-i+1}}c_{n-i+1}\ldots t^{\epsilon_n} c_n d_0 t^{\epsilon'_1}d_1\ldots d_{i-1}t^{\epsilon'_i}d_i$$
    such that $ w' = c_0t^{\epsilon_1}c_1\ldots t^{\epsilon_{n-i}}yt^{\epsilon'_{i+1}} d_{i+1}\ldots t^{\epsilon'_m} d_m$ is a reduced word of $w$.
    Note that, as $w'$ is reduced, we have both that $y\in A$ and $y\in CD$. We will now prove that $A\cap CD = C_A D_A$. If $y\in A$, it can be represented as a reduced word of length at most $1$, so that $y=a_0$, with $a_0\in A$. If $y\in CD$, then the only way for this word to be reduced is if $y = c_0 d_0$, with $c_0\in C_A$ and $d_0\in D_A$. Hence $y\in A\cap CD\subseteq C_AD_A$. The reverse inclusion is trivial. Hence $y \in C_AD_A$.
    
    Since $w'=c_0t^{\epsilon_1}c_1\ldots t^{\epsilon_{n-i}}yt^{\epsilon'_{i+1}} d_{i+1}\ldots t^{\epsilon'_m} d_m \in CE$, and the element $c$ is the longest prefix of $w$ in $C$, we have, as in the reduced case, that $t^{\epsilon'_{i+1}} d_{i+1}\ldots t^{\epsilon'_m} d_m\in E$. Hence, this means that $y\in CE$. As we also have that $y\in A$, this means that $y\in C_AE_A$. Therefore, as $y\in C_AD_A\cap C_AE_A$, we have, by hypothesis, that $y\in C_A(D_A\cap E_A)$.
    As $c_0t^{\epsilon_1}c_1\ldots t^{\epsilon_{n-i}}\in C$, $y\in C_A(D_A\cap E_A)$ and $t^{\epsilon'_{i+1}} d_{i+1}\ldots t^{\epsilon'_m} d_m\in D\cap E$, then
    $w' = w \in C(D\cap E)$.
\end{proof}

We now construct the HNN-extension of a coset incidence system. Let $\alpha = (A, (A_i)_{i \in I})$ be a coset incidence system. 

\begin{definition}
    Let $C,D\leq A$ such that $\varphi: C\to D$ is an isomorphism. We say that $\varphi$ is \emph{admissible} if 
    \begin{enumerate}
        \item $C$ and $D$ are parabolic subgroups of $\alpha$, so that $C=A_J$ and $D=A_K$, for $J,K\subseteq I$;
        \item The isomorphism $\varphi$ induces an isomorphism between the lattice of parabolic subgroups contained in $C$ and the lattice of parabolical subgroups contained in $D$. More precisely, this means that $\varphi$ sends each parabolic subgroup of $\alpha$ contained in $C$ to a unique parabolic subgroup of $\beta$ contained in $D$ in a way that preserves intersections.
    \end{enumerate} 
\end{definition}

Let $\varphi: C\to D$ be an admissible isomorphism between two parabolics of the coset incidence system $\alpha$, with $C=A_J$ and $D=A_K$, for some $J,K\subseteq I$. We define an equivalence relation on $I$ under the action of $\varphi$ as follows: 
\begin{enumerate}
    \item for each $i \in I$, we have that $i \sim_\varphi i$,
    \item for $i_1\in I\setminus J$ and $i_2\in I\setminus K$, we say $i_1 \sim_\varphi i_2$ if $\varphi(A_{J\cup\{i_1\}}) = A_{K\cup\{i_2\}}$ or $\varphi^{-1}(A_{K\cup\{i_2\}}) = A_{J\cup\{i_1\}}$.
\end{enumerate}

Let $\widetilde{I}$ be the set of equivalence classes of $I$ under the action of $\varphi$. 
The set $\widetilde{I}$ is thus a set of disjoint subsets of $I$. A few technicalities and cumbersome notations arise from this fact. We will often be quite relaxed regarding these, whenever we believe no possible confusion can arise. Nonetheless, we introduce some precise notations to be used when clarity is needed.

Let $S\subseteq \widetilde{I}$. We have that $A_S = \cap_{s\in S} A_s$, where each $s\in S$ is a subset of $I$. Hence we have that $A_S = \cap_{s\in S} A_s = \cap_{s\in S}(\cap_{i\in s} A_i)$. Consider the following notation for the union of sets inside a set, i.e. $\widehat{S}=\cup_{s\in S} s$. This is simply the set of elements that appear in each subset of $S$. Notice that $\widehat{S}\subseteq I$ and that $\widehat{\widetilde{I}}=I$.
We thus have that $A_S = A_{\widehat{S}}$.
Similarly, $A^S = \langle A^s\mid s\in S\rangle = \langle A^i \mid s\in S\wedge i\in s\rangle = A^{\widehat{S}}$. In many cases, we will treat $S$ as $\widehat{S}$ and vice-versa, when it eases notations without too much possible confusion.

We now define a coset incidence system of the HNN-extension of $\alpha =(A,(A_i)_{i\in I})$ by $\varphi$.
\begin{definition}[HNN-extension of $\alpha$]
    The HNN-extension of $\alpha$ by a $\varphi$-admissible isomorphism between two parabolics of $\alpha$ is the coset incidence system $\alpha*_\varphi = (G,(G_s)_{s\in \II})$, where $G = A*_\varphi$, $\II = \widetilde{I}\cup \{t\}$, such that, for each equivalence class $s=\{i_1,i_2,\ldots,i_k\}\in \widetilde{I}$, $G_s=\langle t, A_s\rangle$, and $G_t=A_F$, where $F=(I\setminus K)\cap J = J\setminus K$.
\end{definition}

Note that $G_s=\langle t, A_s\rangle$ is the subgroup of $G = A *_\varphi$ generated by $t$ and by the parabolic subgroup  $A_s=\cap_{i\in s} A_i$. 
Therefore, all parabolic subgroups $G_s$, with $s\in \widetilde{I}$, contain the cyclic subgroup $\langle t \rangle$ of $G= A *_\varphi$ and thus all the parabolic subgroups $G_{\widetilde{I}}$ of $\alpha *_\varphi$ also contains $\langle t \rangle$.

\begin{lemma}
If the isomorphism $\varphi$ is admissible, then the maximal parabolic subgroups $G_s$ are special subgroups of $A *_\varphi$(for $s\in \widetilde{I}$).
\end{lemma}
\begin{proof}
Suppose $\varphi$ is admissible. Then $C=A_J$ and $D=A_K$, for $J,K\subseteq I$, and $\varphi$ sends parabolic subgroups of $C(=A_J)$ to parabolic subgroups of $D(=A_K)$.

Let $s\in \widetilde{I}$ be an equivalence class of $I$ under the action of $\varphi$ and let $G_s=\langle t,A_s\rangle$ be the corresponding maximal parabolical subgroup.
To prove that this group is a special subgroup, we have to show that $\varphi(A_s\cap C) = A_s\cap D$.
As $C=A_J$ and $D=A_K$, this turns out to be equivalent to $\varphi(A_s\cap A_J)=A_s\cap A_K$. Let us define $s_J = s \cap (I\setminus J)$ and $s_K = s\cap (I\setminus K)$.
Notice that for each $i_1\in s_J$ there exists $i_2\in s_K$ such that $i_1\sim_\varphi i_2$.
Hence $A_s\cap A_J = A_{J\cup s_J}$ and $A_s\cap A_K=A_{K\cup s_K}$. Moreover, $A_{J\cup s_J} = A_{J\cup \{i_1\}}\cap A_{J\cup \{i_2\}}\cap \ldots\cap A_{J\cup \{i_p\}}$, for $s_J=\{i_1,i_2,\ldots,i_p\}$ and $A_{K\cup s_K} = A_{K\cup \{j_1\}}\cap A_{K\cup \{j_2\}}\cap \ldots\cap A_{J\cup \{j_l\}}$, for $s_K=\{j_1,j_2,\ldots,j_l\}$.
Also, we know that for each $i\in s_J$, there is a $j\in s_K$ such that $i\sim_\varphi j$, so that $\varphi(A_{J\cup\{i\}})=A_{K\cup \{j\}}$. Hence, $|s_J|=|s_K|$ and $p=l$. Since $\varphi$ is bijective, we then get that
$$\varphi(A_s\cap A_J) = \varphi(A_{J\cup \{i_1\}}\cap A_{J\cup \{i_2\}}\cap \ldots\cap A_{J\cup \{i_p\}}) = \varphi(A_{J\cup \{i_1\}})\cap \varphi(A_{J\cup \{i_2\}})\cap \ldots \cap \varphi(A_{J\cup \{i_p\}}) = $$ 
$$= A_{K\cup \{j_1\}}\cap A_{K\cup \{j_2\}}\cap \ldots\cap A_{K\cup \{j_p\}} = A_s\cap A_K,$$
as desired.
Hence, $G_s=A_s *_{\varphi_{A_s}}$ is a special subgroup, where $\varphi_{A_s}$ is an isomorphism $\varphi_{A_s}: A_s\cap C\to A_s\cap D$.
\end{proof}

As long as $\varphi$ is admissible, we can thus use Lemma \ref{lem:HNNIntersection} to study the maximal parabolic subgroups $G_s$, for $s\in \widetilde{I}$, are all special subgroups. The case of $G_t = A_F$, can be handled separately without much trouble. Indeed, for any $s\in \widetilde{I}$, $G_t\cap G_s = A_F\cap G_s = A_{F\cup s}$, since $A_F$ is a subgroup of $A$ and $G_s$ is a special subgroup of $G$ with $A_{s}\leq G_s$. In general, for $S\subseteq \II$ such that $t\in S$, we have that $G_{S}=G_{S_{\widetilde{I}}}\cap G_t = A_{F\cup S_{\widetilde{I}}} = A_{F\cup \widehat{S_{\widetilde{I}}}}$, where $S_{\widetilde{I}} = S\cap \widetilde{I}$.
Since for each $s\in \widetilde{I} $, we have that $t\in \II\setminus \{s\}$, we have that the minimal parabolic subgroups of $\alpha *_\varphi$ are $G^s = \cap_{i \in (\II\setminus \{s\})} G_i = \cap_{i\in (\widetilde{I}\setminus \{s\})} A_{F\cup i} = A_{F\cup (\widetilde{I}\setminus \{s\})} = A_{F\cup (I\setminus s)}=A_{I\setminus (s\setminus F)}$. 

We are now ready to investigate the usual properties of $\alpha *_\varphi$.

\begin{prop}\label{prop:HNNRC1}
    If $\alpha$ is (RC1) then  the HNN-extension $\alpha *_\varphi$ is (RC1).
\end{prop}
\begin{proof}
    Since $\alpha$ is (RC1), by Proposition~\ref{lem:RCequivs}(\ref{lem:RCequivs:itm:BottomUp}--\ref{lem:RCequivs:itm:UpBottom}), for any $L\subseteq I$, we have that $A_L = A^{I\setminus L}$ and $A^L = A_{I\setminus L}$.
    
    We need to show that a similar statement holds for $\alpha *_\varphi$. Let $S\subseteq \II$ be any subset of $\II$. 
    First, suppose that $t\notin S$. 
    Hence, $S\subseteq \widetilde{I}$ and therefore $S = \{s_1,s_2,\ldots,s_k\}$, with each $s_i$ being a equivalence class. By the discussion preceding this proof, we then have that $G^S = \langle G^{s_1}, G^{s_2},\ldots,G^{s_k}\rangle = \langle A_{I\setminus (s_1\setminus F)}, A_{I\setminus (s_2\setminus F)},\ldots,A_{I\setminus (s_k\setminus F)}\rangle$. As $\alpha$ is (RC1), we have that $G^S= \langle A^{s_1\setminus F}, A^{s_2\setminus F}, \ldots, A^{s_k\setminus F}\rangle = A^{\widehat{S}\setminus F}$. 
    Now, as $t\notin S$, it must be that $t\in \II\setminus S$, so that $$G_{\II\setminus S}=\cap_{i\in \II\setminus S} G_i = \cap_{i\in \widetilde{I}\setminus S} A_{F\cup i} =A_{I\setminus(\widehat{S}\setminus F)}.$$
    Finally, since $A^{\widehat{S}\setminus F} = A_{I\setminus (\widehat{S}\setminus F)}$, we get that $G^S = G_{\II\setminus S}$, as desired.

    Suppose now that $t\in S$. Then $S=\{s_1,\ldots, s_k, t\}$, with $s_i\in \widetilde{I}$. First, notice that $G^t = \cap_{s\in \widetilde{I}}G_s$. By Lemma~\ref{lem:HNNIntersection}(a), we thus have that the $G^t = \langle t, \cap_{s\in \widetilde{I}}A_s\rangle = \langle t, A_I\rangle$ is a special subgroup.
    Moreover, $$G^S = \langle G^{s_1}, \ldots, G^{s_k}, G^t\rangle = \langle A_{I\setminus (s_1\setminus F)},\ldots A_{I\setminus (s_k\setminus F)}, t, A_I \rangle.$$
   As $\alpha$ is (RC1), we have that $\langle A_{I\setminus (s_1\setminus F)},\ldots, A_{I\setminus (s_k\setminus F)} \rangle = \langle A^{s_1\setminus F},\ldots, A^{s_k\setminus F}\rangle = A^{\widehat{S_{\widetilde{I}}}\setminus F}$, where $S_{\widetilde{I}} = \{s_1,s_2,\ldots,s_k\}$.
   Let $i\in \widehat{S_{\widetilde{I}}}\setminus F$ and consider the  subgroup $\langle A^{i},t\rangle$ of $G$. Since $i\notin F= J\setminus K$, we must have either $i\notin J$ or $ i\in K$. If $i\notin J$, we have $i\in I\setminus J$ and $A^i\leq A_J$. Hence, as $\varphi$ is admissible and $A^i$ is a subgroup of $A_J$, there exists a parabolic subgroup of $A_K$ such that $\varphi(A^i)=A^j$, with $j\in I\setminus K$. Therefore, $i\sim_\varphi j$ and there exists an equivalence class $s\in S$ such that $\{i,j\}\subseteq s$. This means that $\langle A^i, t\rangle = \langle A^s, t\rangle$.
   Suppose now that $i\in K$ and $i\in J$ so that $i \in K \cap J$. We thus have that $i\notin I\setminus K$ and $i\notin I\setminus J$. Hence, by the definition of the equivalence class, $i$ is only equivalent to itself and
    $\langle A^i, t\rangle = \langle A^s, t\rangle$.
   Additionally, since $A_I\leq A_{I\setminus (s\setminus F)}$ for all $s\in \widetilde{I}$, we have that $$G^S = \langle A^{\widehat{S_{ \widetilde{I}}}\setminus F}, t\rangle = \langle A^{\widehat{S_{ \widetilde{I}}}}, t\rangle.$$
    Now, since $\II\setminus S\subseteq \widetilde{I}$, we have that $G_{\II\setminus S}= \cap_{s\in \II\setminus S} G_s$. As the $G_s$ are all special subgroups, by Lemma~\ref{lem:HNNIntersection}(a), we have that $G_{\II\setminus S} = \langle t, \cap_{s\in \II\setminus S} A_s\rangle$.
    Note that we can write $\II\setminus S$ as $\widetilde{I}\setminus S_{\widetilde{I}}$. Hence, $\cap_{s\in \II\setminus S} A_s = A_{\II\setminus S}=A_{\widetilde{I}\setminus S_{\widetilde{I}}}=A_{I\setminus \widehat{S_{\widetilde{I}}}}$ and $G_{\II\setminus S}=\langle t, A_{I\setminus \widehat{S_{\widetilde{I}}}}\rangle$. As $\alpha$ is (RC1), we conclude that $A^{\widehat{S_{ \widetilde{I}}}}=A_{I\setminus \widehat{S_{\widetilde{I}}}}$ and $G^S=G_{\II\setminus S}$, as desired. 
\end{proof}

\begin{prop}\label{prop:HNNFIRM}
    If $\alpha$ is (FIRM) then the HNN-extension $\alpha*_\varphi$ is (FIRM).
\end{prop}
\begin{proof}
    Suppose that $\alpha$ is (FIRM). Recall that this means that $A_{I\setminus\{i\}}\neq A_I$ for all $i\in I$.
    First of all, we have that $G_\II = A_I$. By Lemma~\ref{lem:HNNIntersection}, we also have that $G_{\II\setminus \{t\}} = \langle t, A_I\rangle$, which is clearly not equal to $G_\II$. Additionally, for any $s\in \widetilde{I}$, we have that $G_{\II\setminus\{s\}} = A_{I\setminus (s\setminus F)}$. We thus have to show that $s\setminus F\neq \emptyset$. We can rewrite the condition $i\in s\setminus F$ as $i\in s\wedge (i\notin J\vee i\in K)$.  If $J=\emptyset$, then $s\setminus F\neq \emptyset$. Let $J\subset I$ but $J\neq I$. As in the proof of the (RC1) condition above, if $i\notin J$ we will have that there exists $j\in I\setminus K$ such that $\{i,j\}\subseteq s$, with $i\sim_{\varphi}j$, or if $i\in K\cap J$, $s=\{i\} $, with $i\notin I\setminus J$ and $i\notin I\setminus K$. In either cases, we get that $i\in s\setminus F$ so that $s \setminus F$ is not empty. Finally, if $J= I$, by the admissibility of $\varphi$, the lattice of parabolic subgroups of $A_J$ must be isomorphic to the lattice of parabolic subgroups of $A_K$, meaning that $K=I$ as well. Hence $s\setminus F = s$ and thus it is also not empty. This shows that $s \setminus F$ is never empty, and concludes the proof.

\end{proof}

We now examine the last remaining property, the flag-transitivity, and tie it together with the previous results.

 \begin{thm}\label{thm:FT_HNN_RC_FRM}
     Let $A$ act flag-transitively on $\alpha$ and suppose that $\varphi$ is admissible. Then $\alpha*_\varphi$ is flag-transitive.
    Additionally, 
     \begin{enumerate}
         \item $\alpha*_\varphi$ is a geometry;
         \item if $\alpha$ is also residually-connected, then $\alpha*_\varphi$ is residually connected;
         \item if $\alpha$ is also firm, then $\alpha*_\varphi$ is firm.
     \end{enumerate}
 \end{thm}
 \begin{proof}
     Since $A$ acts flag-transitively on $\alpha$, we have, by Proposition~\ref{prop:FTequivs}(\ref{prop:FTequivs:itm:FTIntersectionOfProducts}), that  $A_S A_Q \cap A_S A_L = A_S(A_Q\cap A_L)$ for any $S,Q,L\subseteq I$.
     To show that $A*_\varphi$ acts flag-transitively on $\Gamma=\alpha*_\varphi$, we must prove that $G_SG_Q \cap G_SG_L = G_S(G_Q\cap G_L)$ for any $S,Q,L\subseteq \II$, where $\II=\widetilde{I}\cup\{t\}$.

    Suppose first that $S,Q,L\subseteq \widetilde{I}$. Then, by Lemma~\ref{lem:HNNIntersection}(a), we have that $G_S, G_Q$ and $G_L$ are special subgroups, where $G_S = \langle A_S, t\rangle$, $G_Q = \langle A_Q, t\rangle$ and $G_L = \langle A_L, t\rangle$.
    As $A$ is flag-transitive in $\alpha$ we have that $A_SA_Q \cap A_SA_L = A_S(A_Q\cap A_L)$. Therefore, by Lemma~\ref{lem:HNNIntersection}(b), we have $G_SG_Q \cap G_SG_L = G_S(G_Q\cap G_L)$.

    Suppose now that the type $t\in \II$ is in either $S$, $Q$ or $L$. If $t$ is in all $S$, $Q$ and $L$, then we have immediately that $G_SG_Q \cap G_SG_L = G_S(G_Q\cap G_L)$, since $G_S=A_{F\cup S_{\widetilde{I}}}$, $G_Q=A_{F\cup Q_{\widetilde{I}}}$ and $G_L=A_{F\cup L_{\widetilde{I}}}$, and $A$ is flag-transitive in $\alpha$. Let us separate the remaining cases as follows:
    \begin{itemize}
        \item[Case 1:] $t\in S$, $t\notin Q$ and $t\notin L$ so that $G_S=A_{F\cup S_{\widetilde{I}}}$. Let $g \in A_{F\cup S_{\widetilde{I}}}G_Q\cap A_{F\cup S_{\widetilde{I}}} G_L$. As $g\in A_{F\cup S_{\widetilde{I}}}G_Q$, we have that $g= a q$, where $a\in A_{F\cup S_{\widetilde{I}}}$, $q\in G_Q$ and $q=a_0 t^{\epsilon_1}a_1t^{\epsilon_2}a_2\ldots t^{\epsilon_m}a_m$ is reduced, with $a_i\in A_{Q}$. Notice that $g$ is also reduced. Following a similar proof to the one given in Lemma~\ref{lem:HNNIntersection}(b) for the reduced word $w$, we can conclude that $q\in G_L$. Hence, $g\in A_{F\cup S_{\widetilde{I}}}(G_Q\cap G_L)$.
    \item[Case 2:] $t\notin S$, $t\in Q$ and $t\notin L$ so that $G_Q=A_{F\cup Q_{\widetilde{I}}}$. Let $g\in G_S A_{F\cup Q_{\widetilde{I}}}\cap G_SG_L$. As $g\in G_S A_{F\cup Q_{\widetilde{I}}}$, we have that $g=sa$, where $a\in A_{F\cup Q_{\widetilde{I}}}$ and $s=a_0 t^{\epsilon_1}a_1t^{\epsilon_2}a_2\ldots t^{\epsilon_m}a_m$, with $a_i\in A_S$. First, let us assume that $s$ is the biggest prefix of $g$ that belongs to $G_S$. Notice that, as $g$ is reduced, following a similar proof to Lemma~\ref{lem:HNNIntersection}(b) for the reduced word case, we get that $a\in G_L$. As $a\in A_{F\cup Q_{\widetilde{I}}}\cap G_L$, we have that $g\in G_S(A_{F\cup Q_{\widetilde{I}}}\cap G_L)$. The case where $t$ is instead in $L$ and not in $Q$ is identical.
    \item[Case 3:] $t\in S$, $t\in Q$ and $t\notin L$ so that both  $G_S=A_{F\cup S_{\widetilde{I}}}$ and $G_Q=A_{F\cup Q_{\widetilde{I}}}$.
        Let $g \in A_{F\cup S_{\widetilde{I}}} A_{F\cup Q_{\widetilde{I}}}\cap A_{F\cup S_{\widetilde{I}}}G_L$.
        As $g\in A_{F\cup S_{\widetilde{I}}}G_L$, we have, similarly as in Case $1$, that $g=al$ is reduced, with $a\in A_{F\cup S_{\widetilde{I}}}$ and $l\in G_L$ is reduced.
        However, as $g\in A_{F\cup S_{\widetilde{I}}} A_{F\cup Q_{\widetilde{I}}}$, then $g\in A$, which implies that $l=a_0$, with $a_0\in A_L$. As $A$ is flag-transitive in $\alpha$, we have that $g\in A_{F\cup S_{\widetilde{I}}}(A_{F\cup Q_{\widetilde{I}}}\cap A_L)\subseteq A_{F\cup S_{\widetilde{I}}}(A_{F\cup Q_{\widetilde{I}}}\cap G_L)$. The case where $t$ is instead in $L$ and not in $Q$ is identical.
        
        \item[Case 4:] $t\notin S$, $t\in Q$ and $t\in L$ so that both  $G_Q=A_{F\cup Q_{\widetilde{I}}}$ and $G_L=A_{F\cup L_{\widetilde{I}}}$. Let $g\in G_SA_{F\cup Q_{\widetilde{I}}}\cap G_S A_{F\cup L_{\widetilde{I}}}$. As Case $2$, considering $g$ to have the longest prefix in $G_S$, allows us to prove that the suffix in $A_{F\cup Q_{\widetilde{I}}}$ belongs also to $A_{F\cup L_{\widetilde{I}}}$. Hence $g\in G_S(A_{F\cup Q_{\widetilde{I}}}\cap A_{F\cup L_{\widetilde{I}}})$.

    \end{itemize}
Putting all cases together, we have proved that $G$ acts flag-transitively on $\Gamma$. By Theorem~\ref{thm:cosetFT}, $\Gamma$ is then a coset geometry.
Additionally,  putting together Proposition~\ref{prop:HNNRC1} with Theorem~\ref{thm:CosetRC}, we get that $\Gamma$ is residually connected, if $\alpha$ is residually-connected.
Finally, by Proposition~\ref{prop:HNNFIRM} and Theorem~\ref{thm:CosetFIRM}, we get that $\Gamma$ is firm if $\alpha$ is firm.   
 \end{proof}

We conclude this section by some examples to illustrate the concept.

\begin{example}\label{example:HNN}
    Let $I=\{1,2,3,4,5\}$ and consider a group $A$ to be the automorphism group of a $5$--simplex, i.e. $$A=\langle a_1,a_2,a_3,a_4,a_5\mid (a_ia_j)^{m_{ij}}= 1\rangle,$$
    where, for $i\neq j$, $m_{ij}=3$ if $|i-j|=1$, and $m_{ij}=2$ otherwise, and, for $i=j$, $m_{ii}=1$.
    Consider $A_i = \langle a_j \mid j\in I\setminus\{i\}\rangle$ and, as $A$ is a Coxeter group, we have from Proposition~\ref{prop:CoxeterProperties} (b) that $A_J = \langle a_i\mid I\setminus J\rangle$. Hence, we can define the coset geometry $\alpha=(A,(A_i)_{i\in I})$, which is known to be a regular hypertope.
    
    Let $\varphi: A_{\{1,4,5\}}\to A_{\{1,2,5\}}$ be an isomorphism between $A_{\{1,4,5\}}=\langle a_2, a_3\rangle$ and $A_{\{1,2,5\}}=\langle a_3,a_4\rangle$ such that, generator wise, $\varphi(a_2)=a_3$, $\varphi(a_3)=a_4$. Notice that $\varphi$ is admissible, since $\varphi(A_{\{1,2,4,5\}})=\langle a_4\rangle = A_{\{1,2,3,5\}}$, and $\varphi(A_{\{1,3,4,5\}})=\langle a_3\rangle = A_{\{1,2,4,5\}}$.
    Hence, the set of equivalence classes $\widetilde{I}$ is $\{\{1\},\{2,3,4\},\{5\}\}$. In this case, $F = \{4\}$.

    Let $G=A*_\varphi$ and $\II=\widetilde{I}\cup\{t\}$, with $G_s=\langle A_s, t\rangle$ for each $s\in \widetilde{I}$, and $G_t = A_4$.
    Then, following Theorem~\ref{thm:FT_HNN_RC_FRM}, $G$ is flag-transitive on the coset incidence system $\Gamma=\alpha*_\varphi$, and $\Gamma$ is a firm, residually connected coset geometry ($\alpha$ being thin, since it is a regular hypertopes, implies $\alpha$ is firm).
    
\end{example}

\begin{example}\label{example:SemiDirectProduct}
Consider a coset incidence system $\alpha=(A,(A_i)_{i\in I})$ and let $\varphi: A\to A$ be an automorphism of $A$ which is admissible.
In this particular case, we have that $G = A*_\varphi = A\rtimes_\varphi \langle t\rangle\cong A\rtimes_\varphi \mathbb{Z}$.
As $\varphi$ is admissible,  the automorphism defines orbits of the maximal parabolics of $A_i$, for $i\in I$. Hence, the equivalence classes $\widetilde{I}$ can be considered to be the orbits of the type-set $I$ of the maximal parabolics $A_i$ under $\varphi$. 

The coset incidence system $\alpha*_\varphi = (G,(G_j)_{j\in \II})$ is the coset incidence system under $\II =\widetilde{I}\cup \{t\}$, where $G_j = \langle A_j,t\rangle$, for $j\in \widetilde{I}$, and $G_t = A$ (since $F = \emptyset$). This coset incidence system coincides with the one we will construct for semi-direct products in Section~\ref{sec:split:subsec:semidirect}.    

\end{example}

\section{Split Extension of Coset Incidence Systems}\label{sec:SplitExt}

In this section, we propose two ways to build a coset incidence system for the semi-direct product of two groups. The first construction presented is simpler and more natural from the point of view of algebra. However, as will become clear in Section \ref{sec:applications}, the second construction is needed in many applications.

Let $A$ and $B$ be any two groups and let $\varphi \colon B \to \Aut(A)$ be a group homomorphism. The semi-direct product $A \rtimes_\varphi B$ is defined as follows:

\begin{itemize}
    \item The underlying set for $A \rtimes_\varphi B$ is the Cartesian product $A \times B$;
    \item The group operation is given by $(a_1,b_1) (a_2,b_2) = (a_1 a_2^{b_1}, b_1b_2)$ where $a^b$ designates the action of $b$ on $a$ via $\varphi$, so that $a^b:=\varphi(b)(a)$.
\end{itemize}

Given two coset incidence systems $\alpha = (A,(A_i)_{i \in I_\alpha})$ and $\beta = (B,(B_i)_{i \in I_\beta})$, together with an action $\varphi \colon B \to \Aut(A)$ with some additional properties, we want to construct a coset incidence system $(G,(G_i)_{i \in I})$ for the semi-direct product $G = A \rtimes_\varphi B$.

We say that the action of $\varphi$ is \textit{admissible} if it preserves the lattice of maximal parabolical subgroups of $\alpha$. Note that, by Lemma~\ref{lem:lattice}, it is sufficient to ask that $\varphi$ preserves the set $\{A_i, i \in I_\alpha\}$ of maximal subgroups of $\alpha$. By abuse of notation, we will consider $\varphi$ to also be an action of $B$ on $I$. More precisely, for any $i,j \in I_\alpha$, we define $i^b = j$ if and only if $A_i^b = A_j$.

\subsection{The semi-direct geometry}\label{sec:split:subsec:semidirect}
We begin by constructing a natural geometry for the semi-direct product of two coset geometries.
\begin{definition}[Semi-direct product of $\alpha$ by $\beta$]
    Let $\alpha = (A,(A_i)_{i \in I_\alpha})$ and $\beta = (B,(B_i)_{i \in I_\beta})$ be two coset incidence systems and suppose that there is an admissible action $\varphi \colon B \to \Aut(A)$. Then, the semi-direct geometry $\alpha \rtimes_\varphi \beta$ is the coset incidence system $\alpha \rtimes_\varphi \beta := (G, (G_i)_{i \in I})$, where $G = A \rtimes_\varphi B$ and
    \begin{itemize}
        \item $I$ is the disjoint union of $I_\beta$ and $K$, where $K$ is the set of orbits of $I_\alpha$ under $\varphi$,
        \item The maximal parabolical subgroups are defined by $G_i = \langle A, B_i \rangle \cong A \rtimes B_i$, if $i \in I_\beta$, and $G_i = \langle A_i, B \rangle \cong A_i \rtimes B$ if $i \in K$.
    \end{itemize}
\end{definition}

Notice that the isomorphism $G_i = \langle A_i, B \rangle \cong A_i \rtimes B$ for $i \in K$ indeed holds, since $i$ is a subset of $I_\alpha$ closed under the action of $\varphi$.

As a set, the semi-direct product $G = A \rtimes_\varphi B$ is just the Cartesian product of $A$ and $B$. Hence, for any subgroup $C$ of $G$, we will say that the subset of $A \times B$ consisting of the elements of $C$ is the \emph{underlying set} for $C$. Let $X$ and $Y$ be two sets and let $U$ and $V$ be two subsets of $X \times Y$ and suppose that $U = U_X \times U_Y$ with $U_X \subseteq X$ and $U_Y \subseteq Y$ and $V = V_X \times V_Y$ with $V_X \subseteq X$ and $V_Y \subseteq Y$. We always have that $U \cap V \subseteq (U_X \cap V_X) \times (U_Y \cap V_Y)$, but in general, these two sets are not equal. Nonetheless, in our setting, if $C$ and $D$ are two subgroups of $A \rtimes_\varphi B$, whose underlying sets are $C_A \times C_B$ and $D_A \times D_B$, then the underlying set for $C \cap D$ will always be $(C_A \cap D_A) \times (C_B \cap D_B)$. Indeed, if $(x,y) \in C \cap D$ and $z \in C_B \cap D_B$, then $(x,z)$ is also in $C \cap D$ as $(x,z) = (x,y) (1, y^{-1}z)$. Similarly, if $w \in C_A \cap D_A$, then $(w,y) \in C \cap D$ as $(w,y) = (wx^{-1},1)(x,y) $.

With these remarks on subsets of Cartesian products in mind, we now determine the structure of all parabolic subgroups of $\alpha \rtimes_\varphi \beta$.

\begin{lemma}\label{lem:parabolicsNatural}
    Let $\alpha \rtimes_\varphi \beta = (G, (G_i)_{i \in I})$ be as above. Then, for any $J = J_\alpha \sqcup J_\beta$ where $J_\beta \subseteq I_\beta$ and $J_\alpha \subseteq K$, we have that $G_J \cong A_{J_\alpha} \rtimes B_{J_\beta}$.
\end{lemma}
\begin{proof}
We first note that the underlying sets of the maximal parabolic subgroups $G_i$, $i \in I$ are clearly $(A \times B_i)$ if $i\in I_\beta$ and $A_i \times B$ if $i \in K$. It follows from the discussion preceding this lemma that the underlying set for $G_J$ is $(\cap_{j\in J_\alpha} A_{j}) \times (\cap_{i \in J_\beta} B_i) = A_{J_\alpha} \times B_{J_\beta}$. Hence, the groups $G_J$ must be isomorphic to $A_{J_\alpha} \rtimes_\varphi B_{J_\beta}$. 
\end{proof}

We can now investigate the properties of $\alpha \rtimes_\varphi \beta$.
\begin{prop}\label{prop:RCSemiDirect}
    If both $\alpha = (A,(A_i)_{i \in I_\alpha})$ and $\beta = (B,(B_i)_{i \in I_\beta})$ satisfy (RC1), then $\alpha \rtimes_\varphi \beta$ also satisfies (RC1).
\end{prop}
\begin{proof}
    We need to show that $G_J = \langle G_{J \cup \{i\}}, i \in I \setminus J \rangle$ for every $J \subset I$ such that $|I\setminus J|\geq 2$. 
    The reverse inclusion is trivial, hence we will prove that  $G_J \subseteq \langle G_{J \cup \{i\}}, i \in I \setminus J \rangle$.
    By Lemma~\ref{lem:parabolicsNatural}, we know that $G_J \cong A_{J_\alpha} \rtimes B_{J_\beta}$, where $J = J_\alpha \sqcup J_\beta$ with $J_\beta \subseteq I_\beta$ and $J_\alpha \subseteq K$. Hence, $G_J$ is generated by $A_{J_\alpha}$ and $B_{J_\beta}$, i.e. $G_J=\langle A_{J_\alpha}, B_{J_\beta}\rangle$. We clearly have that $B_{J_\beta} \subseteq \langle G_{J \cup \{i\}}, i \in I \setminus J \rangle$. It remains thus to check that $A_{J_\alpha} \subseteq \langle G_{J \cup \{i\}}, i \in I \setminus J \rangle$.
    By Lemma \ref{lem:RCequivs}(\ref{lem:RCequivs:itm:Parabolics}) we have that $A_{J_\alpha} = \langle A_{J_\alpha \cup \{j\}}, j \in K \setminus J_\alpha \rangle$. Hence, $A_{J_\alpha}\subseteq \langle G_{J \cup \{i\}}, i \in I \setminus J \rangle$. Therefore, $G_J = \langle A_{J_\alpha}, B_{J_\beta}\rangle \subseteq \langle G_{J \cup \{i\}}, i \in I \setminus J \rangle$.
\end{proof}
Notice that Proposition \ref{prop:RCSemiDirect} is not an ``if and only if" statement. Indeed, only the subgroups $A_{J}$, with $J \subseteq K$, have to satisfy (RC1) for $\alpha \rtimes_\varphi \beta$ to satisfy (RC1). The proposition can thus be turned into an if and only if statement, if necessary, under these slightly complicated hypotheses. A similar comment holds for the next proposition about flag-transitivity.
\begin{prop}\label{prop:semidir_FT}
If $A$ is flag-transitive on $\alpha = (A,(A_i)_{i \in I_\alpha})$ and $B$ is flag-transitive on $\beta = (B,(B_i)_{i \in I_\beta})$, then $G = A \rtimes_\varphi B$ is flag-transitive on $\alpha \rtimes_\varphi \beta$.
\end{prop}
\begin{proof}
    We want to show that $G_JG_i = \cap_{j\in J} G_j G_i$ for any $J \subseteq I$ and $i\in I\setminus J$. Let $J = J_\alpha \sqcup J_\beta$, with $J_\beta \subseteq I_\beta$ and $J_\alpha \subseteq K$.
    
    First suppose that $i \in I_\beta$. Then, the equation we need to verify becomes

    $$ (A_{J_\alpha} \rtimes B_{J_\beta}) (A \rtimes B_i) = (\cap_{j \in I_\beta \cap J} (A \rtimes B_j) (A \rtimes B_i)) \cap ( \cap_{k \in K \cap J} (A_j \rtimes B) (A \rtimes B_i)) $$

    We will once again work with the underlying sets of every subgroups appearing in the equation. In doing so, the equation splits into the following two equations.
    
    \begin{equation}\label{eq:SDFT1}
          A_{J_\alpha} A^{B_{J_\beta}} =  \cap_{k \in K \cap J} A_j A^{B_i}
    \end{equation}
    where $A^{B_{J_\beta}}$ denotes the set of all elements of the form $a^{b}$ with $a \in A$ and $b \in B_{J_\beta}$,
    and 
  \begin{equation}\label{eq:SDFT2}
      B_{J_\beta}B_i = \cap_{j\in I_\beta \cap J} B_j B_i.
  \end{equation}
    Equation \ref{eq:SDFT2} holds directly from the fact that $B$ acts flag-transitively on $\beta$. For Equation \ref{eq:SDFT1}, note that $A^{B_i} = A^{B_{J_\beta}} = A$, so the equation holds trivially.

    Suppose now that $i \in K$. The equation then becomes

        $$ (A_{J_\alpha} \rtimes B_{J_\beta}) (A_i \rtimes B) = (\cap_{j \in I_\beta \cap J} (A \rtimes B_j) (A_i \rtimes B) )\cap (\cap_{k \in K \cap J} (A_k \rtimes B) (A_i \rtimes B))$$

which then splits into

\begin{equation}\label{eq:SDFT3}
          A_{J_\alpha} A_i^{B_{J_\beta}} =  \cap_{k \in K \cap J} A_k A_i^{B}
    \end{equation}
    and 
  \begin{equation}\label{eq:SDFT4}
      B_{J_\beta}B = \cap_{j\in I_\beta \cap J} B_j B
  \end{equation}

  This time, Equation \ref{eq:SDFT4} holds trivially. For Equation \ref{eq:SDFT3}, note that since $i\in K$, $i$ is a orbit of types closed under the action $\varphi$. Hence, $A_i^{B_{J_\beta}} = A_i =A_i^B$, so that the equation holds by flag-transitivity of $A$ of $\alpha$.

\end{proof}

Although this construction is very natural and direct, it can be improved by the approach given in Section~\ref{subsec:twist_section}.
In particular, the coset geometries constructed via the above construction are never thin, even if the components $\alpha$ and $\beta$ were thin.
Suppose that $A$ and $B$ are flag-transitive on $\alpha$ and $\beta$, respectively, and suppose both $\alpha$ and $\beta$ are thin.
Given $J\subseteq I$, we have that $G_{J}=A_{J_\alpha}\rtimes_\varphi B_{J_\beta}$. 
Now, from the fact that $\beta$ is thin, one can prove that, for $i\in I_\beta$, $[G^i:G_I]=2$.
However, for $i\in K$, $G^{i} = A_{I_\alpha\setminus\{i\}}\rtimes_\varphi B_{I_\beta}$.
As $i$ is a set of subsets of $I_\alpha$, $A_{I_\alpha\setminus\{i\}}$ only coincides with a minimal parabolic of $\alpha$ if $i$ is a singleton or if $A_{I_\alpha\setminus\{i\}} = A^j$, for $j\in I_\alpha$. This is not in general the case.
Hence thinness is usually not inherited in this construction.

Nonetheless, this construction provides a good introduction to the next operation, the twisting. Moreover, it also has the advantage of showing that the semi-direct construction coincides with the HNN-extension construction, as expressed in Example~\ref{example:SemiDirectProduct}. As thinness is never achieved using the HNN-extension, the semi-direct construction is the most suitable choice for that comparison.

\subsection{The Twisting geometry}\label{subsec:twist_section}

We now construct another geometry for a semi-direct product of two coset incidence systems, under some additional conditions. This second construction might seem somewhat artificial at first, but it is in fact a generalization of the well known twisting operation on polytopes, and will turn out to be essential for most of the applications in Section \ref{sec:applications}.

Let $\alpha=(A,(A_i)_{i\in I_\alpha})$ and $\beta=(B,(B_i)_{i\in I_\beta})$ be two coset incidence systems, and as before, assume that we have an admissible action $\varphi$ of $B$ on $A$ by automorphisms.

\begin{definition}\label{def:GammaPhiAdmiss}
    We say $\alpha$ is $(\beta, \varphi)$-admissible 
    if
    \begin{enumerate}
        \item The action $\varphi$ of $B$ on $A$ permutes the maximal parabolic subgroups $(A_i)_{i\in I_\alpha}$;
        \item Consider that $\varphi$ also acts on the type set $I_\alpha$ in the natural way and let $K$ be the set of orbits of $B$ on $I_\alpha$. Then, there is exactly one orbit $L\in K$ such that $|L|>1$;
        \item Given a fixed $F_0\in L$, for any $J \subseteq I_\beta$, define $O_J$ to be the orbit of $F_0$ under the action of the parabolic subgroup $B_{I_\beta\setminus J}$. Then, for all $M,N\subseteq I_\beta$, we have that
        \begin{equation}\tag{IPO}\label{eq:IPO}
          O_M\cap O_N = O_{M\cap N} 
        \end{equation}
        
    \end{enumerate}
\end{definition}

Let $I = I_\beta \sqcup K$, where $K$ is the set of orbits of the action of $B$ on $I_\alpha$. We want to define a new geometry on the type set $I$ for the group $G = A \rtimes_\varphi B$. Let us define $O^J := O_{I_\beta \setminus J}$ for all $J \subseteq I_\beta$. Notice that then $O^J$ is the orbit of $F_0$ under the action of $B_J$.

\begin{definition}[Twisting of $\alpha$ by $\beta$]
    Let $\alpha$ and $\beta$ be two coset incidence systems such that $\alpha$ is $(\beta,\varphi)$-admissible
    The twisting of $\alpha$ by $\beta$ is the coset incidence system $\TT(\alpha,\beta) = (G,(G_i)_{i \in I})$, where
    the maximal parabolic subgroups $(G_i)_{i \in I}$ are:

\begin{equation}
 G_i= \begin{cases}
     A_{L \setminus O^i} \rtimes B_i, & \text{if $i \in I_\beta$}.\\
     A_i \rtimes B, & \text{if $i \in K$}.
  \end{cases}
\end{equation}
\end{definition}

Recall that $K$ is the disjoint union of $L$ and singleton sets. Hence, for $i \in K$, we have that $G_i = A_L \rtimes B$ if $i = L$, and $G_i = A_i \rtimes B$ if $i \notin L$. Note that all maximal parabolic subgroups can be described in a single expression. Indeed, for all $i \in I$, we have that $G_i = A_{(L\setminus O^{I_\beta\cap \{i\}})\cup (K\cap \{i\})} \rtimes B_{I_\beta \cap \{i\}}$. This will be useful to treat all cases simultaneously in some proofs. We begin by stating the following elementary result. Its proof is trivial, and thus omitted. Nonetheless, it will help navigate the heavy notation appearing in the rest of this section.

\begin{lemma}\label{lem:setIntersections}
The following two equalities always hold.
    \begin{enumerate}
        \item $O^M \cap O^N = O^{M \cup N}$
        \item $L \setminus O^M \cup L \setminus O^N = L \setminus O^{M \cup N}$
    \end{enumerate}
\end{lemma}

We now determine the structure of all standard parabolical subgroups $G_J$ for any $J \subseteq I$.

\begin{lemma}\label{lem:TwistG_J}
    Let $J \subseteq I$ and let $J_\beta = J \cap I_\beta$ and $J_\alpha = J \cap K$. Then, $G_J = A_{(L\setminus O^{J_\beta}) \cup J_\alpha}\rtimes G_{J_\beta}$ 
\end{lemma}
\begin{proof}
    By definition, we have that $G_J = \cap_{j \in J} G_j$.
    Using the usual argument, we obtain that the underlying set for $G_J$ is $( (\cap_{j \in J_\alpha} A_j) \cap ( \cap_{j \in J_\beta} A_{L \setminus O^{\{j\}}}) , \cap_{j \in J_\beta} B_j)$. The second factor is equal to $B_{J_\beta}$ as desired. We therefore focus on the first factor. We clearly have that $\cap_{j \in J_\alpha} A_j = A_{J_\alpha}$ and that $\cap_{j \in J_\beta} A_{L \setminus O^{\{j\}}} = A_{(\cup_{j \in J_\beta} L \setminus O^{\{j\}})}$. By Lemma \ref{lem:setIntersections}, we get that $\cup_{j \in J_\beta} L \setminus O^{\{j\}} = L \setminus O^{(\cup_{j\in J_\beta} \{j\})} = L \setminus O^{J_\beta}$, concluding the proof.
\end{proof}

Finally, we explore the properties of $\TT(\alpha,\beta)$.

\begin{prop}\label{prop:TwistingFT}
    If $A$ is flag-transitive on $\alpha$ and $B$ is flag-transitive on $\beta$, then $G$ is flag-transitive on $\TT(\alpha,\beta)$.
\end{prop}
\begin{proof}
    Fix a pair $i,k\in I$ of distinct types, a subset $J \subseteq I\setminus\{i,k\}$, and an element $x\in G_J$. By Proposition~\ref{prop:FTequivs}$(\ref{prop:FTequivs:itm:FTPassini})$, we will prove that if $G_i\cap xG_k\neq \emptyset$, then $G_J\cap G_i\cap xG_k\neq \emptyset$.
    Let us set $M_i = (L\setminus O^{I_\beta\cap \{i\}})\cup (K\cap \{i\})$ and $M_k = (L\setminus O^{I_\beta\cap \{k\}})\cup (K\cap \{k\})$. Then, we have that $G_i = A_{M_i} \rtimes B_{I_\beta \cap \{i\}}$, $G_k = A_{M_k} \rtimes B_{I_\beta \cap \{k\}}$ and $G_J =  A_{(L\setminus O^{J_\beta}) \cup J_\alpha}\rtimes B_{J_\beta}$ where $J_\beta = J \cap I_\beta$ and $J_\alpha = J \cap K$.

    Consider $x\in G_J$ and suppose that $G_i\cap xG_k\neq \emptyset$.
    This means that we can find $y\in G_i$ and $z\in G_k$ such that $y = xz$.
    By the definition of $G_i$, $G_k$ and $G_J$, we have that $y=(y_A,y_B)$, $z = (z_A,z_B)$, and $x = (x_A,x_B)$, with $y_A \in A_{M_i},z_A \in A_{M_j},x_A\in A_{(L \setminus O^{J_\beta}) \cup J_\alpha}$ and $y_B \in B_{I_\beta \cap \{i\}},z_B \in B_{I_\beta \cap \{j\}},x_B\in B_{J_\beta}$.
    We then have that $$y=xz \Leftrightarrow (y_A,y_B)=(x_A,x_B)(z_A,z_B) = (x_A z_A^{x_B},x_B z_B),$$
    where we recall that $ z_A^{x_B} = \varphi(x_B)(z_A)$. 
    The equation $y = xz$ therefore holds if and only if we have that $y_A = x_A z_A^{x_B}$ and $y_B = x_B z_B$.
    
    First, note that $y_B = x_B z_B$ implies that $B_{I_\beta\cap\{i\}}\cap x_B B_{I_\beta\cap\{k\}}\neq \emptyset$ (recall that $B_\emptyset = B$).
    By the flag-transitivity of $B$ in $\beta$, we then have that 
    $B_{J_\beta}\cap B_{I_\beta\cap\{i\}}\cap x_B B_{I_\beta\cap\{k\}}\neq \emptyset$. Here, notice that $i$ or $k$ could be equal to an element of $K$, and therefore the fact that $B_{J_\beta}\cap B_{I_\beta\cap\{i\}}\cap x_B B_{I_\beta\cap\{k\}}\neq \emptyset$ is not a direct consequence of Proposition~\ref{prop:FTequivs}$(\ref{prop:FTequivs:itm:FTPassini})$ for $\beta$. Instead, this is a consequence of Proposition~\ref{prop:FTequivs}$(\ref{prop:FTequivs:itm:FTIntersectionCosets})$. Indeed, recall that $x_B \in B_{J_\beta}$ so that all the pairwise intersections are not empty.
    The fact that $B_{J_\beta}\cap B_{I_\beta\cap\{i\}}\cap x_B B_{I_\beta\cap\{k\}}\neq \emptyset$ in any of the cases then means that we can find
    $t_B\in B_{J_\beta}$, $y'_B\in B_{I_\beta\cap\{i\}}$ and $z'_B\in B_{I_\beta\cap\{k\}}$ such that $t_B = y'_B = x_B z_B'$. 

    Similarly, since $y_A = x_A z_A^{x_B}$, we have that 
    $$A_{M_i}\cap x_A (A_{M_k})^{x_B} \neq \emptyset.$$
    Note that, by the $(\beta,\varphi)$-admissibility and Lemma~\ref{lem:lattice}, $(A_{M_k})^{x_B}$ is a standard parabolic subgroup of $\alpha$.
    We therefore get that
    $$A_{(L\setminus O^{J_\beta}) \cup J_\alpha} \cap A_{M_i}\cap x_A ( A_{M_k} )^{x_B} \neq \emptyset,$$
    by Proposition~\ref{prop:FTequivs}$(\ref{prop:FTequivs:itm:FTIntersectionCosets})$. Once again, it is readily checked that all the pairwise intersections are non-empty.
    We thus find $t_A\in A_{(L\setminus O^{J_\beta}) \cup J_\alpha}$, $y'_A\in A_{M_i}$ and $z_A'\in A_{M_k}$ such that $t_A = y_A' = x_A (z_A')^{x_B}$.

    Putting both parts together, we have that $(t_A,t_B) = (y_A',y_B') = (x_A (z_A')^{x_B}, x_B z_B')$.
    Clearly $(t_A,t_B)\in G_J$, $(y_A',y_B') \in G_i$ and $(x_A (z_A')^{x_B}, x_B z_B') = (x_A,x_B)(z_A',z_B') \in x G_k$.
    Hence, we get that $G_J\cap G_i\cap x G_k\neq \emptyset$, concluding the proof.
\end{proof}

\begin{prop} \label{prop:TwistingRC1}
    If both $\alpha$ and $\beta$ satisfy (RC1), then $\TT(\alpha,\beta)$ also satisfies (RC1).
\end{prop}

\begin{proof}
    Let $J \subset I$ with $|I \setminus J| \geq 2$ and let $J_\beta = J \cap I_\beta$ and $J_\alpha = J \cap K$. We need to show that 
    \begin{equation}\label{eq:RC1Twisting}
          G_J = \langle G_{J \cup \{i\}}\mid i \in I \setminus J \rangle.
    \end{equation}
   By Lemma~\ref{lem:TwistG_J}, we have that $G_J = A_{(L\setminus O^{J_\beta}) \cup J_\alpha}\rtimes B_{J_\beta}$, $G_{J \cup \{i\}} = A_{(L\setminus O^{J_\beta \cup \{i\}}) \cup J_\alpha}\rtimes B_{J_\beta \cup \{i\}}$ if $i \in I_\beta$ and $G_{J \cup \{i\}} = A_{(L\setminus O^{J_\beta}) \cup (J_\alpha \cup \{i\})}\rtimes G_{J_\beta}$ if $i \in K$. 
   We remind the reader that $A_{I_\alpha}$ and $B_{I_\beta}$ are, respectively, the Borel subgroups of $\alpha$ and $\beta$.

    Suppose first that $K \subseteq J$, so that $J_\alpha = K$. Then, $A_{(L \setminus O^{J_\beta}) \cup K} = A_{I_\alpha}$ and $G_J = A_{I_\alpha} \rtimes B_{J_\beta}$. In this case, Equation~\ref{eq:RC1Twisting} holds because it holds in $\beta$ by assumption and that all parabolic subgroups contain $A_{I_\alpha}$. Indeed, by Lemma~\ref{lem:RCequivs}$(\ref{lem:RCequivs:itm:RC2})$, we have that Equation~\ref{eq:RC1Twisting} is equivalent to $G_J = \langle A_{I_\alpha} \rtimes B_{J_\beta \cup \{i\}}\ , A_{I_\alpha} \rtimes B_{J_\beta \cup \{j\}} \rangle$, for any $i,j \in I_\beta \setminus J_\beta$.
 
 Suppose now that $I_\beta \subseteq J$, so that $J_\beta = I_\beta$. Then, $G_J = A_{(L \setminus O^{I_\beta}) \cup J_\alpha} \rtimes B_{I_\beta}$. Take any distinct $i,j \in K \setminus J_\alpha$. We will show, using Lemma \ref{lem:RCequivs}$(\ref{lem:RCequivs:itm:Parabolics})$, that $G_J = \langle G_{J \cup \{i\}}, G_{J \cup \{j\}} \rangle$. 
 Indeed, $G_{J \cup \{i\}} = A_{(L\setminus O^{I_\beta}) \cup (J_\alpha \cup \{i\})}\rtimes B_{I_\beta}$ and $G_{J \cup \{j\}} = A_{(L\setminus O^{I_\beta}) \cup (J_\alpha \cup \{j\})}\rtimes B_{I_\beta}$.
 We thus simply need to verify that $((L\setminus O^{I_\beta}) \cup (J_\alpha \cup \{i\})) \cap ((L\setminus O^{I_\beta}) \cup (J_\alpha \cup \{j\})) = (L \setminus O^{I_\beta}) \cup J_\alpha$. If neither $i$ nor $j$ is equal to $L$, it is obvious. Suppose instead that $i = L$. Then, we have $(L \cup J_\alpha) \cap ((L\setminus O^{I_\beta}) \cup (J_\alpha \cup \{j\})) = (L \setminus O^{I_\beta}) \cup J_\alpha$, as desired.

    Finally, suppose that $J$ contains neither $I_\beta$ nor $K$. Therefore, it is always possible to find $i \in I_\beta \setminus J_\beta$ and $j \in K \setminus J_2$. Hence, we have that $G_{J \cup \{i\}} = A_{(L\setminus O^{J_\beta \cup \{i\}}) \cup J_\alpha}\rtimes G_{J_\beta \cup \{i\}}$ and $G_{J \cup \{j\}} = A_{(L\setminus O^{J_\beta}) \cup (J_\alpha \cup \{j\})}\rtimes B_{J_\beta}$. Since $G_J = A_{(L\setminus O^{J_\beta}) \cup J_\alpha}\rtimes G_{J_\beta}$, we need to show that $B_{J_\beta} \subset \langle G_{J \cup \{i\}}, G_{J \cup \{j\}} \rangle$ and $A_{(L\setminus O^{J_\beta}) \cup J_\alpha} \subset \langle G_{J \cup \{i\}}, G_{J \cup \{j\}} \rangle$. The first containment holds since $B_{J_\beta}$ is already contained in $G_{J \cup \{j\}}$. For the second one, we can either notice that $A_{(L\setminus O^{J_\beta}) \cup J_\alpha}$ is contained in $\langle A_{(L\setminus O^{J_\beta \cup \{i\}}) \cup J_\alpha} ,A_{(L\setminus O^{J_\beta}) \cup (J_\alpha \cup \{j\})} \rangle$ or we can use the $B_{J_\beta}$ factor of $G_{J \cup \{j\}}$ acting on the factor $A_{(L \setminus O^{J_\beta \cup \{i\}}) \cup J_\alpha}$ of $G_{J \cup \{i\}}$ to transform it into the desired $A_{(L\setminus O^{J_\beta}) \cup J_\alpha}$.

\end{proof}

\begin{prop}\label{prop:GammaLambFIRM-TwistFIRM}
    If $\alpha$ and $\beta$ satisfy (FIRM), then $\TT(\alpha,\beta)$ satisfies (FIRM).
\end{prop}
\begin{proof}
    Let $\alpha$ and $\beta$ satisfy (FIRM), so that $A_{I_\alpha\setminus \{i\}} \neq A_{I_\alpha}$ and $B_{I_\beta\setminus\{j\}}\neq B_{I_\beta}$, for each $i\in I_\alpha$ and $j\in I_\beta$.
    Moreover, by Lemma~\ref{lem:TwistG_J}, we can check that
    \begin{equation*}
        G_{I\setminus\{i\}} =
        \begin{cases}
            A_K\rtimes B_{I_\beta\setminus\{i\}} &, \textnormal{ if } i\in I_\beta \\
            A_{(L\setminus O^{I_\beta})\cup (K\setminus\{i\})}\rtimes B_{I_\beta}&, \textnormal{ if } i\in K
        \end{cases}
    \end{equation*}
    and $G_{I}=A_K\rtimes B_{I_\beta}$. Notice that $A_K = A_{I_\alpha}$.
    As $\alpha$ and $\beta$ satisfy (FIRM), we directly get that $G_{I}\neq G_{I\setminus\{i\}}$, for each $i\in I$.
\end{proof}

Finally, we regroup all the results on the twisting geometry in the following theorem.

\begin{thm}\label{thm:Twisting_FT_Geo_RC_FRM}
    Let $\alpha=(A,(A_i)_{i\in I_\alpha})$ and $\beta=(B,(B_i)_{i\in I_\beta})$ be two coset incidence systems, $\varphi: B\rightarrow \Aut(A)$ a group homomorphism, and let $\alpha$ be $(\beta,\varphi)$-admissible. Suppose that $A$ and $B$ are flag-transitive on $\alpha$ and $\beta$, respectively. Then
    \begin{enumerate}
        \item the group $G=A\rtimes_\varphi B$ acts flag-transitively on $\TT(\alpha,\beta)$;
        \item $\TT(\alpha,\beta)$ is a coset geometry;
        \item $\TT(\alpha,\beta)$ is a finite coset geometry if and only if $\alpha$ and $\beta$ are both finite coset geometries.
        \item $\TT(\alpha,\beta)$ is residually connected if both $\alpha$ and $\beta$ are residually connected;
        \item $\TT(\alpha,\beta)$ is firm (resp. thin) if both $\alpha$ and $\beta$ are firm (resp. thin);
    \end{enumerate}
\end{thm}
\begin{proof}
    Part (a), (b), (c) and (d) are direct consequences of Proposition~\ref{prop:TwistingFT}, Theorem~\ref{thm:cosetFT}, Proposition \ref{prop:FiniteCosetGeom}, and Proposition \ref{prop:TwistingRC1}.

    It remains thus only to show (e). We already know that $\TT(\alpha,\beta)$ is flag-transitive. Additionally, since $\alpha$ and $\beta$ are firm, by Theorem~\ref{thm:CosetFIRM}, we have that they both satisfy (FIRM). Therefore, by Proposition~\ref{prop:GammaLambFIRM-TwistFIRM} and Theorem~\ref{thm:CosetFIRM}, we have that $\TT(\alpha,\beta)$ is at least firm.
    
    Suppose now that $\alpha$ and $\beta$ are actually thin. This implies that $|A_{I_\alpha\setminus\{i\}}:A_{I_\alpha}|=2$ and $|B_{I_\beta\setminus\{j\}}:B_{I_\beta}|=2$, for each $i\in I_\alpha$ and $j\in I_\beta$.
    Moreover, we have that $B_{I_\beta} = \{1\}$ and $A_{I_\alpha}=\{1\}$. 
    Hence, we have $O^{I_\beta}=\{F_0\}$, the representative of the orbit $L$.
    Therefore, we get that 
    \begin{equation*}
        G_{I\setminus\{i\}} =
        \begin{cases}
            B_{I_\beta\setminus\{i\}}&, \textnormal{ if } i\in I_\beta \\
            A_{(L\setminus\{F_0\})\cup (K\setminus\{i\})}&, \textnormal{ if } i\in K 
        \end{cases}
    \end{equation*}
    and $G_{I}=\{1\}$. Notice that $A_{(L\setminus\{F_0\})\cup K\setminus\{i\}}=A_{K\setminus\{i\}} = A_{I_\alpha\setminus\{i\}}$ for $i \neq L \in K$, and 
    $A_{(L\setminus\{F_0\})\cup K\setminus\{L\}}= A_{I_\alpha\setminus\{F_0\}}$. This shows that the minimal parabolic subgroups of $\TT(\alpha,\beta)$ are always minimal parabolic subgroups of $\alpha$ or $\beta$.
    Hence, we get that $|G_{I\setminus\{i\}}:G_{I}|=2$, showing that $\TT(\alpha,\beta)$ is thin.
\end{proof}

\begin{coro}\label{coro:TwistingHypertope}
    Let $\alpha$ and $\beta$ be regular hypertopes. 
    For a group homomorphism $\varphi: B\rightarrow \Aut(A)$, if $\alpha$ is $(\beta,\varphi)$-admissible, then $\TT(\alpha,\beta)$ is a regular hypertope.
\end{coro}
\begin{proof}
    This is a direct consequence of Theorem~\ref{thm:Twisting_FT_Geo_RC_FRM} and the definition of a regular hypertope.
\end{proof}

\section{Applications} \label{sec:applications}

In this section, we explore multiple applications of the operations defined in Section \ref{sec:Amalgams+HNN} and Section \ref{sec:SplitExt}. We begin by defining a large family of groups that have the particularity of admitting a presentation that is entirely determined by a graph. These groups will be building blocks used to perform amalgams and twistings.

Coxeter and Artin-Tits groups are famous examples of groups defined by some underlying graphs. Shephard groups, as defined in~\cite{Kapovich1998,Goldman2024}, are a generalization of both Coxeter groups and Artin-Tits groups. We now make their definition explicit.

A graph $\GG$ is a set of vertices $V = V(\GG)$ together with a set $E = E(\GG)$ of unordered pairs of vertices, called edges. Therefore, the edges in the graph $\GG$ are undirected and we allow (undirected) loops. For convenience, we will say that an edge of $\GG$ is \emph{proper} if it is not a loop. A graph $\GG$ is \emph{labeled} if each edge $e$ is labeled by an integer $m_e \in \mathbb{N}_{\geq 3} \cup \{ \infty \}$. For any vertex $v$, we denote by $m_v:=m_{(v,v)}$ the label of the loop $(v,v)$, if it exists. If $e = (v,w)$ is an edge, we denote the label $m_{(v,w)}$ by $m_{v,w}$. Finally, we say a labeled graph is \emph{properly labeled} if $m_v = m_w$, whenever $m_{v,w}$ is odd.

\begin{definition}[Shephard group]
Let $\GG$ be a properly labeled graph with $n$ vertices. Then, the \emph{Shephard group} $A = A(\GG)$ is the group admitting the presentation $A = \langle S \mid R \rangle$ where $S = V(\GG)$, and $R$ is generated by
    \begin{itemize}
    \item $a^{m_a} = 1$, for each $a \in S$ such that the loop $(a,a) \in E(\GG)$ and $m_a \neq \infty$,
    \item $a^2 = 1$ for each $a \in S$ such that there is no loop $(a,a) \in E(\GG)$,
    \item $(ab)_{m_{a,b}} =(ba)_{m_{a,b}} $ for all edge $(a,b) \in E(\GG)$ such that $m_{a,b} \neq \infty$,
    \item $ab = ba$ for all $a \neq b \in S$ such that there is no edge $(a,b) \in E(\GG)$.
\end{itemize}
where $$a^{m_a} = \underbrace{aa\ldots a}_{m_a\textnormal{ factors}}$$
and
  \begin{equation*}
       (ab)_{m_{a,b}} =
        \begin{cases}
            \underbrace{aba\ldots b }_{m_{a,b}\textnormal{ factors}}, \textnormal{ if } m_{a,b} \textnormal{ is even } \\
            \underbrace{aba\ldots a }_{m_{a,b}\textnormal{ factors}}, \textnormal{ if } m_{a,b} \textnormal{ is odd }
        \end{cases}
    \end{equation*}
\end{definition}

Note that if the graph $\GG$ has no loops, the group $A = A(\GG)$ is then a Coxeter group, since all generators have order two. Similarly, if $\GG$ contains every possible loop and they are all labeled by $\infty$, the group $A = A(\GG)$ is an Artin-Tits group. Additionally, if all proper edges of $\GG$ are labeled by $\infty$, we say that the group $A = A(\GG)$ is \textit{right-angled}. This is consistent with the usual meaning of right-angled Coxeter groups and right-angled Artin-Tits groups. Finally, if all generators have finite order, the group $A = A (\GG)$ is called a complex reflection group~\cite{ShephardTodd1954}.

Notice that in $A = A(\GG)$, the generators $a$ and $b$ are conjugate whenever $m_{a,b}$ is odd. That is why we required the labeling of $\GG$ to be proper.

\begin{example}
    Consider the properly labeled graph $\GG$ below, with $V(\GG)=\{a,b,c,d\}$.
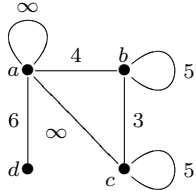
    $$\xymatrix@-1pc{*{\bullet} \ar@{-}[rr]^4^(0.99){b}\ar@{-}[rrdd]_(0.99){c}_\infty \ar@(ul,ur)@{-}[]^\infty && *{\bullet} \ar@{-}[dd]^3 \ar@(ur,dr)@{-}[]^5 \\
\\
*{\bullet}\ar@{-}[uu]^6^(0.99){a}^(0.01){d} && *{\bullet}\ar@(ur,dr)@{-}[]^5 }$$

    Then, $A=A(\GG)$ is the Shephard group with the following  presentation
    $$A=\langle a, b, c, d \mid b^5=c^5=d^2=1_A, abab=baba,dc=cd,bd=db,$$
    $$bcb=cbc,adadad=dadada\rangle.$$

\end{example}

\begin{example}
    Consider the following properly labeled graph $\GG$ with $n$ vertices.
    $$
    
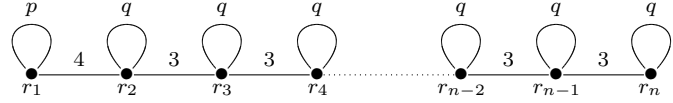
\xymatrix@-1pc{
    *{\bullet}\ar@{-}[rr]^4_(0.01){r_1} \ar@(ul,ur)@{-}[]^p &&  *{\bullet}\ar@{-}[rr]^3_(0.01){r_2} \ar@(ul,ur)@{-}[]^q &&  *{\bullet}\ar@{-}[rr]^3_(0.01){r_3} \ar@(ul,ur)@{-}[]^q && *{\bullet}\ar@{.}[rrr]_(0.01){r_4}\ar@(ul,ur)@{-}[]^q &&& *{\bullet}\ar@{-}[rr]^3_(0.01){r_{n-2}}\ar@(ul,ur)@{-}[]^q && *{\bullet}\ar@{-}[rr]^3_(0.01){r_{n-1}}_(0.99){r_n}\ar@(ul,ur)@{-}[]^q && *{\bullet}\ar@(ul,ur)@{-}[]^q
    }
    $$
    Then, $A=A(\GG)$ is the Shephard group with the following presentation
    $$A=\langle S \mid r_1^p=r_i^q=1_A \textnormal{ for }2\leq i\leq n, r_1r_2r_1r_2 = r_2r_1r_2r_1,$$ 
    $$ r_jr_{j+1}r_j=r_{j+1}r_jr_{j+1} \textnormal{ for } 2\leq j\leq n-1\rangle $$
    where $S = \{r_1,r_2,\ldots,r_n\}$.
    This Shephard group is usually denoted by $B_n(p,q)$ and two families of these groups have been previously discussed: the $B_n(2,\infty)$ groups, which are associated with Mikado braids, and the $B_n(\infty,2)$ groups, which are usually called the ``middle groups".
\end{example}

Let $\GG$ be a labeled graph with $V(\GG)=\{a_1,\ldots,a_n\}$. There is a canonical way of building a coset incidence system of rank $n$ from the Shephard group $A=A(\GG)$. Let $I_\alpha=\{1,\ldots,n\}$ and consider the family of subgroups $(A_i)_{i\in I_\alpha}$, where $A_i=\langle a_j \mid j\in I_\alpha\setminus\{i\}\rangle$. 
The \textit{standard coset incidence system associated to $A(\GG)$} is the coset incidence system $\alpha=(A,(A_i)_{i\in I_\alpha})$.

Notice that, we do not know at the moment whether the intersection of maximal parabolics of Shephard groups behaves like those from Coxeter and Artin-Tits groups. Indeed, in general, $A_J = \cap_{j\in J}A_i$ is not always equal to $\langle a_i\mid i\in I_\alpha\setminus J\rangle$. Thanks to the theory developed in the previous section, we will show that it is the case for many families of Shephard groups.

\subsection{Free Product of Artin-Tits groups with Coxeter groups}\label{sec:subsec:ProductArtinCoxeter}

Let $A,B$ be either Artin-Tits groups or  Coxeter groups.
The free product $G = A*B$ is always a Shephard group, with each generator having order equal to either two or infinity. In particular, if both $A$ and $B$ are Artin-Tits groups (resp. Coxeter groups), the free product $G = A * B$ is again an Artin-Tits group (resp. Coxeter group). 

\begin{thm}\label{thm:FreeProductArtinCoxeter}
    Let $G = A* B$, where $A,B$ are either Artin-Tits groups or Coxeter groups. Then, $G$ is a Shephard group and the standard coset incidence system $\Gamma = (G,(G_i)_{i\in I})$ associated to $G$ is a flag-transitive, residually connected and firm geometry. 
\end{thm}

\begin{proof}
     Let $\alpha=(A,(A_i)_{i\in I_\alpha})$ and $\beta=(B,(B_j)_{j\in I_\beta})$ be, respectively, the standard coset geometries of $A$ and $B$, as in Theorem~\ref{thm:Artin_Coxeter_CG}. Then, Theorem~\ref{thm:Artin_Coxeter_CG} guarantees that $A$ and $B$ are flag-transitive on $\alpha$ and $\beta$, respectively. Moreover, we know that $\alpha$ and $\beta$ are residually connected, and firm geometries. Proposition~\ref{prop:FTfree} and Theorem~\ref{thm:FP_a)FT_b)Geo_c)_RC_b)_FIRM} then show that $G$ is also flag-transitive on $\Gamma$, with $\Gamma$ being also a residually connected firm geometry. 
\end{proof}

We now point out some consequences of Theorem \ref{thm:FreeProductArtinCoxeter} that could be of interest to people studying Shephard groups.
\begin{coro} \label{coro:FreeProductArtinCoxeter}
    Let $G$ be a Shephard group generated by a set $S$ of generators. If $G=A*B$, where $A,B$ are either Artin-Tits or Coxeter groups, then, for $K,L,P\subseteq S$, the following statements hold:
    \begin{enumerate}
        \item $\langle K\rangle \cap \langle L\rangle = \langle K\cap L\rangle$.
        \item $\langle K\rangle\langle L\rangle\cap \langle K\rangle\langle P\rangle = \langle K\rangle(\langle L\rangle\cap \langle P\rangle)$.
        \item The simplicial complex $\mathcal{K}(G,(G_i)_{i\in I})$ is a flag-complex, with $I=\{1,\ldots,|S|\}$.
    \end{enumerate}
\end{coro}
\begin{proof}
Let $(A,S_A)$ and $(B,S_B)$ be two systems of either Artin-Tits or Coxeter groups and their respective sets of generators $S_A=\{a_i\mid i\in I_\alpha\}$ and $S_B=\{b_j\mid j\in I_\beta\}$.
    Theorem \ref{thm:FreeProductArtinCoxeter} shows that the standard coset incidence system $\Gamma=(G,(G_i)_{i\in I})$ associated to $G$ is residually connected and $G$ is flag-transitive on $\Gamma$, where $I=I_\alpha\sqcup I_\beta$ and $G_i= A_{\{i\}\cap I_\alpha}* B_{\{i\}\cap I_\beta}$.
    In general, for $J\subseteq I$, we have that $J=J_\alpha\sqcup J_\beta$ and $G_J = A_{J_\alpha}*B_{J_\beta}$, with $J_\alpha=J\cap I_\alpha$ and $J_\beta=J\cap I_\beta$. As shown in Theorem~\ref{thm:Artin_Coxeter_CG}, and since both $\alpha$ and $\beta$ are residually connected, then $$G_J = \langle a_i \mid i\in I_\alpha\setminus J_\alpha\rangle*\langle b_j \mid j\in I_\beta\setminus J_\beta\rangle = \langle a_i,b_j\mid i\in I_\alpha\setminus J_\alpha \wedge j\in I_\beta\setminus J_\beta\rangle.$$
    As the set of generators of $G$ is $S=S_A\sqcup S_B$, we have that for all $K\subseteq S$, $\langle K\rangle$ is a standard parabolic subgroup of $G$; that is, the subset of generators $K$ can be indexed by the subsets $J_\alpha$ and $J_\beta$ of $I_\alpha$ and $I_\beta$ such that $K=\{a_i,b_j\mid i\in J_\alpha\wedge j\in J_\beta\}$ such that $\langle K\rangle = G^{J_\alpha\sqcup J_\beta}$.

    Hence, part $(a)$ and $(b)$ are then consequences of Lemma~\ref{lem:RCequivs}(\ref{lem:RCequivs:itm:Intersection}) and Proposition~\ref{prop:FTequivs}(\ref{prop:FTequivs:itm:FTIntersectionOfProducts}).
    Finally, by Proposition~\ref{prop:FComp:equiv:FT}, item $(b)$ implies that the simplicial complex $\KK(G,(G_i)_{i\in I})$ is a flag-complex.

\end{proof}

\subsection{Twisting of Shephard groups}\label{sec:twist_shephard_groups}

As previously seen, Shephard groups have defining graphs, and symmetries of these graphs induce automorphisms of the associated group in a natural way. In this section, we make use of this simple observation to perform twisting on Shephard groups.
We begin by writing down a quick proof of the fact that labeled graph automorphisms of $\GG$ induce automorphisms of the group $A(\GG)$.

\begin{lemma}\label{lem:varphiGraphAutToGroupAut}
    There is an injective homomorphism $\varphi \colon \Aut(\GG) \to \Aut(A(\GG))$.
\end{lemma}

\begin{proof}
   Let $A = A(\GG)$ and let $\psi \in \Aut(\GG)$. Since $V(\GG)$ is the set of generators of $A$, any element $g \in A$ can be written as a word $g = w_1 \cdots w_n$, where $w_i \in V(\GG)$ for all $i = 1, \cdots, n$. We can thus set $\varphi(\psi)(g) := \psi(w_1) \cdots \psi(w_n)$. Since, by definition, $\varphi(\psi)$ sends the set of relators $R$ to itself, it immediately follows that $\varphi(\psi) \in \Aut(A)$. It remains to show that $\varphi$ is injective. This follows from the fact that if $\psi$ does not fix all vertices, then $\varphi(\psi)$ cannot be the identity on $A$.
\end{proof}

Let $\psi\in Aut(\GG)$ be any automorphism of $\GG$. This automorphism $\psi$ can be naturally extended to an automorphism on the type set $I_\alpha$, $\psi_{I_\alpha}$, by setting $\psi_{I_\alpha}(i) = j$ if and only if $\psi(a_i) = a_j$. We thus also have that $\psi$ acts on the set of standard maximal parabolic subgroups $(A_i)_{i \in I_\alpha}$,  which by Lemma~\ref{lem:lattice}, implies that each $\psi\in \Aut(\GG)$ is admissible, in the sense of Section~\ref{sec:SplitExt}.

Let $B\leq Aut(\GG)$ and consider a family $(B_i)_{i\in I_\beta}$ of subgroups of $B$, where $I_\beta$ is any index set.
We can then build the associated coset incidence system $\beta=(B,(B_i)_{i\in I_\beta})$.
Let $\varphi^B$ be the restriction of the homomorphism $\varphi$ to $B\leq Aut(\GG)$.
Whenever $\alpha$ is $(\beta,\varphi^B)$-admissible, we can build the twisting incidence system $\TT(\alpha,\beta)$. Before going further, we give one example in which $\alpha$ is indeed $(\beta,\varphi^B)$-admissible.

\begin{example}
    Consider the following properly labeled graph $\GG$.

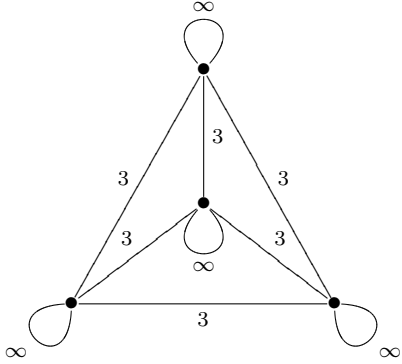
    $$\xymatrix@-1.5pc{
    &&&& *{\bullet}\ar@(ul,ur)@{-}[]^\infty \ar@{-}[dddd]^3\ar@{-}[dddddddrrrr]^3\ar@{-}[dddddddllll]_3 &&\\
    \\
    \\
    \\
    && &&*{\bullet}\ar@(dl,dr)@{-}[]_\infty &&\\
    \\
    \\
    *{\bullet}\ar@(l,d)@{-}[]_\infty\ar@{-}[rrrrrrrr]_3*\ar@{-}[rrrruuu]^3 && &&&& &&*{\bullet}\ar@(r,d)@{-}[]^\infty\ar@{-}[lllluuu]_3}$$

    The group $A=A(\GG)$ is the Artin-Tits group with braid relations $a_ia_ja_i=a_ja_ia_j$, for any $a_i,a_j\in V(\GG)$ such that $a_i\neq a_j$.
    Setting $I_\alpha=\{1,2,3,4\}$, the group has the following Artin-Tits presentation:
    $$\langle a_1, a_2, a_3, a_4\mid a_ia_ja_i=a_j a_i a_j\textnormal{ for }i,j\in I_\alpha\textnormal{ such that }i\neq j\rangle.$$
    Let $\alpha$ be the standard coset incidence system $(A,(A_i)_{i\in I_\alpha})$, with $A_i = \langle a_j\mid j\in I_\alpha\setminus\{i\}\rangle$.
    
    The automorphism group of the labeled graph $\GG$ is $S_4$, the automorphism group of a tetrahedron.
    Let $B$ be the subgroup of $Aut(\GG)$ that fixes the vertex $a_1$. We have that $B\cong S_3 =\langle \psi_{23},\psi_{34}\rangle$ where
     $\psi_{ij}=(a_i,a_j)$, the permutation exchanging the vertex $a_i$ and the vertex $a_j$ while fixing the other two vertices of $\GG$.
    Let $B_{23} = \langle \psi_{34}\rangle$, $B_{34}=\langle \psi_{23}\rangle$, and consider the coset incidence system $\beta=(B,(B_j)_{j\in\{23,34\}})$.
    By extending the action of $B$ on the type set $I_\alpha$, we have that $B$ fixes the type $1$, while permuting the remaining types. In other words, the set of orbits of the type set $I_\alpha$ under $B$ is  $K=\{\{1\},\{2,3,4\}\}$. Finally, setting $F_0 = 2$ in the definition of admissibility, we can easily see that $O_M\cap O_N=O_{M\cap N}$, where $O_M$ is the orbit of $2$ under the action of the parabolic subgroup $B_{\{23,34\}\setminus M}$ and $M\subseteq \{23,34\}$. 
    Hence, $\alpha$ is $(\beta,\varphi^B)$-admissible.
    
    We can thus construct the twisting coset incidence system $\TT(\alpha,\beta)$ from group $G=A\rtimes_{\varphi^B} B$. 
    Since $A$ is an Artin-Tits group, by Theorem~\ref{thm:Artin_Coxeter_CG}, we have that $\alpha$ is a thick, residually connected coset geometry and that $A$ acts flag-transitively on $\alpha$.
    Additionally, as $B$ is a Coxeter group, $\beta$ is a regular hypertope, i.e. a thin residually connected coset geometry and flag-transitive geometry.
    Therefore, by Theorem~\ref{thm:Twisting_FT_Geo_RC_FRM}, we have that $\TT(\alpha,\beta)$ is a firm residually connected coset geometry, with $G$ acting flag-transitively on $\TT(\alpha,\beta)$. 
\end{example}

The exact same process can be applied to some quotients of Shephard groups. This is especially useful when we want to work with groups that resemble Shephard groups, but have a few additional relations that do not come from the underlying graph. That said, the twisting will work only when the quotients are carefully selected.

Let $H\leq A =A(\GG)$ and let $B\leq Aut(\GG)$. We say \emph{$H$ is closed under $B$} if for all $\psi\in B$ we have $\varphi(\psi)(H)= H$.

\begin{lemma}\label{lem:normalclosureClosedUnderGraphAut}
    Let $H\leq A =A(\GG)$ and let $B\leq Aut(\GG)$.
    If $H$ is closed under $B$, then the normal closure $\langle\langle H\rangle\rangle$ of $H$ is also closed under $B$.
\end{lemma}
\begin{proof}
    Let $H\leq A =A(\GG)$ and let $B\leq Aut(\GG)$.
    Suppose that $H$ is closed under $B$ and let $\langle\langle H\rangle\rangle = \langle a^{-1}ha \mid h\in H, a\in A\rangle$ be the normal closure of $H$ in $A$.
    Let $\psi\in B$. Then, $\varphi^B(\psi)(a^{-1}ha)=\psi(a^{-1})\psi(h)\psi(a)$.
    Notice that both $\psi(a^{-1})=\psi(a)^{-1}$ and $\psi(a)$ are inverses and group elements of $A$. Moreover, as $H$ is closed under $B$, we have that $\psi(h)\in H$. Therefore, for any element $a^{-1}ha\in \langle\langle H\rangle\rangle$, we have that $\psi(a^{-1})\psi(h)\psi(a)\in \langle\langle H\rangle\rangle$.
    Hence, for all $\psi\in B$, we have that $\varphi(\psi)(\langle\langle H\rangle\rangle)=\langle\langle H\rangle\rangle$.
\end{proof}

Being closed under $B$ is an essential condition. Hence, we explicitly show how to close any subgroup under the action of $B$. Given a subgroup $H\leq A$, the \emph{$B$-symmetrization of $H$} is the subgroup $$^B(H)= \langle\varphi(\psi)(h)\mid h\in H, \psi\in B \rangle.$$

\begin{lemma}\label{lem:B-symmHClosed}
    Given a subgroup $H\leq A=A(\GG)$, the subgroup $^B(H)$ is closed under graph automorphisms $B\leq \Aut(\GG)$.
\end{lemma}
\begin{proof}
    This is a direct consequence of the definition of the $B$-symmetrization of $H$.
\end{proof}

Let $B\leq \Aut(\GG)$ and, for a set $I_\beta$, consider a family of subgroups $(B_i)_{i\in I_\beta}$ and construct the coset incidence system $\beta=(B,(B_i)_{i\in I_\beta})$. We now show that taking the quotient by a normal subgroup closed under $B$ does not break $(\beta,\varphi^B)$-admissibility.

\begin{prop}\label{prop:QuotientAdmissible}
    Consider a properly labeled graph $\GG$ and $A=A(\GG)$ its Shephard group.
    Let $H\leq A$ be a subgroup closed under graph automorphisms $B$ and let $\tilde{A}=A/\langle\langle H\rangle\rangle$, $\tilde{A_i} = A_i/\langle\langle H\rangle\rangle$.
    If $\alpha$ is $(\beta,\varphi^B)$-admissible, then $\tilde{\alpha}=(\tilde{A},(\tilde{A}_i)_{i\in I_\alpha})$ is $(\beta,\varphi^B)$-admissible.
\end{prop}
\begin{proof}
    Suppose that $\alpha$ is $(\beta,\varphi^B)$-admissible. Hence, 
    as $\varphi^B$ permutes the maximal parabolic subgroups $(A_i)_{i\in I_\alpha}$, we also have that $\varphi^B$ permutes the maximal parabolic subgroups $(\tilde{A}_i)_{i\in I_\alpha}$. Indeed, 
    for $\psi\in B$, we have that $\varphi^B(\psi)(\tilde{A_i})= \langle \psi(a_j \langle\langle H\rangle\rangle) \mid j\in I_\alpha\setminus\{i\}\rangle$. As $H$ is a subgroup closed under $B$, by Lemma~\ref{lem:normalclosureClosedUnderGraphAut}, we obtain that $\psi(a_j \langle\langle H\rangle\rangle) = \psi(a_j)\langle\langle H\rangle\rangle$ and thus that 
    $\varphi^B(\psi)(\tilde{A_i}) = \langle a_j\langle\langle H\rangle\rangle \mid j\in I_\alpha\setminus\psi_{I_\alpha}(i)\rangle$. In conclusion, the action on the standard maximal parabolic subgroups, and thus on the type set $I_\alpha$, is unchanged. This shows that the action itself is unchanged, and thus the admissibility of $\alpha$ is inherited by $\tilde{\alpha}$.

\end{proof}

We just showed that taking quotients by a normal subgroup closed under $B$ does not break $(\beta,\varphi^B)$-admissibility. Nonetheless, it does in general break other properties such as flag-transitivity, residual connectedness or even firmness.

\begin{example}
    Consider the following properly labeled graph $\GG$.

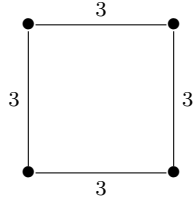
    $$\xymatrix@-1pc{
     *{\bullet}\ar@{-}[rrr]^3 &&& *{\bullet} \\
     \\
     \\
     *{\bullet}\ar@{-}[uuu]^3\ar@{-}[rrr]_3  &&& *{\bullet}\ar@{-}[uuu]_3} $$

    The group $A=A(\GG)$ is the Coxeter group $\tilde{A_3}$ (the Coxeter group with bracket notation $[3^{4}]$).
    Considering $I_\alpha=\{1,2,3,4\}$, the Coxeter presentation of $A$ is 
    \begin{align*}
        A= \langle& a_1,a_2,a_3,a_4\mid  a_i^2=1 \textnormal{ for } i\in I_\alpha, a_1a_4 = a_4a_1, a_2a_3=a_3a_2, \\ &a_1a_2a_1=a_2a_1a_2, a_1a_3a_1=a_3a_1a_3, a_2a_4a_2=a_4a_2a_4, a_3a_4a_3=a_4a_3a_4\rangle
    \end{align*}
    
    The full automorphism group of $\GG$ is $D_4$, the dihedral group of order $8$. Considering $B\cong D_4$, we have that $B\cong \langle \psi_x,\psi_y \rangle$ can be expressed as a permutation group where $\psi_x=(a_1,a_2)(a_3,a_4)$ and $\psi_y=(a_2,a_3)$.
    Set $B_x=\langle \psi_y\rangle$ and $B_y = \langle \psi_x\rangle$.
    Given $\alpha = (A,(A_i)_{i\in I_\alpha})$ and $\beta=(B,(B_i)_{i\in \{x,y\}})$, it is straightforward  to show that $\alpha$ is $(\beta,\varphi)$-admissible. Note that in this case, $B$ acts transitively on $I_\alpha$. Additionally, as $(B,\{\psi_x,\psi_y\})$ is also a Coxeter system, by Theorem \ref{thm:Artin_Coxeter_CG}, we have that both $\alpha$ and $\beta$ are regular hypertopes. Therefore, by Corollary~\ref{coro:TwistingHypertope} we have that the Twisting coset incidence system $\TT(\alpha,\beta)$ is a regular hypertope.

    Consider now the subgroup $H=\langle (a_3a_4a_1a_2)^{2s}\rangle$, for some $s\geq 1$. Let $u=(a_3a_4a_1a_2)^{2s}$.
    Clearly, this subgroup is not closed under $B$, as $\varphi(\psi_x)(u)=(a_4a_3a_2a_1)^{2s}\notin H$.
    Even though $H$ is not closed under $B$, we have that its normal closure $\langle\langle H\rangle\rangle$ is. Indeed, for any element $g^{-1}ug\in \langle\langle H\rangle\rangle$, with $g\in A$, we have $$\varphi(\psi_x)(g^{-1}ug)= \psi_x(g)^{-1}(a_4a_3a_2a_1)^{2s}\psi_x(g)=\psi_x(g)^{-1} a_4a_3(u)^{-1}a_3a_4\psi_x(g)\in \langle\langle H\rangle\rangle,$$ 
    and 
    $$\varphi(\psi_y)(g^{-1}ug)= \psi_y(g)^{-1}(a_2a_4a_1a_3)^{2s}\psi_y(g)= \psi_y(g)^{-1}a_3a_2(u)a_2a_3\psi_y(g)\in \langle\langle H\rangle\rangle.$$
    
    We can thus consider $\tilde{A}=A/\langle\langle H\rangle\rangle$, $\tilde{A_i}=A_i/\langle\langle H\rangle\rangle$ and $\tilde{\alpha}=(\tilde{A},(\tilde{A}_i)_{i\in I_\alpha})$. By Proposition~\ref{prop:QuotientAdmissible}, we have that $\tilde{\alpha}$ is $(\beta,\varphi)$-admissible. 
    Note that when $s=1$, $\tilde{\alpha}$ is not a regular hypertope, as it fails to be residually connected, for example. Therefore, when $s=1$, the twisting incidence system $\TT(\tilde\alpha,\beta)$ is not a regular hypertope. Hence, one must be aware that properties of the incidence system $\alpha$ do not automatically transfer to $\tilde{\alpha}$.
    Nonetheless, by \cite[Theorem 5.1]{ens2018rank}, we know that $\tilde{\alpha}$ is a regular hypertope whenever $s \geq 2$. Therefore, for $s\geq 2$, the twisting coset incidence system $\TT(\tilde\alpha,\beta)$ is a regular hypertope.
\end{example}

\subsection{Twisting and Amalgamation of Artin-Tits groups}\label{sec:TwistArtin}

We now focus on Artin-Tits groups and explore more applications of twistings and amalgamations for these groups.
To start, we show how to obtain the Shephard group $B_n(2,\infty)$, associated with Mikado braids, by twisting an Artin-Tits groups. This can then be used to show some intersection properties of parabolic subgroups of the group $B_n(2,\infty)$.
Consider the Artin-Tits group $A$ of type $D_n$, defined by the following graph $\GG$.
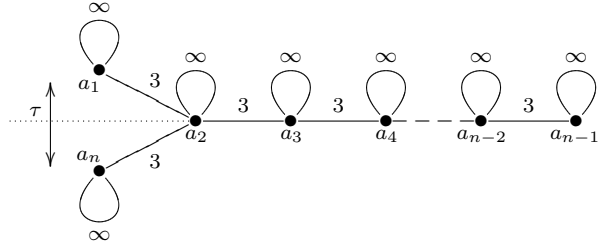
$$\xymatrix@-1pc{
&& *{\bullet} \ar@{-}[rrd]^3_(0.01){a_1}\ar@(ul,ur)@{-}[]^\infty\\
\ar@{.}[rrrr] & *{}\ar@[->][u]^(0.2){\tau}\ar@[->][d]& && *{\bullet}\ar@(ul,ur)@{-}[]^\infty \ar@{-}[rr]^3_(0.01){a_2}&& *{\bullet}\ar@(ul,ur)@{-}[]^\infty  \ar@{-}[rr]^3_(0.01){a_3} && *{\bullet}\ar@(ul,ur)@{-}[]^\infty\ar@{--}[rr]_(0.01){a_4} && *{\bullet}\ar@(ul,ur)@{-}[]^\infty \ar@{-}[rr]^3_(0.01){a_{n-2}}_(0.99){a_{n-1}} && *{\bullet}\ar@(ul,ur)@{-}[]^\infty \\
&& *{\bullet} \ar@{-}[rru]_3^(0.01){a_n} \ar@(dl,dr)@{-}[]_\infty
}$$
Setting $I_\alpha=\{1,\ldots,n\}$, we have that  $A=\langle S\mid R\rangle$, where $S=\{a_i\mid i\in I_\alpha\}$ and 
$$R=\{a_1a_na_1^{-1}a_n^{-1}, a_na_3a_n(a_3a_na_3)^{-1}\}\cup \{a_ia_{i+1}a_i(a_{i+1}a_ia_{i+1})^{-1}\mid i,j\in I_\alpha\setminus\{n\}\}$$ $$\cup \{a_ia_j(a_ja_i)^{-1} \mid i,j\in I_\alpha \textnormal{ s.t. } |i-j|>1 \wedge \{i,j\}\neq\{3,n\}\}.$$
Start by building the coset incidence system $\alpha=(A,(A_i)_{i\in I_\alpha})$, where $A_i = \langle a_j\mid j\in I_\alpha\setminus\{i\}\rangle$.

Let $B=\langle \tau \rangle\leq \Aut(\GG)$, where $\tau$ is the graph automorphism of order two shown on the graph.
Consider $\beta=(B,(B_i)_{i\in I_\beta})$, where $|I_\beta|=1$. Setting the type set as $I_\beta=\{0\}$, we have that $B_0=\langle 1_B\rangle$, and $\beta= (\langle\tau\rangle,(\langle 1_B\rangle))$. Moreover, $\beta$ is a rank 1 coset geometry, hence it is trivially residually connected and flag-transitive. 

It is easy to check that $\alpha$ is $(\beta,\varphi)$-admissible. In particular, $\varphi(\tau)(A_1)=A_n$, $\varphi(\tau)(A_n)=A_1$ and $\varphi(A_i)=A_i$, for $2\leq i\leq n-1$. We have that the orbit on the type set of $I_\alpha$ is $K=\{\{1,n\},\{2\},\{3\},\ldots,\{n-1\}\}$, with $L=\{1,n\}$, and, defining $F_0=1$, it respects the orbit intersection property.
Hence, for $I=K\cup\{0\}$, we can build the twisting geometry $\TT(\alpha,\beta)=(G,(G_i)_{i\in I})$, where $G=A\rtimes_\varphi \langle \tau\rangle$, with subgroups $G_i= A_i\rtimes \langle \tau \rangle $, for $i\in K$, and $G_0 = A_n$ (notice that the orbit $O^0 = \{F_0\}$).
In $G$, we have that $\tau^{-1}a\tau=\varphi(\tau)(a)$ for an arbitrary element $a\in A$, which in turn is equivalent to $\tau a\tau$, as $\tau^2=1_G$.
Hence, we obtain the following result.
\begin{prop}\label{prop:Artin_to_MikadoBraid}
    Let $A$ be the Artin-Tits group of type $D_n$, $B$ be a group of order two, seen as a Coxeter group of rank one acting on $A$, and let $G=A\rtimes_\varphi \langle \tau\rangle$. Then $G$ is isomorphic to the Shephard group $B_n(2,\infty)$.
\end{prop}
\begin{proof}
    Let $G=A\rtimes_\varphi B$, where $A=\langle S\mid R\rangle$ has the presentation as above and $B=\langle \tau\mid \tau=\tau^{-1}\rangle$. Then, we have that the presentation of $G$ is
    $$G = \langle S\cup \{\tau\} \mid R, \tau^2=1_G, \tau a_1\tau = \varphi(\tau)(a_1),$$ 
    $$\tau a_n\tau = \varphi(\tau)(a_n), \tau a_i \tau = \varphi(\tau)(a_i),\textnormal{for }2\leq i\leq n-1\rangle.$$
    
    Notice that we can write the generator $a_n$ as $\tau a_1\tau$ and that this implies that $(a_1\tau)_4=(\tau a_1)_4$, $(\tau a_i)_2=(a_i\tau)_2$, for $2\leq i\leq n-1$, and $(a_na_i)_l=(a_ia_n)_l$ if and only if $ \tau(a_1a_i)_l\tau=\tau(a_ia_1)_l\tau$, with $l=3$ for $i=2$ and $l=2$ for $3\leq i\leq n-1$. Using Tietze transformations (see \cite{magnus2004combinatorial} for a precise definition), we can remove the generator $a_n$ from $G$, and rewrite all relations in the presentation so that we get
    $G = \langle \bar{S}\mid \bar{R}\rangle$
    with $\bar{S} = \{ \tau,a_1,a_2,a_3,\ldots,a_{n-1}\}$ and 
    $$\bar{R}=\{(a_1\tau)_4 (\tau a_1)_4^{-1}\}\cup \{(a_ia_{i+1})_3(a_{i+1}a_i)_3^{-1}\mid i,j\in I_\alpha\setminus\{n\}\}$$ $$\cup \{a_ia_j(a_ja_i)^{-1} \mid i,j\in I_\alpha\setminus\{n\} \textnormal{ s.t. } |i-j|>1 \}\cup \{(\tau a_i)_2(a_i\tau)_2^{-1}\mid 2\leq i \leq n-1 \}.$$
This is then the presentation associated to the following graph.
$$
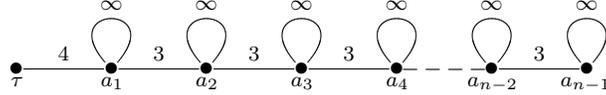
\xymatrix@-1pc{
 *{\bullet}\ar@{-}[rr]^4_(0.01){\tau} && *{\bullet}\ar@(ul,ur)@{-}[]^\infty \ar@{-}[rr]^3_(0.01){a_1} && *{\bullet}\ar@(ul,ur)@{-}[]^\infty \ar@{-}[rr]^3_(0.01){a_2}&& *{\bullet}\ar@(ul,ur)@{-}[]^\infty  \ar@{-}[rr]^3_(0.01){a_3} && *{\bullet}\ar@(ul,ur)@{-}[]^\infty\ar@{--}[rr]_(0.01){a_4} && *{\bullet}\ar@(ul,ur)@{-}[]^\infty \ar@{-}[rr]^3_(0.01){a_{n-2}}_(0.99){a_{n-1}} && *{\bullet}\ar@(ul,ur)@{-}[]^\infty
}$$
This shows that $G$ is indeed the Shephard group $B_n(2,\infty)$.
\end{proof}

Notice that, if we consider $I'= \{0\}\cup (I_\alpha\setminus\{n\}) = \{0,1,\ldots,n-1\}$ and $G$ with the presentation of the proof above, we still have that $G_0 = A_n$, $G_1 = A_{\{1,n\}}\rtimes\langle \tau\rangle$ and $G_i=A_i\rtimes\langle \tau\rangle$ for $2\leq i\leq n-1$. Hence, the coset incidence system of $G$ with this new presentation, $(G,(G_i)_{i\in I'})$, is isomorphic to the twisting geometry $\TT(\alpha,\beta)$.
\begin{prop}~\label{prop:Mikado_intersect_flagchmabercomplex}
    Let $G$ be the Shephard group $B_n(2,\infty)$, generated by $S=\{a_0,a_1,\ldots,a_{n-1}\}$. Then, for $J,K,L\subseteq S$, we have
    \begin{enumerate}
        \item $\langle J\rangle\cap \langle K\rangle = \langle J\cap K\rangle$;
        \item $\langle J\rangle\langle K\rangle\cap \langle J\rangle\langle L\rangle = \langle J\rangle(\langle K\rangle\cap \langle L\rangle)$.
        \item The simplicial complex $\mathcal{K}(G,(G_i)_{i\in I})$ is a flag-complex, where $I=\{0,\ldots,n-1\}$ and $G_i=\langle a_j\mid j\in I\setminus\{i\}\rangle$.
    \end{enumerate}
\end{prop}
\begin{proof}
    The group $G$, defined as in the proof of Proposition~\ref{prop:Artin_to_MikadoBraid}, yields a coset incidence system isomorphic to the twisting geometry $\TT(\alpha,\beta)$. Theorem \ref{thm:Twisting_FT_Geo_RC_FRM} then guarantees that this coset incidence system is a residually connected geometry on which $G$ acts flag-transitively. The three statements of the theorem are then consequences of various results detailed in Section \ref{sec:prelims}, following similar proofs as those in Corollary~\ref{coro:FreeProductArtinCoxeter}.
\end{proof}

Notice that, putting together the fact that the twisting geometry $T(\alpha,\beta)$ is residually connected and item (a) of the Proposition \ref{prop:Mikado_intersect_flagchmabercomplex}, it follows that, for any $J\subseteq \{0,\ldots,n-1\}$, we have $G_J=\langle a_i\mid i\in I\setminus J\rangle$ for the Shephard group $B_n(2,\infty)$. Hence, the parabolics subgroups of these groups as well behaved as those of Coxeter and Artin-Tits groups. This will be also the case for each of the groups we will investigate in the rest of this section.

We can generalize the process we just described. Indeed, as long as as there is a symmetry of the defining graph that fixes all generators except for a finite set, we can proceed in a similar fashion.
Let $A$ be the Artin-Tits group defined by the graph below, where $p_{i,j}\geq 2$ ($p_{i,j}\in\mathbb{N}$) or $p_{i,j}=\infty$, and the subgraph $\GG'$ with $V(\GG')=\{a_2,\ldots,a_{n-1}\}$ defines an Artin-Tits group $A' = \langle a_2,\ldots,a_{n-1}\mid R'\rangle$. Moreover, we require that $p_{1,i}=p_{n,i}$, so that we have a two-fold symmetry $\tau$ as before. 
   
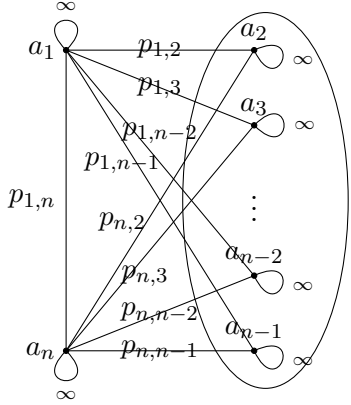
\begin{center}
         \begin{tikzpicture}[scale = 0.5]
         \filldraw[black] (2,-4) circle (2pt)  node[anchor=east]{$a_n$};
         \filldraw[black] (2,4) circle (2pt)  node[anchor=east]{$a_1$};
         \filldraw[black] (7,0.8) circle (0pt)  node[anchor=north]{$\vdots$};
         \filldraw[black] (7,2) circle (2pt)  node[anchor=south]{$a_3$};
         \filldraw[black] (7,-2) circle (2pt)  node[anchor=south]{$a_{n-2}$};
         \filldraw[black] (7,4) circle (2pt)  node[anchor=south]{$a_2$};
         \filldraw[black] (7,-4) circle (2pt)  node[anchor=south]{$a_{n-1}$};
         \draw [-] (2,-4) to [out=-135,in=-45,style={min distance=15mm}] node[below] {\tiny$\infty$} (2,-4);
         \draw [-] (2,4) to [out=135,in=45,style={min distance=15mm}] node[above] {\tiny$\infty$} (2,4);
         \draw [-] (7,2) to [out=-45,in=45,style={min distance=15mm}] node[right] {\tiny$\infty$} (7,2);
         \draw [-] (7,-2) to [out=-65,in=25,style={min distance=15mm}] node[right] {\tiny$\infty$} (7,-2);
         \draw [-] (7,4) to [out=-65,in=25,style={min distance=15mm}] node[right] {\tiny$\infty$} (7,4);
         \draw [-] (7,-4) to [out=-65,in=25,style={min distance=15mm}] node[right] {\tiny$\infty$} (7,-4);
         \draw (2,4) -- (2,-4)node [midway,left] (TextNode) {$p_{1,n}$};
         \draw (2,4) -- (7,2)node [midway] (TextNode) {$p_{1,3}$};
         \draw (2,-4) -- (7,2)node [midway, xshift=-0.2cm, yshift=-0.5 cm] (TextNode) {$p_{n,3}$};
         \draw (2,4) -- (7,4)node [midway] (TextNode) {$p_{1,2}$};
         \draw (2,-4) -- (7,4)node [midway, xshift=-0.5cm, yshift=-0.3 cm] (TextNode) {$p_{n,2}$};
         \draw (2,4) -- (7,-2)node[midway,xshift=0cm, yshift=0.4 cm] (TextNode) {$p_{1,n-2}$};
         \draw (2,-4) -- (7,-2)node[midway] (TextNode) {$p_{n,n-2}$};
         \draw (2,4) -- (7,-4)node[midway, xshift=-0.5cm, yshift=0.5cm] (TextNode) {$p_{1,n-1}$};
         \draw (2,-4) -- (7,-4)node[midway] (TextNode) {$p_{n,n-1}$};
         \draw (7.3,0) ellipse (2.2cm and 5cm);
    \end{tikzpicture}
    \end{center}

 We have the Artin-Tits group $A=\langle S\mid R\rangle$, where $S=\{a_i\mid i\in I_\alpha\}$, for $I_\alpha=\{1,\ldots,n\}$, and
$$R=R'\cup \{(a_1a_n)_{p_{1,n}}(a_na_1)_{p_{1,n}}^{-1}\mid \textnormal{if }p_{1,n}\neq\infty\}\cup$$
$$\cup\{(a_ia_j)_{p_{i,j}}=(a_ja_i)_{p_{i,j}}\mid i\in\{1,n\}\wedge 2\leq j\leq n-1\wedge p_{i,j}\neq\infty\}.$$

Let $B=\langle \tau\mid \tau^2=1_B\rangle$ and $\varphi: \Aut(\GG)\to \Aut(A)$ be the usual injective homomorphism. Notice that $\varphi(\tau)(a_1)=a_n$, $\varphi(\tau)(a_n)=a_1$ and, for $2\leq i \leq n-1$, $\varphi(\tau)(a_i)=a_i$.
Considering $\alpha=(A,(A_i)_{i\in I_\alpha})$ and $\beta=(\langle\tau\rangle,(\langle 1_B\rangle))$, we once again have that $\alpha$ is $(\beta,\varphi)$-admissible. Notice that the orbit of the type set is as before, $K = \{\{1,n\},\{2\},\ldots,\{n-1\}\}$.
Similar to Proposition~\ref{prop:Artin_to_MikadoBraid}, we can prove the following.

\begin{thm}\label{thm:twistArtinGroups_ShephardGroups}
    Let $A$ be the Artin-Tits group defined by the graph $\GG$ above, $B=\langle\tau \mid \tau^2=1_B\rangle$ and $\varphi$ defined as above. Then $G=A\rtimes_\varphi B$ is isomorphic to the Shephard group with presentation given by the following graph:
   
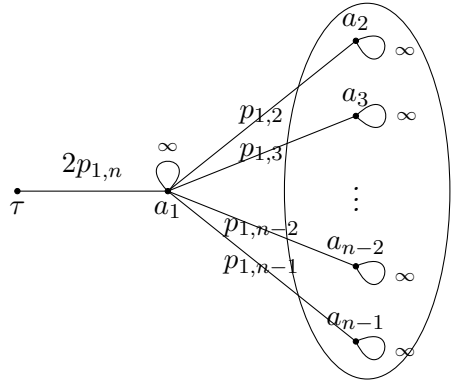
\begin{center}
    \begin{tikzpicture}[scale = 0.5]
    
   \filldraw[black] (-2,0) circle (2pt)  node[anchor=north]{$\tau$};
   \filldraw[black] (2,0) circle (2pt)  node[anchor=north]{$a_1$};
   \filldraw[black] (7,0.8) circle (0pt)  node[anchor=north]{$\vdots$};
    \filldraw[black] (7,2) circle (2pt)  node[anchor=south]{$a_3$};
         \filldraw[black] (7,-2) circle (2pt)  node[anchor=south]{$a_{n-2}$};
         \filldraw[black] (7,4) circle (2pt)  node[anchor=south]{$a_2$};
         \filldraw[black] (7,-4) circle (2pt)  node[anchor=south]{$a_{n-1}$};
         \draw [-] (2,0) to [out=135,in=45,style={min distance=15mm}] node[above] {\tiny$\infty$} (2,0);
         \draw [-] (7,2) to [out=-45,in=45,style={min distance=15mm}] node[right] {\tiny$\infty$} (7,2);
         \draw [-] (7,-2) to [out=-65,in=25,style={min distance=15mm}] node[right] {\tiny$\infty$} (7,-2);
         \draw [-] (7,4) to [out=-65,in=25,style={min distance=15mm}] node[right] {\tiny$\infty$} (7,4);
         \draw [-] (7,-4) to [out=-65,in=25,style={min distance=15mm}] node[right] {\tiny$\infty$} (7,-4);
    \draw (2,0) -- (-2,0)node [midway,above] (TextNode) {$2p_{1,n}$};
    \draw (2,0) -- (7,2)node [midway] (TextNode) {$p_{1,3}$};
    \draw (2,0) -- (7,4)node [midway] (TextNode) {$p_{1,2}$};
    \draw (2,0) -- (7,-2)node[midway] (TextNode) {$p_{1,n-2}$};
    \draw (2,0) -- (7,-4)node[midway] (TextNode) {$p_{1,n-1}$};
    \draw (7.3,0) ellipse (2.2cm and 5cm);

    \end{tikzpicture}
    \end{center}
\end{thm}
\begin{proof}
We proceed similarly to the proof of Proposition~\ref{prop:Artin_to_MikadoBraid}. We have that $A=\langle S\mid R\rangle$ with the above presentation and $B=\langle \tau\mid \tau^2=1_B\rangle$. Hence. the presentation of $G$ is
$$\langle S\cup \{\tau\} \mid R, \tau^2=1_B, \tau a_1\tau = \varphi(\tau)(a_1), \tau a_n\tau = \varphi(\tau)(a_n),$$
$$ \tau a_i \tau = \varphi(\tau)(a_i),\textnormal{for }2\leq i\leq n-1\rangle.$$
Using Tietze transformations, we can rewrite these relations without using $a_n$.
Notice that, for $2\leq i\leq  n-1$ we have $\varphi(\tau)(a_i)=a_i$, and thus $\tau a_i\tau=a_i$ if and only if $\tau a_i=a_i\tau$. This also implies that all relations $R'$ of $A'$ are unchanged by $\varphi(\tau)$. Additionally, since we have that $p_{1,i}=p_{n,i}$ for $2 \leq i \leq n-1$, we will prove that we can obtain all existant relations with $a_n$ from the existant relations with $a_1$ using $\tau$ (when $p_{1,i}=p_{n,i}\neq \infty$).
Since $a_n = \varphi(\tau)(a_1)=\tau a_1\tau$, consider the relation $(a_na_i)_{p_{n,i}}=(a_ia_n)_{p_{n,i}}$ with $p_{n,i}\neq\infty$ and $2\leq i\leq  n-1$. If $p_{n,i}$ is even, then  $$(a_na_i)_{p_{n,i}}= (a_na_i)^{p_{n,i}/2} = (\tau a_1 \tau a_i)^{p_{n,i}/2}= (\tau a_1  a_i\tau)^{p_{n,i}/2}=\tau( a_1  a_i)^{p_{n,i}/2}\tau.$$ Similarly we have that $(a_ia_n)_{p_{n,i}} = \tau (a_ia_1)^{p_{n,i}/2}\tau$. Hence, $$(a_na_i)_{p_{n,i}}=(a_ia_n)_{p_{n,i}}\Leftrightarrow \tau( a_1  a_i)^{p_{n,i}/2}\tau = \tau (a_ia_1)^{p_{n,i}/2}\tau\Leftrightarrow ( a_1  a_i)_{p_{n,i}} = (a_ia_1)_{p_{n,i}}.$$
When $p_{n,i}$ is odd, we can similarly derive that $(a_na_i)_{p_{n,i}}=(a_ia_n)_{p_{n,i}} \Leftrightarrow (a_1a_i)_{p_{n,i}}=(a_ia_1)_{p_{n,i}}$.
Notice that in either cases, when $p_{1,i}=p_{n,i}=\infty$, there is no relation defined in $R$. Finally, consider the relation $(a_1a_n)_{p_{1,n}}=(a_na_1)_{p_{1,n}}$. As $\varphi(\tau)(a_1)=a_n$ and $a_1^\tau=\tau a_1 \tau = \varphi(\tau)(a_1)$, we can rewrite the relation above as $(a_1 a_1^\tau )_{p_{1,n}}= (a_1^\tau a_1)_{p_{1,n}}$. Notice that, when $p_{1,n}$ is even, we have immediately that $(a_1\tau)_{2{p_{1,n}}}=(\tau a_1)_{2{p_{1,n}}}$. Consider then the case when $p_{1,n}$ is odd.
In this case we have that $(a_1 a_1^\tau)_{p_{1,n}} = (a_1 a_1^\tau)_{{(p_{1,n}})-1} a_1 = (a_1\tau)_{2(({p_{1,n}})-1)}a_1$ and $ (a_1^\tau a_1)_{p_{1,n}} = (a_1^\tau a_1)_{{(p_{1,n})}-1} a_1^\tau = (\tau a_1)_{2({(p_{1,n})}-1)} \tau a_1 \tau $. As $(a_1 a_1^\tau )_{p_{1,n}}= (a_1^\tau a_1)_{p_{1,n}}$ and $\tau=\tau^{-1}$, we have that $(a_1\tau)_{2({(p_{1,n})}-1)}a_1 = (\tau a_1)_{2({(p_{1,n})}-1)} \tau a_1 \tau \Leftrightarrow (a_1\tau)_{2{p_{1,n}}}=(\tau a_1)_{2{p_{1,n}}}$, as we wanted.
Finally, when ${p_{1,n}}=\infty$, this relation doesn't exist.
With this and using Tietze transformations, we can easily remove the generator $a_n$ from $G$, and rewrite all relations presentation as
    $G = \langle \bar{S}\mid \bar{R}\rangle$
    with $\bar{S} = \{ \tau,a_1,a_2,a_3,\ldots,a_{n-1}\}$ and 
    $$\bar{R}=R'\cup \{(a_1\tau)_{2p_{1,n}} (\tau a_1)_{2p_{1,n}}^{-1}\mid \textnormal{ if }p_{1,n}\neq \infty\}\cup$$
    $$\cup \{(a_1a_i)_{p_{1,i}}=(a_ia_1)_{p_{1,i}}\mid 2\leq i\leq n-1\wedge p_{1,i}\neq\infty\}$$
which can be represented by the graph in the statement of the theorem. 
\end{proof} 

The previous Proposition~\ref{prop:Artin_to_MikadoBraid} is a particular case of Theorem~\ref{thm:twistArtinGroups_ShephardGroups}, where $A'$ is the Artin-Tits group of type $A_{n-2}$, $p_{1,i}=p_{n,i}=2$, for $3\leq i\leq n-1$, $p_{1,2}=p_{n,2}=3$ and $p_{1,n}=2$.

As before, the maximal parabolic subgroups of $G$ are unchanged by the redefinition of the generators of $G$, i.e. considering $I'=\{0\}\cup(I_\alpha\setminus\{n\})=\{0,1,\ldots,n-1\}$, we have $G_0=A_{n}$, $G_1 = A_{1,n}\rtimes \langle \tau\rangle$ and $G_i=A_i\rtimes\langle\tau\rangle$ for $2\leq i\leq n-1$. Hence, the coset incidence system of $G$ with the presentation given by the above theorem is isomorphic to the twisting geometry $\TT(\alpha,\beta)$.

\begin{thm}
    Let $G$ be the Shephard group with presentation given by the following graph, where $p_{i,j}\in\{2,3,\ldots\}\cup\{\infty\}$.
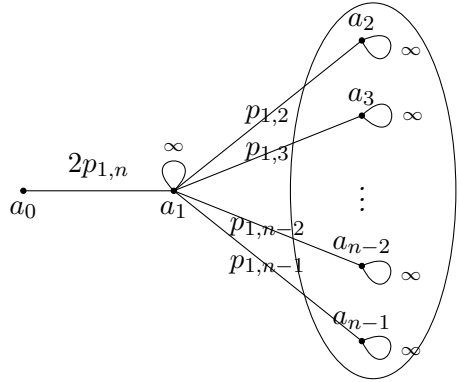
\begin{center}
    \begin{tikzpicture}[scale = 0.5]
    
   \filldraw[black] (-2,0) circle (2pt)  node[anchor=north]{$a_0$};
   \filldraw[black] (2,0) circle (2pt)  node[anchor=north]{$a_1$};
   \filldraw[black] (7,0.8) circle (0pt)  node[anchor=north]{$\vdots$};
    \filldraw[black] (7,2) circle (2pt)  node[anchor=south]{$a_3$};
         \filldraw[black] (7,-2) circle (2pt)  node[anchor=south]{$a_{n-2}$};
         \filldraw[black] (7,4) circle (2pt)  node[anchor=south]{$a_2$};
         \filldraw[black] (7,-4) circle (2pt)  node[anchor=south]{$a_{n-1}$};
         \draw [-] (2,0) to [out=135,in=45,style={min distance=15mm}] node[above] {\tiny$\infty$} (2,0);
         \draw [-] (7,2) to [out=-45,in=45,style={min distance=15mm}] node[right] {\tiny$\infty$} (7,2);
         \draw [-] (7,-2) to [out=-65,in=25,style={min distance=15mm}] node[right] {\tiny$\infty$} (7,-2);
         \draw [-] (7,4) to [out=-65,in=25,style={min distance=15mm}] node[right] {\tiny$\infty$} (7,4);
         \draw [-] (7,-4) to [out=-65,in=25,style={min distance=15mm}] node[right] {\tiny$\infty$} (7,-4);
    \draw (2,0) -- (-2,0)node [midway,above] (TextNode) {$2p_{1,n}$};
    \draw (2,0) -- (7,2)node [midway] (TextNode) {$p_{1,3}$};
    \draw (2,0) -- (7,4)node [midway] (TextNode) {$p_{1,2}$};
    \draw (2,0) -- (7,-2)node[midway] (TextNode) {$p_{1,n-2}$};
    \draw (2,0) -- (7,-4)node[midway] (TextNode) {$p_{1,n-1}$};
    \draw (7.3,0) ellipse (2.2cm and 5cm);

    \end{tikzpicture}
    \end{center}
    For $S=\{ a_0, a_1,\ldots,a_{n-1}\}$ and $J,K,L\subseteq S$, we have
    \begin{enumerate}
        \item $\langle J\rangle\cap \langle K\rangle = \langle J\cap K\rangle$;
        \item $\langle J\rangle\langle K\rangle\cap \langle J\rangle\langle L\rangle = \langle J\rangle(\langle K\rangle\cap \langle L\rangle)$.
        \item The simplicial complex $\mathcal{K}(G,(G_i)_{i\in I})$ is a flag-complex, where $I=\{0,\ldots,n-1\}$ and $G_i=\langle a_j\mid j\in I\setminus\{i\}\rangle$.
    \end{enumerate}
\end{thm}
\begin{proof}
    The proof is identical to the one of Proposition~\ref{prop:Mikado_intersect_flagchmabercomplex}
\end{proof}

Finally, we will consider the more general case where the symmetry of the graph is not an involution. Let $I=\{1,\ldots,n-1\}$ and $A$ be an Artin-Tits group generated by $S=\{a_i\mid i\in I\}$ and, for $k\geq 1$, let $$*^k_{A_1}A= \overbrace{A*_{A_1}A*_{A_1}A*_{A_1}\ldots*_{A_1}A}^{k \textnormal{ times}},$$ be the free amalgamated product of $k$ copies of $A$, where each amalgamation is over the subgroup $A_1=\langle a_j\mid I\setminus\{1\}\rangle$. Note that when $k=1$, $*^1_{A_1}A = A$ and, when $k=2$, we get the usual free product with amalgamation $*^2_{A_1}A = A*_{A_1}A$.
For now on, consider $k\geq 2$.
Hence, the presentation of $*^k_{A_1}$ is $\langle S\cup\{a_n,a_{n+1},\ldots,a_{n+k-2}\}\mid R\rangle$ with
$$R = R_{A_1}\cup \{(a_ia_j)_{p_{i,j}}=(a_ia_j)_{p_{i,j}}\mid i\in P \wedge j\in Q\wedge p_{i,j}=p_{1,j}\neq\infty\},$$
where $R_{A_1}$ are the relations of the subgroup $A_1$, $P=\{1,n,n+1,\ldots,n+k-2\}$ and $Q=\{2,3,\ldots,n-1\}$. As $A_1$ is the amalgamated subgroup, we have that all $p_{i,j}=\infty$, for $i,j\in P$.
Due to space limitations, we will not draw the graph $\GG$ of the presentation of $*^k_{A_1}A$, but we direct the readers' attention to the fact that, as before, this graph will have a symmetry on the vertices $\{a_1,a_n,\ldots,a_{n+k-2}\}$. In particular, we will consider a cyclic permutation of these vertices.
Let $B=\langle \tau\mid \tau^k=1_B\rangle\leq \Aut(\GG)$ and consider $\varphi: \Aut(\GG)\to \Aut(*^k_{A_1}A)$ such that $\varphi(\tau)(a_1)=a_n$, $\varphi(\tau)(a_{n+k-2})=a_1$ and, for $0\leq j\leq k-3$, $\varphi(\tau)(a_{n+j})=a_{n+j+1}$. As the previous cases, $\alpha$ is $(\beta,\varphi)$-admissible, with $K=\{\{1,n,n+1,\ldots, n+k-2\},\{2\},\{3\},\ldots,\{n-1\}\}$

\begin{thm}\label{thm:Amalg_twistArtinGroups_ShephardGroups} For $k\geq 2$, let $*^k_{A_1}A$ be the free amalgamated product of $k$ copies of the Artin-Tits group $A$ where each amalgamation is performed over the subgroup $A_1=\langle a_j\mid I\setminus\{1\}\rangle$. Let $\GG$ be the defining graph of $*^k_{A_1}A$ and $B=\langle\tau \mid \tau^k=1_B\rangle$. Then $G=A\rtimes_\varphi B$ is isomorphic to the Shephard group with presentation given by the following graph:
   
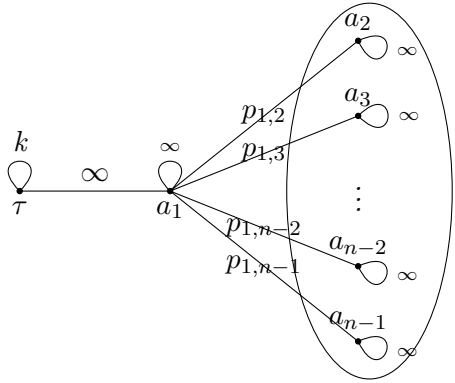
\begin{center}
    \begin{tikzpicture}[scale = 0.5]
    
   \filldraw[black] (-2,0) circle (2pt)  node[anchor=north]{$\tau$};
   \filldraw[black] (2,0) circle (2pt)  node[anchor=north]{$a_1$};
   \filldraw[black] (7,0.8) circle (0pt)  node[anchor=north]{$\vdots$};
    \filldraw[black] (7,2) circle (2pt)  node[anchor=south]{$a_3$};
         \filldraw[black] (7,-2) circle (2pt)  node[anchor=south]{$a_{n-2}$};
         \filldraw[black] (7,4) circle (2pt)  node[anchor=south]{$a_2$};
         \filldraw[black] (7,-4) circle (2pt)  node[anchor=south]{$a_{n-1}$};
         \draw [-] (-2,0) to [out=135,in=45,style={min distance=15mm}] node[above] {k} (-2,0);
         \draw [-] (2,0) to [out=135,in=45,style={min distance=15mm}] node[above] {\tiny$\infty$} (2,0);
         \draw [-] (7,2) to [out=-45,in=45,style={min distance=15mm}] node[right] {\tiny$\infty$} (7,2);
         \draw [-] (7,-2) to [out=-65,in=25,style={min distance=15mm}] node[right] {\tiny$\infty$} (7,-2);
         \draw [-] (7,4) to [out=-65,in=25,style={min distance=15mm}] node[right] {\tiny$\infty$} (7,4);
         \draw [-] (7,-4) to [out=-65,in=25,style={min distance=15mm}] node[right] {\tiny$\infty$} (7,-4);
    \draw (2,0) -- (-2,0)node [midway,above] (TextNode) {$\infty$};
    \draw (2,0) -- (7,2)node [midway] (TextNode) {$p_{1,3}$};
    \draw (2,0) -- (7,4)node [midway] (TextNode) {$p_{1,2}$};
    \draw (2,0) -- (7,-2)node[midway] (TextNode) {$p_{1,n-2}$};
    \draw (2,0) -- (7,-4)node[midway] (TextNode) {$p_{1,n-1}$};
    \draw (7.3,0) ellipse (2.2cm and 5cm);

    \end{tikzpicture}
    \end{center}
\end{thm}
\begin{proof}
    The proof follows similar arguments as the proof of Theorem~\ref{thm:twistArtinGroups_ShephardGroups}. We will just focus on the main difference, which is the fact that now $\tau$ is not an involution. Our goal is to rewrite all generators in $\{a_n,\ldots,a_{n+k-2}\}$ in terms of $a_1$ and $\tau$, and achieve that all relators involving these generators can be obtained from relations independent of $\{a_n,\ldots,a_{n+k-2}\}$, by conjugation by $\tau^{j+1}$, with $0\leq j\leq k-2$.
    First, notice that $R_{A_1}$ are unchanged by the action of $\tau$. Consider a relation $(a_1a_i)_{p_{1,i}}=(a_ia_1)_{p_{1,i}}$ of $G$, with $i\in Q$. Hence, we have a relation $(a_{n+j}a_i)_{p_{n+j,i}}=(a_ia_{n+j})_{p_{n+j,i}}$, with $0\leq j\leq k-2$ and $p_{1,i}=p_{n+j,i}$. We have that $a_{n+j}=\varphi(\tau^{j+1})(a_1)=\tau^{-(j+1)}a_1\tau^{j+1}$.
    We can then rewrite $(a_{n+j}a_i)_{p_{1,i}}$ as $(\tau^{-(j+1)}a_1\tau^{j+1}a_i)^{p_{1,i}/2} $, if $p_{1,i}$ is even, otherwise $(\tau^{-(j+1)}a_1\tau^{j+1}a_i)^{(p_{1,i}-1)/2}\tau^{-(j+1)}a_1\tau^{j+1}$.
    Since, for $i\in Q$, all $a_i$ commute with $\tau$ ($\tau^{-1}a_i\tau = a_i$), we have either 
    $$(\tau^{-(j+1)}a_1\tau^{j+1}a_i)^{p_{1,i}/2} =\tau^{-(j+1)}(a_1a_i)^{p_{1,i}/2} \tau^{j+1}=\tau^{-(j+1)}(a_1a_i)_{p_{1,i}} \tau^{j+1}$$
    or
    $$(\tau^{-(j+1)}a_1\tau^{j+1}a_i)^{(p_{1,i}-1)/2}\tau^{-(j+1)}a_1\tau^{j+1}= \tau^{-(j+1)}(a_1a_i)^{(p_{1,i}-1)/2}\tau^{j+1}\tau^{-(j+1)}a_1\tau^{j+1}$$ $$=\tau^{-(j+1)}(a_1a_i)^{(p_{1,i}-1)/2}a_1\tau^{j+1} = \tau^{-(j+1)}(a_1a_i)_{p_{1,i}} \tau^{j+1}.$$
    Hence, in either cases, $(a_{n+j}a_i)_{p_{1,i}}=\tau^{-(j+1)}(a_1a_i)_{p_{1,i}} \tau^{j+1}$. Similarly, we have that $(a_ia_{n+j})_{p_{1,i}}=\tau^{-(j+1)}(a_ia_1)_{p_{1,i}} \tau^{j+1}$. Then, $(a_{n+j}a_{i})_{p_{1,i}}=(a_ia_{n+j})_{p_{1,i}}\Leftrightarrow\tau^{-(j+1)}(a_1a_i)_{p_{1,i}} \tau^{j+1}=\tau^{-(j+1)}(a_ia_1)_{p_{1,i}} \tau^{j+1}$, i.e. we can obtain these relation from the relations involving $a_1$.
    
    Hence, we can represent $G$ as the group given by the graph presented in the statement of the theorem.
\end{proof}

\begin{thm}
    Let $G$ be the Shephard group with presentation given by the following graph, where $p_{i,j}\in\{2,3,\ldots\}\cup\{\infty\}$.
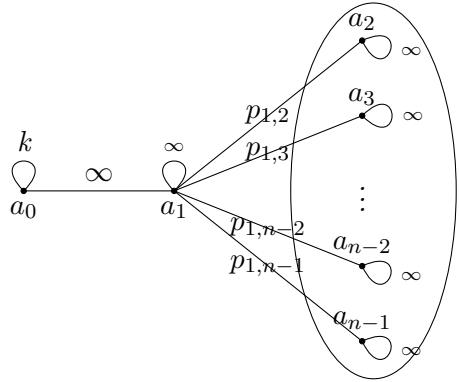
   \begin{center}
    \begin{tikzpicture}[scale = 0.5]
    
   \filldraw[black] (-2,0) circle (2pt)  node[anchor=north]{$a_0$};
   \filldraw[black] (2,0) circle (2pt)  node[anchor=north]{$a_1$};
   \filldraw[black] (7,0.8) circle (0pt)  node[anchor=north]{$\vdots$};
    \filldraw[black] (7,2) circle (2pt)  node[anchor=south]{$a_3$};
         \filldraw[black] (7,-2) circle (2pt)  node[anchor=south]{$a_{n-2}$};
         \filldraw[black] (7,4) circle (2pt)  node[anchor=south]{$a_2$};
         \filldraw[black] (7,-4) circle (2pt)  node[anchor=south]{$a_{n-1}$};
         \draw [-] (-2,0) to [out=135,in=45,style={min distance=15mm}] node[above] {k} (-2,0);
         \draw [-] (2,0) to [out=135,in=45,style={min distance=15mm}] node[above] {\tiny$\infty$} (2,0);
         \draw [-] (7,2) to [out=-45,in=45,style={min distance=15mm}] node[right] {\tiny$\infty$} (7,2);
         \draw [-] (7,-2) to [out=-65,in=25,style={min distance=15mm}] node[right] {\tiny$\infty$} (7,-2);
         \draw [-] (7,4) to [out=-65,in=25,style={min distance=15mm}] node[right] {\tiny$\infty$} (7,4);
         \draw [-] (7,-4) to [out=-65,in=25,style={min distance=15mm}] node[right] {\tiny$\infty$} (7,-4);
    \draw (2,0) -- (-2,0)node [midway,above] (TextNode) {$\infty$};
    \draw (2,0) -- (7,2)node [midway] (TextNode) {$p_{1,3}$};
    \draw (2,0) -- (7,4)node [midway] (TextNode) {$p_{1,2}$};
    \draw (2,0) -- (7,-2)node[midway] (TextNode) {$p_{1,n-2}$};
    \draw (2,0) -- (7,-4)node[midway] (TextNode) {$p_{1,n-1}$};
    \draw (7.3,0) ellipse (2.2cm and 5cm);

    \end{tikzpicture}
    \end{center}
    For $S=\{ a_0, a_1,\ldots,a_{n-1}\}$ and $J,K,L\subseteq S$, we have
    \begin{enumerate}
        \item $\langle J\rangle\cap \langle K\rangle = \langle J\cap K\rangle$;
        \item $\langle J\rangle\langle K\rangle\cap \langle J\rangle\langle L\rangle = \langle J\rangle(\langle K\rangle\cap \langle L\rangle)$.
        \item The simplicial complex $\mathcal{K}(G,(G_i)_{i\in I})$ is a flag-complex, with $I=\{0,\ldots,n-1\}$, where $G_i=\langle a_j\mid j\in I\setminus\{i\}\rangle$.
    \end{enumerate}
\end{thm}
\begin{proof}
    The proof is identical to the one of Proposition~\ref{prop:Mikado_intersect_flagchmabercomplex}
\end{proof}

\subsection{Graph of Coset Incidence Systems}\label{subsec:graphofcosetgeom}

We conclude by highlighting a promising direction for the operations defined in this paper.
As previously stated, amalgamations and HNN-extensions are the two building blocks of Bass-Serre theory. Since we now have analogous constructions for amalgams and HNN-extensions in the settings of coset incidence systems, it seems natural to wonder whether more concepts of Bass-Serre theory can be applied to these structures. In this section, we propose a possible definition of \textit{graph of coset incidence systems}. This is a hands-on, and perhaps naive, approach to the problem. We believe a more detailed investigation of the interplay between Bass-Serre theory and coset incidence systems is certainly worthwhile and necessary.

We first briefly recall Serre's formalism for graphs, which is slightly different from the one we have used elsewhere. In this section, a \textit{graph} $\mathcal{G}$ consists of a vertex set $V$, an edge set $E$, a pairing of each edge $e$ with a reverse edge $\bar{e}$ such that $e \neq \bar{e}$ and finally an origin map $o \colon E \to V$. The vertex $o(e)$ is called the origin vertex for $e$ and the vertex $o(\bar{e})$ is called the terminal vertex of $e$. We now define a version of graph of groups for coset incidence systems. A \textit{graph of coset incidence systems} $\mathbf{G}$ consists of the following data:

\begin{enumerate}
    \item A graph $\mathcal{G}$,
    \item An assignment of a \textit{vertex coset incidence system} $\alpha_v = (G^v, (G^v_i)_{i \in I_{\alpha_v}})$ to each vertex $v \in \mathcal{G}$,
    \item An assignment of an \textit{edge coset incidence system} $\alpha_e = (G^e, (G^e_i)_{i \in I_{\alpha_e}})$ to every edge $e \in \mathcal{G}$ such that $\alpha_e = \alpha_{\bar{e}}$,
    \item \textit{Boundary monomorphisms} $\varphi_e \colon G^e \to G^{o(e)}$ for all edge $e \in \mathcal{G}$ such that each $\varphi_e$ is an injective group homomorphism such that $\varphi_e \circ \varphi_{\bar{e}}^{-1}$ is an admissible automorphism between $\alpha_{o(e)}$ and $\alpha_{o(\bar{e})}$.
\end{enumerate}

Notice that in Section \ref{sec:Amalgams+HNN}, in the case of amalgamations, for simplicity of notation, we got rid of the defining isomorphism $\varphi$ by considering both the groups $A$ and $B$ as subgroups of $G = A *_\varphi B$. Therefore, instead of a notion of admissibility for $\varphi$, we defined a notion of compatibility of $\alpha$ and $\beta$. Equivalently, we could just as well have said that $\varphi: C \to D$ is admissible if it induces a lattice automorphism between parabolic subgroups of $\alpha$ inside of $C=A_{I_\alpha\setminus J}$ and parabolic subgroups of $\beta$ inside of $D=B_{I_\beta\setminus K}$, where $A_{I_\alpha\setminus J}\cong B_{I_\beta\setminus K}$. This is the notion we use in the definition of graph of coset incidence systems.
Moreover, the definition of the equivalence relation can be translated into this context. In particular, for $i_1\in J$ and $i_2\in K$, we have that $i_1\sim_\varphi i_2$ if and only if $\varphi(A_{I_\alpha\setminus J \cup \{i_1\}})=B_{I_\beta\setminus K \cup \{i_2\}}$ or  $\varphi^{-1}(B_{I_\beta\setminus K \cup \{i_2\}}) = A_{I_\alpha\setminus J \cup \{i_1\}}$. It can be easily checked that the type set resulting from this equivalence relation, $\widetilde{I}$, gives exactly the same parabolic subgroups as the ones when we consider $A$ and $B$ as subgroups of $G$, sharing a type subset. In particular, in that context, for $L=I_\alpha\cap I_\beta$, the above equivalence relation simply gives a set of singletons for each type.

A key feature of graphs of groups is that they have a fundamental group, which is uniquely defined up to natural isomorphism. Loosely speaking, one first chooses an orientation and a spanning tree $T$ for $\mathcal{G}$ and then considers the group obtained by amalgamating all the vertex groups along the edge groups of the spanning tree. It then remains to perform HNN-extensions of that group using the edge groups of the edges not in $T$, respecting the chosen orientation. The group obtained does not depend, up to isomorphism, on the choice of spanning tree $T$. Does the same hold for graphs of coset incidence systems? Can we define a fundamental geometry of a graph of coset incidence systems? We leave this question open, but we present a non-trivial example in which everything works out as desired.

\begin{example}\label{example:graphcosetgeom}
   Let $\mathcal{G}$ be the cyclic graph with three vertices, where a choice of orientation is indicated by the directed black arrows. 

$$
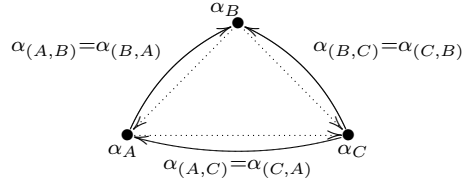
\xymatrix@-0.8pc{&& *{\bullet} && \\
\\
*{\bullet}\ar@/^/@{->}[rruu]^{\alpha_{(A,B)}=\alpha_{(B,A)}}^(0.99){\alpha_B} \ar@{<.}[rruu]
\ar@/_/@{<-}[rrrr]_{\alpha_{(A,C)}=\alpha_{(C,A)}}_(0.01){\alpha_A}_(0.99){\alpha_C} \ar@{.>}[rrrr]
&& &&*{\bullet}\ar@/_/@{->}[lluu]_{\alpha_{(B,C)}=\alpha_{(C,B)}} \ar@{<.}[lluu] \\
}$$
The three vertex coset incidence systems are $\alpha_A = (A,(A_i)_{i\in I_A})$, $\alpha_B = (B,(B_i)_{i\in I_\beta})$ and $\alpha_C = (C,(C_i)_{i\in I_\gamma})$, where the vertex groups $A=\langle a_1,a_2\rangle$, $B=\langle b_3,b_4\rangle$, $C=\langle  c_5,c_6 \rangle$ are free groups. The type sets are $I_A=\{1,2\}$, $I_B=\{3,4\}$ and $I_C=\{5,6\}$, respectively, and the maximal parabolical subgroups $A_i$ (resp. $B_i$ or $C_i$) are $\langle a_j\mid I_A\setminus\{i\}\rangle$ (resp. $\langle b_j\mid I_B\setminus\{i\}\rangle$ or $\langle c_j\mid I_C\setminus\{i\}\rangle$).
The groups $A$, $B$ and $C$ can be seen as Artin-Tits groups with no braid relations. Therefore, by Theorem~\ref{thm:ArtinProperties}, we have that for $J\subseteq I_A$, $A_J=\langle a_j \mid j\in I_A\setminus J\rangle$, and equivalently for the parabolics of $\alpha_B$ and $\alpha_C$. Also, by Theorem \ref{thm:Artin_Coxeter_CG}, the three incidence systems $\alpha_A, \alpha_B$ and $\alpha_C$ are residually connected, flag-transitive coset geometries.

Let $e=(i,j)$ be an edge of $\GG$, with $o(e)=i$ and $o(\bar{e})=j$. Recall that, by definition, we have that $\alpha_{(i,j)}=\alpha_{(j,i)}$.
In this example, we will have that the edge groups $D=\langle d_{7}\rangle$, $E=\langle e_{8}, e_{9}\rangle$ and $F=\langle f_{10}\rangle$ are free groups and the edge coset incidence systems are
$\alpha_{(A,B)}=\alpha_{(B,A)}= (D,(\langle \{1_D\}\rangle))$, $\alpha_{(A,C)}=\alpha_{(C,A)}=(E,(\langle e_{9}\rangle,\langle e_{8}\rangle))$ and $\alpha_{(B,C)}=\alpha_{(C,B)}=(F ,(\langle \{1_F\}\rangle))$.
Finally, the monomorphisms between edges and vertices are as follows:
\begin{itemize}
    \item $\varphi_{(A,B)}:D\to A$, where $\varphi_{(A,B)}(d_{7})=a_1$; and $\varphi_{(B,A)}:D\to B$ where $\varphi_{(B,A)}(d_{7})=b_4$;
    \item $\varphi_{(A,C)}:E\to A$, where $\varphi_{(A,C)}(e_8)=a_1$ and $\varphi_{(A,C)}(e_9)=a_2$; and $\varphi_{(C,A)}:E\to C$ where $\varphi_{(C,A)}(e_8)=c_{6}$ and $\varphi_{(C,A)}(e_9)=c_{5}$;
    \item $\varphi_{(B,C)}:F\to B$, where $\varphi_{(B,C)}(f_{10})=b_3$; and $\varphi_{(C,B)}:F\to C$ where $\varphi_{(C,B)}(f_{10})=c_{6}$.
\end{itemize}
Notice that, in each of the cases above, $\varphi_{(i,j)}\circ\varphi_{(j,i)}^{-1}$ is an admissible isomorphism between a parabolic subgroup of $\alpha_i$ and a parabolic subgroup of $\alpha_j$.
For example, consider the edge $(B,C)$ and set $\varphi_{(B,C)}\circ\varphi_{(C,B)}^{-1}=:\gamma$. We have that
$\gamma(c_{6})=b_3$. 
Notice that this means that $B_{\{4\}}\cong C_{\{5\}}$.
Moreover, for any $J\subseteq I_C\setminus\{5\} = \{6\}$, there exists a $K\subseteq I_B\setminus \{4\}=\{3\}$ such that 
$\gamma(C_{\{5\}\cup J})= B_{\{4\}\cup K}$.
Therefore, $\gamma$ induces an isomorphism between the lattice of parabolic subgroups of $C_{\{5\}}$ and the lattice of parabolic subgroups of $B_{\{4\}}$.

We now would like to compute the fundamental coset geometry of this graph of coset geometries using the three different possible spanning trees and show that they are isomorphic. Consider the three distinct spanning trees of $\GG$. These are $T_1$, consisting of the edges $(A,B)$ and $(C,B)$, $T_2$, consisting of the edges $(A,B)$ and $(C,A)$, and $T_3$, consisting of the edges $(C,B)$ and $(C,A)$.
Let $S_A=\{a_1,a_2\}$, $S_B=\{b_3,b_4\}$ and $S_C=\{c_5,c_6\}$.

\textbf{Using the spanning tree $T_1$:} We begin by amalgamating along the coset geometries on $T_1$.
The admissible isomorphisms are $\gamma_{(A,B)}=\varphi_{(B,A)}\circ\varphi_{(A,B)}^{-1}$ and $\gamma_{(C,B)}=\varphi_{(B,C)}\circ\varphi_{(C,B)}^{-1}$. Hence, have that 
$\gamma_{(A,B)}(a_1)=b_4$ and $\gamma_{(C,B)}(c_{6})=b_3$.
Notice that, these isomorphisms create equivalence classes between types. We thus have that $i\sim i$, for $i\in I_A\cup I_B\cup I_C$, and $1\sim 4$ and $3\sim 6$.
Hence, the type set $I_{T_1}$ of the geometry obtained by amalgamating along $T_1$ is $I_{T_1}=\{\{1,4\},\{2\},\{3,6\},\{5\}\}$.
We obtain the coset incidence geometry $\Gamma_{T_1}=(G,(G_i)_{i\in I_{T_1}})$, with $G=\langle S_A\cup S_B\cup S_C \mid  a_1 = b_4, b_3=c_{6}\rangle$ and, for $i\in I_{T_1}$, $G_i=\langle A_{i\cap I_A}, B_{i\cap I_B}, C_{i\cap I_C}\rangle$.
It remains to perform the HNN-extension of $\Gamma_{T_1}$ on the remaining edge of $\mathcal{G}$, using the isomorphism $\gamma_{(C,A)}=\varphi_{(A,C)}\circ\varphi_{(C,A)}^{-1}$. Notice that $\gamma_{(C,A)}$ establishes an isomorphism from $C$ to $A$. Let $\widehat{I_{T_1}}=I_A\cup I_B\cup I_C$ and notice that $A_{I_A}=\langle 1_A\rangle$, $B_{I_B}=\langle 1_B\rangle$ and $C_{I_C}=\langle 1_C\rangle$.
Then $G_{\widehat{I_{T_1}}\setminus I_C}=\langle A_{I_A},B_{I_B}, C\rangle = C$ and, equivalently, $G_{\widehat{I_{T_1}}\setminus I_A} = A$. 
This means that $\gamma_{(C,A)}$
corresponds to an isomorphism from $G_{\widehat{I_{T_1}}\setminus I_C}$ to $G_{\widehat{I_{T_1}}\setminus I_A}$. We can express $G_{\widehat{I_{T_1}}\setminus I_C}$ equivalently as $G_{\widehat{I_{T_1}}\setminus\{3,5\}}$, as $c_{6}=b_3$, and ultimately, $G_{\widehat{I_{T_1}}\setminus I_C}\cong G_{\widehat{I_{T_1}}\setminus\{5,3\}}\cong G_{I_{T_1}\setminus\{\{5\},\{3,6\}\}}$. Similarly, $G_{\widehat{I_{T_1}}\setminus I_A}\cong G_{I_{T_1}\setminus\{\{1,4\},\{2\}\}}$.
Let $\widetilde{I_{T_1}}$ be the set of equivalence classes induced by $\gamma_{(C,A)}$, so that $\widetilde{I_{T_1}}=\{ \{\{1,4\},\{3,6\}\}, \{\{2\},\{5\} \}\}$. 
The HNN-extension of $\Gamma_{T_1}$ is $\Gamma_{T_1}*_{\gamma_{(C,A)}}=(\overline{G},(\overline{G}_i)_{i\in \widetilde{I_{T_1}}\cup\{t\}})$, where $$\overline{G}=\langle S_A, S_B, S_C, t_G\mid  a_1 = b_4, b_3=c_{6}, t_G^{-1}c_{5}t_G=a_2, t_G^{-1}c_{6}t_G=a_1\rangle,$$ with $\overline{G}_i=\langle G_i, t_G\rangle$ for $i\in \widetilde{I_{T_1}}$ and $$\overline{G}_t=G_{(I_{T_1}\setminus\{\{5\},\{3,6\}\})\setminus(I_{T_1}\setminus\{\{1,4\},\{2\}\})}=G_{\{\{1,4\},\{2\}\}\setminus\{\{5\},\{3,6\}\}}=G_{\{\{1,4\},\{2\}\}}.$$
Now, as each $i\in \widetilde{I_{T_1}}$ is a set, we have that $\overline{G}_i = \langle G_i,t_G\rangle = \langle \cap_{j\in i} G_j,t_G\rangle =\langle G_{\widehat{i}}, t_G \rangle =  \overline{G}_{\widehat{i}}$. Hence, we can simplify $\widetilde{I_{T_1}}$ as the set $\{\{1,4,3,6\},\{2,5\}\}$.

\textbf{Using the spanning tree $T_2$:} Proceeding in a similar fashion as for $T_1$, after amalgamating along $T_2$, we obtain the coset geometry $\Gamma_{T_2}=(H,(H_i)_{i\in I_{T_2}})$, with $H=\langle S_A, S_B, S_C\mid b_4 = a_1 = c_6, a_2=c_{5}\rangle$, $H_i=\langle A_{i\cap I_A}, B_{i\cap I_B}, C_{i\cap I_C}\rangle$ and $I_{T_2}=\{\{1,4,6\},\{2,5\},\{3\}\}$
We then perform the HNN-extension on the last edge of $\GG$, using the isomorphism $\gamma_{(C,B)}=\varphi_{(B,C)}\circ\varphi_{(C,B)}^{-1}$, which maps the parabolic subgroup $H_{I_{T_2}\setminus\{\{1,4,6\}\}}$ to $H_{I_{T_2}\setminus\{\{3\}\}}$.
The resulting geometry is $\Gamma_{T_2}*_{\gamma_{(C,B)}}=(\overline{H},(\overline{H}_i)_{i\in \widetilde{I_{T_2}}\cup \{t\}})$, with $\widetilde{I_{T_2}}= \{\{1,4,3,6\},\{2,5\}\}$, with $$\overline{H}=\langle S_A, S_B, S_C,t_H \mid b_4 = a_1 = c_6, a_2=c_{5},  t_H^{-1}c_{6}t_H=b_3\rangle,$$ $\overline{H}_i=\langle H_i,t\rangle$ for $i\in \widetilde{I_{T_2}}$, and $\overline{H}_t=H_{\{\{3\}\}}$.
The groups $\overline{G}$ and $\overline{H}$ are naturally isomorphic, as they are both the fundamental group of the same graph of groups. Additionally, as $\widetilde{I_{T_2}}=\widetilde{I_{T_1}}$, for each $J\subseteq \widetilde{I_{T_2}}$, we have $\overline{H}_J\cong\overline{G}_J$. 
Indeed, taking $J=\{\{1,3,4,6\}\}$, we have that $\overline{H}_J=\langle  a_2, c_5, t_H\rangle$ and $\overline{G}_J=\langle a_2, c_5, t_G\rangle$, which, as $a_2=c_5$ in both cases, both subgroups are isomorphic to $\ZZ*\ZZ$. Taking $J=\{\{2,5\}\}$ instead gives $\overline{H}_J = \langle a_1,b_3,b_4,c_6,t_H\rangle$ and $\overline{H}_J = \langle a_1,b_3,b_4,c_6,t_G\rangle$, which, by the relations of $\overline{H}$ and $\overline{G}$, can be rewritten as $\overline{H}_J = \langle c_6,t_H\rangle$  and $\overline{G}_J = \langle c_6,t_G\rangle$, where both are isomorphic also to $\ZZ*\ZZ$. As $\overline{H}_{\widetilde{I_{T_2}}}=\langle t_H\rangle \cong \langle t_G\rangle = \overline{G}_{\widetilde{I_{T_1}}}$, we have that $\overline{H}_J \cong \overline{G}_J$, for any $J\subseteq \widetilde{I_{T_2}}$. Finally, we have that $\overline{H}_t = H_3 = \langle a_1, a_2, b_4, c_5, c_6\rangle\cong \ZZ*\ZZ$ and $\overline{G}_t = G_{\{\{1,4\},\{2\}\}} = \langle b_3, c_5, c_6\rangle\cong \ZZ*\ZZ$, i.e. $\overline{H}_t\cong \overline{G}_t$. Moreover, for $J\subseteq \widetilde{I_{T_2}}$, we can equally prove that $\overline{H}_{J\cup \{t\}}\cong\overline{G}_{J\cup \{t\}}$.

Hence, the lattice of parabolic subgroups of $\Gamma_{T_1}*_{\gamma_{(C,A)}}$ is isomorphic to $\Gamma_{T_2}*_{\gamma_{(C,B)}}$, and therefore, $\Gamma_{T_1}*_{\gamma_{(C,A)}}\cong \Gamma_{T_2}*_{\gamma_{(C,B)}}$.

\textbf{Using the spanning tree $T_3$:} Lastly, by considering the spanning tree $T_3$, we get the coset geometry $\Gamma_{T_3}=(K,(K_i)_{i\in I_{T_3}})$, with $K=\langle S_A,S_B,S_C\mid a_1 = c_6=b_3, a_2=c_5 \rangle$, $K_i = \langle A_{i\cap I_A}, B_{i\cap I_B}, C_{i\cap I_C}\rangle$ and $I_{T_3}=\{\{1,3,6\},\{2,5\},\{4\}\}$.
We then perform the HNN-extension on the last edge of $\GG$, using the isomorphism $\gamma_{(A,B)}=\varphi_{(B,A)}\circ\varphi_{(A,B)}^{-1}$, which maps the parabolic subgroup $K_{I_{T_3}\setminus\{\{1,3,6\}\}}$ to $K_{I_{T_3}\setminus\{\{4\}\}}$.
The resulting geometry is $\Gamma_{T_3}*_{\gamma_{(A,B)}}=(\overline{K},(\overline{K}_i)_{i\in \widetilde{I_{T_3}}\cup \{t\}})$, with $\widetilde{I_{T_3}}= \{\{1,4,3,6\},\{2,5\}\}$, with $$\overline{K}=\langle S_A, S_B, S_C,t_K \mid b_3 = a_1 = c_6, a_2=c_{5},  t_K^{-1}a_{1}t_K=b_4\rangle,$$ $\overline{K}_i=\langle K_i,t\rangle$ for $i\in \widetilde{I_{T_3}}$, and $\overline{K}_t=K_{\{\{4\}\}}$.
Using similar arguments as before, we can prove that $\Gamma_{T_3}*_{\gamma_{(A,B)}}$ is isomorphic to the coset geometries using the other two trees.

\end{example}

\section{Discussions and Future Work}\label{sec:future}

We conclude this article by discussing possible future directions of research. We have defined several operations on coset incidence systems inspired by classical group constructions and have shown that properties such as flag-transitivity and residual connectedness are naturally preserved. In this sense, our constructions can be considered satisfactory. However, an open question remains as to whether these are the “correct” constructions corresponding to each classical group operation. For the free product with amalgamation, the semi-direct product, and the twisting operation, there are multiple reasons to believe that our constructions are indeed the most natural. 
Firstly, these operations have already been studied in the context of regular polytopes~\cite[Chapters 4 and 8]{ARP}, albeit with many restrictions. For example, the amalgamation of the groups of two regular polytopes always results in an automorphism group that does not have a linear Coxeter diagram, a key feature of the automorphism of regular polytopes. Hence, in order to recover a linear diagram, the resulting group from the amalgamation needs to be quotiented by extra relations. But this process on taking quotients is delicate and often breaks flag-transitivity or residual connectedness. This is usually described as the ``amalgamation problem" of regular polytopes. In our construction, we do not require any linearity condition, hence we are not burdened by these questions. The same can be said about the twisting operation, which is very restrictively defined in~\cite{ARP} to guarantee that the resulting structure has a linear diagram.
The approach we present here extends these constructions in a natural way by loosening the constraints. 

The situation becomes less clear for HNN-extensions. In this case, we are not aware of any prior work relating HNN-extension to incidence geometries. In fact, neither HNN-extensions nor free products (with amalgamation) typically produce groups that can be though of as automorphism groups of polytopes or other structures with linear diagrams.
In the case of HNN-extension, the choice of maximal parabolic subgroups when constructing a coset incidence system for an HNN-extension was not as immediate as in the case of free amalgamated products. In particular, the treatment of the stable letter $t$ of the HNN-extension was a point of ongoing discussion. One could, for instance, define the HNN-extension so that the maximal parabolic subgroup $G_t$ is the whole original group $A$, or alternatively, exclude the type $t$ entirely, with the type set being solely $\widetilde{I}$. Ultimately, our choice for the maximal parabolic subgroups in the HNN-extension was guided by the requirement that certain important properties be preserved. For example, whenever $\varphi$ is an automorphism of the whole group $A$, the HNN-extension $A*_\varphi$ can also be realized as a semi-direct product. It is therefore natural to expect that the coset incidence system constructed by considering $A*_\varphi$ as an HNN-extension or as a semi-direct product should be comparable. This consideration led us to believe that the stable letter $t$ should always be included as a type, consistent with our semi-direct product constructions, as discussed in Example~\ref{example:SemiDirectProduct}.

Nevertheless, defining $G_t = A$ in general leads to complications, particularly with the notion of a graph of coset geometries introduced in Section~\ref{subsec:graphofcosetgeom}. With this definition, different choices of spanning trees for the graph can result in different maximal parabolic subgroups of type $t$ for the fundamental coset geometry, even in simple cases. For instance, in Example~\ref{example:graphcosetgeom}, choosing the spanning tree $T_1$ yields $G \cong F_4$, the free group of rank $4$, whereas choosing $T_2$ gives $H \cong F_3$. If we were to set the maximal parabolic of type $t$ to be equal to the starting groups, we would get that $\overline{G}_t\ncong \overline{H}_t$, and thus we would never obtain a well defined notion of fundamental geometry of a graph of coset incidence systems. 

At a fundamental level, the question of whether the coset geometries we construct are the appropriate analogues of the algebraic operations considered depends on what is regarded as the essence of these operations. One possible approach to formalizing this would be to translate the universal properties defining amalgamations, HNN-extensions, and semi-direct products into the category of coset incidence geometries, and to verify whether our constructions satisfy these versions of 
the universal properties. We believe that such an approach could be fruitful, not only in the context of this paper, but as a general guiding principle for future research topics.

In this article, we have only touched upon Bass-Serre theory with our proposed definition of a graph of coset geometries. Further study in this direction would likely be both interesting in its own right and illuminating for the ``naturality" questions discussed above.
 
Finally, the twisting operation we have introduced has already demonstrated, in Section~\ref{sec:applications}, a wide range of applications, particularly with respect to Shephard groups. We have illustrated different uses of this operation, whether by exploiting symmetries of the defining graph of a Shephard group or by constructing a suitably labeled graph $\GG$ using the $j$-elements of a coset incidence system $\beta$. This highlights the versatility of the twisting construction, including its usefulness in studying intersection properties of parabolic subgroups in certain families of infinite Shephard groups. However, there are some limitations inherent to the current definition of this operation. In Definition~\ref{def:GammaPhiAdmiss} of $(\beta,\varphi)$-admissibility, we allow only one orbit on the type set of $\alpha$ to have size greater than one. This restriction excludes split extensions in which the automorphism group induces two or more distinct non-singleton orbits on the type set. The main obstacle to generalize our construction for multiple complex orbits lies in condition $(c)$, which imposes conditions on the intersection of the orbits. We believe that, to extend the definition of $(\beta,\varphi)$-admissibility, it would be necessary to verify the intersection property (Equation~\ref{eq:IPO}) for each non-singleton orbit. We plan to investigate this direction in future work.

\bibliographystyle{ieeetr} 
\bibliography{refs}
\end{document}